\documentclass[12pt]{article}
\usepackage{geometry}
\usepackage[utf8]{inputenc}
\usepackage[T1]{fontenc}
\usepackage{amsmath}
\usepackage{amsthm}
\usepackage{amssymb}
\usepackage{array}
\usepackage{stmaryrd}
\usepackage{graphicx}
\usepackage[backend=biber,maxalphanames=4,maxnames=4,style=alphabetic,sorting=nty,giveninits]{biblatex}
\usepackage[linktoc = page]{hyperref}
\usepackage{url}
\usepackage{bbm}
\usepackage{bm}
\usepackage{comment}
\usepackage[table,xcdraw,dvipsnames]{xcolor}

\newcounter{continue-counter}

\newcommand{\C}{\mathbb{C}}
\newcommand{\N}{\mathbb{N}}
\newcommand{\Z}{\mathbb{Z}}
\newcommand{\Q}{\mathbb{Q}}
\newcommand{\R}{\mathbb{R}}

\newcommand{\kk}{\mathbbm{k}}

\newcommand{\SymGroup}[1][n]{\mathfrak{S}_{#1}}
\newcommand{\Specht}[1]{\chi_{#1}}
\newcommand{\TransposePartition}[1]{{#1}^{\ast}}
\newcommand{\Schur}[1]{\mathbb{S}_{#1}}
\newcommand{\into}{\hookrightarrow}
\newcommand{\onto}{\twoheadrightarrow}
\newcommand{\MultTriv}[1]{M_{#1}}
\newcommand{\ConfLabeled}[2]{{#1}^{\bullet}\otimes_{\Fin_{\ast}}{#2}_{\bullet}}
\newcommand{\Loday}[2]{\mathcal{L}({#1},{#2})}
\newcommand{\n}[1][n]{\mathbf{#1}}
\newcommand{\FinPointedSet}[1]{\n[#1]_{+}}
\newcommand{\EndSpace}[1]{\operatorname{Maps}({#1},{#1})}

\newcommand{\Vect}[1]{\textrm{Vect}_{#1}}

\newcommand{\ModuliCurve}[1]{\mathcal{M}_{#1}}
\newcommand{\OperSuspension}{\Sigma}

\newcommand{\sgn}[1][n]{\operatorname{sgn}_{#1}}
\newcommand{\Ind}[2]{\operatorname{Ind}_{#1}^{#2}}

\DeclareMathOperator{\Aut}{Aut}

\DeclareMathOperator{\End}{End}
\DeclareMathOperator{\FB}{FB}
\DeclareMathOperator{\Fin}{Fin}
\DeclareMathOperator{\GL}{GL}
\DeclareMathOperator{\gr}{gr}
\DeclareMathOperator{\GradedVect}{grVect}
\DeclareMathOperator{\Hom}{Hom}
\DeclareMathOperator{\Lie}{Lie}
\DeclareMathOperator{\Out}{Out}
\DeclareMathOperator{\Res}{Res}
\DeclareMathOperator{\Suspension}{S}
\DeclareMathOperator{\Set}{Set}
\DeclareMathOperator{\Sym}{Sym}

\theoremstyle{plain}
\newtheorem{thm}{Theorem}[section]
\newtheorem{lem}[thm]{Lemma}
\newtheorem{prop}[thm]{Proposition}
\newtheorem{conj}[thm]{Conjecture}
\newtheorem{cor}[thm]{Corollary}
\newtheorem{claim}[thm]{Claim}

\theoremstyle{definition}
\newtheorem{defi}[thm]{Definition}
\newtheorem{expl}[thm]{Example}
\newtheorem{observation}[thm]{Observation}

\theoremstyle{remark}
\newtheorem{rmk}[thm]{Remark}
\newtheorem{note}[thm]{Note}
\newtheorem{goal}[thm]{Goal}
\newtheorem{terminology}[thm]{Terminology}

\makeatletter
\newcommand{\subjclass}[2][1991]{%
  \let\@oldtitle\@title%
  \gdef\@title{\@oldtitle\footnotetext{#1 \emph{Mathematics subject classification.} #2}}%
}
\newcommand{\keywords}[1]{%
  \let\@@oldtitle\@title%
  \gdef\@title{\@@oldtitle\footnotetext{\emph{Key words and phrases.} #1.}}%
}
\makeatother

\title{Configuration spaces on a wedge of spheres and Hochschild--Pirashvili homology}
\author{Nir Gadish\footnote{Department of Mathematics, University of Michigan, 530 Church St, Ann Arbor, MI 48109, USA } \and Louis Hainaut\footnote{Department of Mathematics, Stockholm University, 106 91 Stockholm, Sweden}}

\date{\vspace{-5ex}}

\begin{document}

\maketitle

\abstract{
	We study the compactly supported rational cohomology of configuration spaces of points on wedges of spheres, equipped with natural actions of the symmetric group and the group $\Out(F_g)$ of outer automorphisms of the free group. These representations show up in seemingly unrelated parts of mathematics, from cohomology of moduli spaces of curves to polynomial functors on free groups and Hochschild--Pirashvili cohomology.
	
	We show that these cohomology representations form a polynomial functor, and use various geometric models to compute many of its composition factors. We further compute the composition factors completely for all configurations of $n\leq 10$ particles. An application of this analysis is a new super-exponential lower bound on the symmetric group action on the weight $0$ component of $ H^*_c(\ModuliCurve{2,n})$.
}

\tableofcontents

\section{Introduction}
This paper explores a circle of ideas centered around compactified configuration spaces of graphs and Hochschild cohomology, with applications to moduli spaces of marked curves and representation theory of outer automorphism groups of free groups. Revisiting ideas of Quoc H\^{o} \cite{Ho17}, we explain the relation between Hochschild cohomology with square-zero coefficients and configuration spaces -- two classes of mathematical objects that have independently been studied by many groups of researchers (see e.g. \cite{TW19,PV18} for the former, and \cite{Pet20,BCGY21} for the latter). This work seeks to popularize the connection, establish a dictionary, and bring the two points of view together to say new things about both objects as well as perform computations. Readers interested in either Hochschild homology, configuration spaces, polynomial representations of $\Out(F_n)$, moduli spaces of curves, or hairy graph complexes should be aware of how these subjects interact and share insights. The original motivation for this work is the following.

\subsection{Super-exponential unstable cohomology of \texorpdfstring{$\ModuliCurve{2,n}$}{M2,n}}
The moduli space $\ModuliCurve{2,n}$ of genus $2$ curves with $n$ marked points has complicated and mostly unknown rational cohomology. Considering the action of the symmetric group $\SymGroup$ by permuting marked points, the cohomology becomes an $\SymGroup$-representation.  In cohomological degree small compared to $n$ it exhibits representation stability as shown by Jiménez Rolland \cite{JR11}, consequently bounding Betti numbers to be eventually polynomial in $n$. In contrast, the orbifold Euler characteristic of $\ModuliCurve{2,n}$ was computed by Harer--Zagier to be $\frac{(-1)^{n+1}}{240}(n+1)!$
\cite[\S6]{HZ86}, thus necessitating super-exponential growth in $n$ as well as lots of odd-dimensional cohomology. This was awkward as there were almost no known infinite families of unstable cohomology classes, certainly not ones that grow super-exponentially in $n$ and reside in odd degrees (see the introduction of \cite{BBP} and Remark 1.3 in particular). Building on an unpublished formula of Petersen and Tommasi \cite{petersen-tommasi} (see Proposition \ref{prop:cohom-moduli-space}) we construct such families in the two highest non-trivial cohomological dimensions.

Our result is most naturally phrased Poincar\'{e} dually, on compact support cohomology. Below, $F(X,n)$ is the ordered configuration space of $n$ distinct points in $X$.

\begin{thm} \label{thm:subrep of M2n}
For every $n\geq 3$, the weight $0$ rational cohomology $\gr^W_0 H^{n+*}_c(\ModuliCurve{2,n};\Q)$ contains the following $\SymGroup$-subrepresentations

\begin{equation*}
\begin{cases}
\underset{k \text{ odd }}{\bigoplus} { \sgn\otimes H_{2(n-k)}(F(S^3,n)/SU(2))}^{\oplus \left\lfloor\frac{k}{6}\right\rfloor} \oplus H_{n-4}(\ModuliCurve{0,n}) & \text{if }  *=2 \\
\underset{k \text{ even}}{\bigoplus}  { \sgn\otimes H_{2(n-k)}(F(S^3,n)/SU(2))} ^{\oplus \left\lfloor\frac{k}{6}\right\rfloor} & \text{if }  *=3 
\end{cases}
\end{equation*}
where $\lfloor x\rfloor$ is the integer part of $x$. These degrees are Poincar\'{e} dual to the top two nontrivial degrees in which ordinary rational cohomology appears.


\end{thm}

The characters of $H_*(\ModuliCurve{0,n})$ and $H_{*}(F(S^3,n)/SU(2))$ have been computed by Getzler and Pagaria, respectively, and are recalled in \S\ref{sec:characters of M0n and FR3n}. In particular, one can check that these subrepresentations have total dimension on the order of $(n-2)!$ thus grow super-exponentially.

\begin{rmk}
Note the rather peculiar appearance 
of configurations on odd dimensional manifolds. This brings to mind Hyde's discovery \cite{Hyde20} that the $\SymGroup$-character of $H^*(F(\R^3,n))$ appears in counts of polynomials over finite fields. Hyde's formula was recently given a geometric explanation by Petersen--Tosteson \cite{PT21}. An analogous geometric explanation of Corollary \ref{thm:subrep of M2n} would be most pleasing, but no such is known.
\end{rmk}


As explained in Proposition \ref{prop:cohom-moduli-space} below, there is a transferal of information from Hochschild homology to moduli spaces of curves. With this, calculations of Powell and Vespa of the former allow us to prove a conjecture from \cite{BCGY21}.

\begin{cor} \label{cor:conjecture}
For all $n\in \N$, the standard representation $\Specht{(n-1,1)}=\operatorname{std_n}$  occurs in $\gr^W_0 H^{n+2}_c(\ModuliCurve{2,n})$ with multiplicity $\left\lfloor \frac{n}{12} \right\rfloor$, while the irreducible representation $\Specht{(2,1^{n-2})}=\operatorname{std_n}\otimes \sgn$ never occurs.
\end{cor}

Our analysis suggests a multitude of new conjectures about the multiplicity of certain $\SymGroup$-isotypical components inside $\gr_0^W H_c^*(\ModuliCurve{2,n})$, detailed in Conjecture \ref{conj:patterns-M2n}

\subsection{Configurations on wedges of spheres}

We access the unstable cohomology of $\ModuliCurve{2,n}$ via configuration spaces of points on graphs. These are complicated topological spaces, which appear e.g. when studying moduli of tropical curves \cite{BCGY21}, geometric group theory and applied topology \cite{GK02}.
We take particular interest in their compactly supported cohomology, which turns out to depend only on the first Betti number of the graph (the loop order $g$), and carries an action of the group $\Out(F_g)$ of outer automorphisms of the free group. These cohomology representations are of interest since, as we explain below, they are closely related to Hochschild--Pirashvili cohomology and play a universal role in the theory of `Outer' polynomials functors on free groups. Notably, the cohomology provides a large natural class of $\Out(F_g)$-representations that do not factor through $\GL_g(\mathbb{Z})$. Detailed in \S\ref{sec:intro-M2n}, the special case of these representations with $g=2$ is the key object that gives us a handle on the cohomology of $\ModuliCurve{2,n}$ and allows for the results mentioned above.


Recall the (ordered) configuration space of $n$ points on a topological space $X$
\[
    F(X,n) = \{(x_1, \ldots, x_n) \in X^n \mid x_i\neq x_j \; \forall i\neq j\} \subseteq X^n.
\]
An often overlooked observation is that, while the homotopy type of the configuration space of $n\geq 2$ points is famously not a homotopy invariant of $X$, the compactly supported cohomology $H^*_c(F(X,n);\Z)$ is an invariant of the proper homotopy type of $X$ under mild point-set hypotheses. We explain this in detail in \S\ref{sec:configurations with compact support}. Let us consider cohomology with $\Q$-coefficients.

Since every finite graph is equivalent to a wedge of circles $X = \bigvee_{i=1}^g S^1$, it suffices to study such wedges. The group of homotopy classes of auto-equivalences of such $X$ is naturally isomorphic to $\Out(F_g)$, and thus $H^*_c(F(\bigvee_{i=1}^g{S^1}, n); \Q)$ acquires a natural $\Out(F_g)$-action.

Varied motivations, including the theory of string-links and spaces of embeddings $\R^m\into \R^n$ \cite{TW17,TW19}, polynomial functors on free groups \cite{HPV15,PV18}, and unstable cohomology of moduli spaces, lead to the following hard open problem.
\begin{goal}\label{goal:understand wedges of circles}
Characterize the $\Out(F_g)$-representations $H^*_c(F(\bigvee_{i=1}^g{S^1}, n); \Q)$ for all finite wedges of circles.
\end{goal}

This problem is essentially out of reach with our methods due to the fact that representations of $\Out(F_g)$ are in general not semi-simple. However, after semi-simplification they become a direct sum of factors arising from $\GL_g(\Z)$-representations. The present work focuses therefore on the simpler question of determining only these composition factors, ignoring the extension problem of these factors.

We situate this problem as a special case of configurations on wedges of spheres of arbitrary dimension. Let $X = \bigvee_{i\in I} S^{d_i}$ be a finite wedge of spheres, possibly of different dimensions $d_i \geq 1$. 
Our main object of study is the rational cohomology with compact support $H^*_c(F(X,n);\Q)$ (or rather an associated graded thereof), and how it behaves under continuous maps, in particular as the number and dimensions of spheres vary.
This seemingly more general problem turns out to be no harder, as it is governed by a kind of \emph{polynomial functor} which is entirely determined by its values on wedges of circles, as explained next.

Recall that a \emph{symmetric sequence} is a sequence of vector spaces $(\Phi[m])_{m\in \mathbb{N}}$ such that $\Phi[m]$ carries an action by the symmetric group $\SymGroup[m]$\footnote{Such sequences are called \emph{linear species} in combinatorics, and \emph{$\FB$-modules} in the context of representation stability.}. Joyal's theory of analytic functors treats symmetric sequences as coefficients of a power series taking vector spaces to vector spaces:
\[
V \longmapsto \bigoplus_{m\in \mathbb{N}} \Phi[m] \otimes_{\SymGroup[m]} V^{\otimes m}.
\]
Such a functor is \emph{polynomial of degree at most $d$} if only the terms with $m\leq d$ are nonzero. Our setup involves functors in two variables, hence we consider coefficients $(\Phi[n,m])_{(n,m)\in \mathbb{N}^2}$ which are graded $\SymGroup\times\SymGroup[m]$-representations.

\begin{thm}[Polynomiality]\label{thm:polynomiality}
    Consider the full subcategory of compact topological spaces $X$ that are homotopy equivalent to wedges of spheres, i.e. $X \simeq \bigvee_{i\in I} S^{d_i}$. For every $n\in \N$ there exists a natural `collision' filtration on the functor $X\mapsto H_c^{\ast}(F(X, n))$, whose associated graded factors through $X\mapsto\tilde{H}^*(X)$ followed by a polynomial functor of degree $n$.
    
    Explicitly, there exists a collection of graded $\SymGroup\times\SymGroup[m]$-representations $\Phi[n,m]$, indexed by $(n,m)\in \N^2$ and not depending on $X$, and a natural $\SymGroup$-equivariant isomorphism of graded vector spaces
    \begin{equation} \label{eq:analytic functor}
    \gr H_c^{*}(F(X, n)) \cong \bigoplus_{m\in\N} {\Phi[n, m]\otimes_{\SymGroup[m]}(\tilde{H}^{\ast}(X))^{\otimes m}}.
    \end{equation}
    Here $\tilde{H}^{\ast}(X) \cong \bigoplus_{i\in I}\Q[-d_i]$ is considered as a graded vector space, and $\otimes$ is the graded tensor product whose symmetry uses the Koszul sign rule.
\end{thm}

We will prove that $\Phi[n,m]=0$ whenever $m>n$ and that $\Phi[n,n]\neq 0$, which means that the right hand side of \eqref{eq:analytic functor} is the evaluation at $\tilde{H}^*(X)$ of a polynomial functor of degree $n$.

\begin{rmk}
    The associated graded is necessary here. For example when $X = \bigvee_{i=1}^g{S^1}$ is a wedge of circles, both sides of \eqref{eq:analytic functor} are representations of $\Out(F_g)$, but the right hand side always factors through a representation of $\GL_g(\Z)$, while on the left hand side this is in general not the case without taking the associated graded -- see Remark \ref{rmk:extensions}.
\end{rmk}

With this polynomiality theorem, understanding the `coefficient' representations $\Phi[n,m]$ is an important step towards the determination of Hochschild--Pirashvili cohomology for all wedges of spheres. Instead of Goal \ref{goal:understand wedges of circles} we thus focus on the following. Below, we let $\Specht{\lambda}$ denote the irreducible representation of $\SymGroup$ associated with partition $\lambda\vdash n$.

\begin{goal} \label{goal:decompose coefficients}
Compute $\Phi[n,m]$ as a graded $\SymGroup\times\SymGroup[m]$-representation. That is, compute the graded multiplicity of the $\SymGroup\times\SymGroup[m]$-irreducible representation $\Specht{\lambda} \boxtimes \Specht{\mu}$ occurring in $\Phi[n,m]$ for every pair of partitions $\lambda \vdash n$ and $\mu \vdash m$.
\end{goal}

This computation is in general a hard problem, though our geometric perspective gives access to important special cases presented in Theorem \ref{thm:symmetric and alternating powers}, and it further reveals structure that has not been explored until recently -- e.g. the Lie structure discussed \S\ref{sec:lie structure}. We remark that Powell and Vespa study the same representations in \cite{PV18} from a purely algebraic perspective; they also address there the problem of determining the extensions between the composition factors.


For $X$ a wedge of circles, $H_c^*(F(X, n))$ turns out to be concentrated in degrees $* = n-1$ and $n$ only, see \S\ref{section:2-step}. Two classes of representations are particularly accessible.
\begin{thm}[\textbf{Symmetric and exterior powers}]\label{thm:symmetric and alternating powers}
    Let $X = \bigvee^g_{i=1} S^1$ be a wedge of circles and let $V = \Q^g \cong \tilde{H}^{\ast}(X)$ with its standard action by $\GL_g(\Z)$. The $\SymGroup$-equivariant isotypic component of the exterior power $\Lambda^m(V)$ in the associated graded $\gr H_c^*(F(X, n))$ is given up to isomorphism by
\begin{align}
    \gr H^n_c(F(X,n)) &\geq \bigoplus_{m\geq 0} H_{n-m}(\ModuliCurve{0,n}) \otimes \Lambda^m(\tilde{H}^*(X)) \\
    \gr H^{n-1}_c(F(X,n)) &\geq \bigoplus_{m\geq 0} H_{n-m-2}(\ModuliCurve{0,n}) \otimes \Lambda^m(\tilde{H}^*(X))
\end{align}
    
    For the symmetric power $\Sym^m(V)$, its $\SymGroup$-equivariant isotypic component in the associated graded $\gr H_c^*(F(X, n))$ is given by sign twists of the Whitehouse modules
    \begin{align}
    \gr H^n_c(F(X,n)) &\geq \bigoplus_{m\geq 0} \sgn\otimes H_{2(n-m)}(F(\R^3,n-1)) \otimes \Sym^m(\tilde{H}^*(X)) \\
    \gr H^{n-1}_c(F(X,n)) &\geq \bigoplus_{m\geq 0} \sgn\otimes H_{2(n-m-1)}(F(\R^3,n-1)) \otimes \Sym^m(\tilde{H}^*(X)) 
\end{align}
    with $\SymGroup$ acting via the identification $F(\mathbb{R}^3,n-1) \cong F(S^3,n)/SU_2$.
\end{thm}
See \S\ref{sec:multiplicity symmetric powers} for proofs. The $\SymGroup$-characters of $H_*(\ModuliCurve{0,n})$ and $H_*(F(\R^3,n-1))$ are recalled in \S\ref{sec:characters of M0n and FR3n}.

Encoding the `coefficient' representations $\Phi[n,m]$ by entries of a matrix indexed by partitions $(\lambda\vdash n, \mu \vdash m)$ as in Table \ref{table:n=7 colors} below, the last theorem fully describes the columns associated with partitions $\mu = (m)$ and $(1^m)$ for all $m\geq 1$.

\begin{rmk}
    We note the unexpected appearance of the spaces $\ModuliCurve{0,n}$ and $F(\R^3,n-1)$ while considering wedges of circles. This is a consequence of polynomiality in Theorem \ref{thm:polynomiality}, and its uniform treatment of all spheres. In particular, for $S^2 \cong \mathbb{C}P^1$ and $S^3 \cong SU_2$.
\end{rmk}

\subsection{Explicit computations}\label{sec:explicit computation}
We approach calculation with two distinct tools, each giving access to one of the two symmetric group actions on $\Phi[n,m]$ while obscuring the other.
\begin{itemize}
    \item A Chevalley--Eilenberg complex for $H^*_c(F(X,n))$, described by Petersen \cite{Pet20}, and an associated `collision' spectral sequence. This gives insight into the right $\SymGroup[m]$-action on $\Phi[n,m]$, determining the polynomial contribution of $\tilde{H}^*(X)$.
    \item An $\SymGroup$-equivariant cell structure on the one-point compactification $F(X,n)^+$, when $X$ is a wedge of circles. This makes the $\SymGroup$-action by permuting point labels relatively computable.
\end{itemize}
Overlaying the two sets of resulting data, we are able to compute many new irreducible multiplicities of $\Phi[n,m]$. In particualr, we calculate the $\SymGroup\times\SymGroup[m]$-character of $\Phi[n,m]$ completely for all $n\leq 10$, and a certain part of $\Phi[11,m]$. Beyond this range, the representations $\Phi[n,m]$ are currently out of reach.

Our calculation technique is illustrated in Appendix \ref{app:full-computation}, culminating in Table \ref{table:final_answer10} which shows all multiplicities for $n=10$. Complete tabulations of $\Phi[n,m]$ for all $n\leq 10$ can be found by following \href{https://louishainaut.github.io/GH-ConfSpace/}{this URL}\footnote{\url{https://louishainaut.github.io/GH-ConfSpace/}}.
{All the computations for this paper were performed on Sage \cite{sagemath}. Our algorithms are freely available on \href{https://github.com/louishainaut/GH-ConfSpace}{GitHub}\footnote{\url{https://github.com/louishainaut/GH-ConfSpace}}.}

\begin{table}[hbt!]
\centering
  \resizebox{0.83\textwidth}{!}{\begin{minipage}{\textwidth}
    \hspace{-0.45cm}\begin{tabular}{|
>{\columncolor[HTML]{C0C0C0}}c ||
>{\columncolor[HTML]{FFFFC7}}c |c|
>{\columncolor[HTML]{FFFFC7}}c |c|c|c|
>{\columncolor[HTML]{FFFFC7}}c |c|c|c|
>{\columncolor[HTML]{FFFFC7}}c |c|c|c|c|c|
>{\columncolor[HTML]{FFFFC7}}c |
>{\columncolor[HTML]{FFCCC9}}c |
>{\columncolor[HTML]{FFCCC9}}c |
>{\columncolor[HTML]{FFCCC9}}c |
>{\columncolor[HTML]{FFCCC9}}c |
>{\columncolor[HTML]{FFCCC9}}c |
>{\columncolor[HTML]{FFCE93}}c |}
\hline
 & $6$ & \cellcolor[HTML]{C0C0C0}$5,1$ & $5$ & \cellcolor[HTML]{C0C0C0} & \cellcolor[HTML]{C0C0C0} & \cellcolor[HTML]{C0C0C0} & $4$ & \cellcolor[HTML]{C0C0C0} & \cellcolor[HTML]{C0C0C0}... & \cellcolor[HTML]{C0C0C0} & $3$ & \cellcolor[HTML]{C0C0C0} & \cellcolor[HTML]{C0C0C0}... & \cellcolor[HTML]{C0C0C0} & \cellcolor[HTML]{C0C0C0}$2,1^2$ & \cellcolor[HTML]{C0C0C0}$2,1$ & $2$ &  & $1^5$ & $1^4$ & $1^3$ & $1^2$ & $1$ \\ \hline \hline
\cellcolor[HTML]{CBCEFB}$7$ &  &  \cellcolor[HTML]{CBCEFB} &  & \cellcolor[HTML]{CBCEFB} & \cellcolor[HTML]{CBCEFB} & \cellcolor[HTML]{CBCEFB} &  & \cellcolor[HTML]{CBCEFB} & \cellcolor[HTML]{CBCEFB}... & \cellcolor[HTML]{CBCEFB} & $1$ & \cellcolor[HTML]{CBCEFB} & \cellcolor[HTML]{CBCEFB}... & \cellcolor[HTML]{CBCEFB} & \cellcolor[HTML]{CBCEFB}$1$ & \cellcolor[HTML]{CBCEFB} &  & \cellcolor[HTML]{CBCEFB} & \cellcolor[HTML]{CBCEFB}1 & \cellcolor[HTML]{CBCEFB} & \cellcolor[HTML]{CBCEFB} & \cellcolor[HTML]{CBCEFB} &  \\ \hline
$6,1$ &  &  &  &  &  &  &  &  & ... &  &  &  & ... &  &  &  &  &  &  &  &  &  &  \\ \hline
$5,2$ &  &  &  &  &  &  &  &  & ... &  & 1 &  & ... &  &  & 1 &  &  &  & $1$ &  &  & $1$ \\ \hline
$5,1^2$ &  &  &  &  &  &  & $1$ &  & ... &  &  &  & ... &  &  &  & $2$ &  &  &  & $1$ & $1$ &  \\ \hline
$4,3$ &  &  &  &  &  &  &  &  & ... &  & $1$ &  & ... &  &  &  & $1$ &  &  &  &  & $1$ &  \\ \hline
$4,2,1$ &  &  &  &  &  &  &  &  & ... &  & $2$ &  & ... &  &  & $1$ & $2$ &  &  &  & $1$ & $1$ & $1$ \\ \hline
$4,1^3$ &  &  &  &  &  &  &  &  & ... &  & $1$ &  & ... &  &  &  & $1$ &  &  &  &  & $1$ &  \\ \hline
$3^2,1$ &  &  &  &  &  &  & $1$ &  & ... &  &  &  & ... &  &  &  & $2$ &  &  &  & $1$ & $1$ &  \\ \hline
$3,2^2$ &  &  &  &  &  &  &  &  & ... &  & $2$ &  & ... &  &  & $1$ &  &  &  &  &  &  & $1$ \\ \hline
$3,2,1^2$ &  &  &  &  &  &  & $1$ &  & ... &  & $1$ &  & ... &  &  &  & $2$ &  &  &  &  & $1$ & $1$ \\ \hline
$3,1^4$ &  &  & $1$ &  &  &  &  &  & ... &  & $1$ &  & ... &  &  &  &  &  &  &  &  &  & $1$ \\ \hline
$2^3,1$ &  &  &  &  &  &  &  &  & ... &  & $1$ &  & ... &  &  &  & $1$ &  &  &  &  & $1$ &  \\ \hline
$2^2,1^3$ &  &  &  &  &  &  & $1$ &  & ... &  &  &  & ... &  &  &  & $1$ &  &  &  &  &  &  \\ \hline
\cellcolor[HTML]{CBCEFB}$2,1^5$ &  & \cellcolor[HTML]{CBCEFB} &  & \cellcolor[HTML]{CBCEFB} & \cellcolor[HTML]{CBCEFB} & \cellcolor[HTML]{CBCEFB} &  & \cellcolor[HTML]{CBCEFB} & \cellcolor[HTML]{CBCEFB}... & \cellcolor[HTML]{CBCEFB} &  & \cellcolor[HTML]{CBCEFB} & \cellcolor[HTML]{CBCEFB}... & \cellcolor[HTML]{CBCEFB} & \cellcolor[HTML]{CBCEFB} & \cellcolor[HTML]{CBCEFB} &  & \cellcolor[HTML]{CBCEFB} & \cellcolor[HTML]{CBCEFB} & \cellcolor[HTML]{CBCEFB} & \cellcolor[HTML]{CBCEFB} & \cellcolor[HTML]{CBCEFB} &  \\ \hline
\cellcolor[HTML]{CBCEFB}$1^7$ & $1$ & \cellcolor[HTML]{CBCEFB} &  & \cellcolor[HTML]{CBCEFB} & \cellcolor[HTML]{CBCEFB} & \cellcolor[HTML]{CBCEFB} &  & \cellcolor[HTML]{CBCEFB} & \cellcolor[HTML]{CBCEFB}... & \cellcolor[HTML]{CBCEFB} &  & \cellcolor[HTML]{CBCEFB} & \cellcolor[HTML]{CBCEFB}... & \cellcolor[HTML]{CBCEFB} & \cellcolor[HTML]{CBCEFB} & \cellcolor[HTML]{CBCEFB} &  & \cellcolor[HTML]{CBCEFB} & \cellcolor[HTML]{CBCEFB} & \cellcolor[HTML]{CBCEFB} & \cellcolor[HTML]{CBCEFB} & \cellcolor[HTML]{CBCEFB} &  \\ \hline
\end{tabular}
    \caption{$(\lambda,\mu)$ multiplicity in $\Phi[7,m]$ in codimension 1, arranged in lex-order of partitions. Unspecified entries are all zero. Yellow and red columns correspond to symmetric and alternating powers, respectively, and are completely described by Theorem \ref{thm:symmetric and alternating powers}. Lilac rows are computed in \cite[Examples 3-4]{PV18}. 
    }
    \label{table:n=7 colors}
    \end{minipage}}
\end{table}
For example, Table \ref{table:n=7 colors} gives the complete decomposition of the lower degree of $\Phi[7,m]$ into irreducibles for all $m$. Equivalently, this is the irreducible decomposition of $\gr H^{6}_c(F(\bigvee_g S^1 ,7))$ as a representation of $\SymGroup[7]\times \Out(F_g)$ up to splitting the collision filtration. Note that our Theorem \ref{thm:symmetric and alternating powers} accounts for nearly all the nontrivial contributions. The cohomology in the other degree $ H^{7}_c(F(\bigvee_g S^1 ,7))$ can be quickly obtained from Table \ref{table:n=7 colors} using the Euler characteristic.

The same Euler characteristic calculation gives a lower bound on the equivariant multiplicities of $H^*_c(F(X,n))$. It shows, in particular, that every Schur functor (see \eqref{eq:defi-Schur-functors}) does occur in the cohomology, in fact super-exponentially many times as $n$ grows. See \S\ref{sec:lower-bound} and the example therein.




\subsection{Weight 0 cohomology of moduli spaces}\label{sec:intro-M2n}

The coefficients $\Phi[n,-]$ discussed above determine the weight $0$ cohomology groups $\gr_0^{W} H_c^{*}(\ModuliCurve{2,n}, \Q)$, studied in recent work of Chan, Galatius and Payne \cite{CGP22,CFGP}, via the following proposition. For a partition $\lambda \vdash m$ let $\Phi[n,\lambda]$ denote the $\SymGroup$-equivariant multiplicity space of the $\SymGroup[m]$-irreducible $\Specht{\lambda}$ in $\Phi[n,m]$, equivalently it is $ \Phi[n,m]\otimes_{\SymGroup[m]}\Specht{\lambda}$. Recall that $\Phi[-,-]$ is graded, and denote by $\Phi[-,-]^i$ the part in degree $i$.

\begin{prop}\label{prop:cohom-moduli-space}
    For every partition with at most $2$ parts $\lambda = (a, b)$ there exists a coefficient $r_{\lambda} \in \N$ such that for each $i\geq 0$ and $n \in \N$ there is an $\SymGroup$-equivariant isomorphism
    \begin{equation} \label{eq:decomposition of moduli spaces}
        \gr_0^{W}H_c^{i}(\ModuliCurve{2,n}) \cong \bigoplus_{\lambda = (a, b)}\left(\Phi[n,\TransposePartition{\lambda}]^{i-3-|\lambda|}\right)^{\oplus r_\lambda}
    \end{equation}
    where $\gr^W_0$ is the weight $0$ subspace of $H^*_c$ in the sense of Hodge theory, and the sum runs over partitions $\lambda$ with up to two parts and conjugate partition $\TransposePartition{\lambda}$.
    
    Explicitly, the coefficients are given by
    \begin{equation}
        r_{(a,b)} = \begin{cases}
        \left\lfloor \frac{a-b}{6} \right\rfloor +1 & \text{if }a\equiv_2 b \equiv_2 1 \\
        \left\lfloor \frac{a-b}{6} \right\rfloor & \text{if }a\equiv_2 b \equiv_2 0 \\
        0 & \text{if }a\not\equiv_2 b 
        \end{cases}
    \end{equation}
    where $\lfloor x \rfloor$ is the integer part of $x\in \Q$ and $x\equiv_2 y$ is equivalence modulo $2$.
\end{prop}
Our Theorem \ref{thm:subrep of M2n} is a direct corollary of this Proposition along with Theorem \ref{thm:symmetric and alternating powers}

Note that the direct sum is finite since $\Phi[n,\lambda] = 0$ if $|\lambda| > n$. Also, due to the grading of $\Phi[n,\lambda]$, Formula \eqref{eq:decomposition of moduli spaces} produces cohomology only in degrees $n+2$ and $n+3$.

\begin{rmk}
    Formula \eqref{eq:decomposition of moduli spaces} was discovered by Petersen and Tommasi \cite{petersen-tommasi}, and they explained it to us in private communication. We present here an alternative proof of this formula building on \cite[Theorem 1.2]{BCGY21}.
    Our determination of the coefficients $r_\lambda$ is new.
    
    The direct analog of \eqref{eq:decomposition of moduli spaces} does not hold in higher genus $g>2$. Instead, Petersen and Tommasi have informed us that it generalizes to the moduli space $\mathcal{H}_{g,n}$ of hyperelliptic curves, thus the calculations of this paper applies to the weight $0$ compactly supported cohomology of those.
\end{rmk}

%

Our computation of $\Phi[n,\lambda]$ for all $n\leq 10$ mentioned in \S\ref{sec:explicit computation} recovers the $\SymGroup$-character of $\gr^W_0 H^*_c(\ModuliCurve{2,n})$ as appearing in \cite{BCGY21}. We furthermore computed all terms $\Phi[11,\TransposePartition{(a,b)}]$, appearing in \eqref{eq:decomposition of moduli spaces}, thus obtaining the character of $\gr^W_0 H^*_c(\ModuliCurve{2,11})$. We do not have explicit calculations beyond that.
Nevertheless, in \S\ref{sec:cohomology-M2n} we give a new lower bound on the character of $\gr^W_0 H^*_c(\ModuliCurve{2,n})$, which grows very rapidly, and for $n\leq 11$ it accounts for a large portion of the full weight $0$ cohomology. 
%
%


\subsection{Hochschild cohomology and configuration spaces}

The Hochschild--Pirashvili homology of a wedge of circles is an invariant of a commutative algebra, which turns out to be fundamental to a wide range of fields: from rational homotopy theory of spaces of `long embeddings' $\mathbb{R}^m\into \mathbb{R}^n$ as studied by Turchin--Willwacher \cite{TW17}, to the theory of `Outer' polynomial functors on the category of free groups studied by Powell--Vespa \cite{PV18}.

We show in \S\ref{sec:hochschild}, these cohomology groups are related via Schur--Weyl duality to $H^*_c(F(\bigvee_{i=1}^g S^1,n);\Q)$. 
This relation is known to some experts, but it does not appear in the literature in a form that we can readily use. For the benefit of the reader we give the following identification. Let $A_{(n)} = \Q[\epsilon_1,\ldots,\epsilon_n]/(\epsilon_i\epsilon_j \mid i\leq j)$ be the square-zero algebra on $n$ generators, concentrated in degree $0$, equipped with an $\SymGroup$-action permuting the generators and multigraded by the degree of each generator; and let $A^{(n)}$ be its linear dual coalgebra.

\begin{thm}[\textbf{Hochschild cohomology and configurations}] \label{thm:hochschild and configurations}
Let $X$ be a simplicial set with finitely many non-degenerate simplices, and geometric realization $|X|$. For every positive integer $n$ the cohomology $H^*_c(F(|X|,n), \Q)$ is isomorphic to the $(1,\ldots,1)$-multigraded component of Hochschild--Pirashvili cohomology of $X_+ = X\sqcup \{*\}$ with coefficients in the coalgebra $A^{(n)}$. This isomorphism respects the $\SymGroup$-action on both sides, and is natural with respect to maps $X\to X'$.

In particular, for $X=\bigvee_{i=1}^g S^1$, whose group homotopy auto-equivalences is naturally identified with $\Out(F_g)$, this specializes to an isomorphism that respects the natural actions of $\SymGroup$ and $\Out(F_g)$ on both objects, with $\Out(F_g)$ identified with the homotopy equivalences $\bigvee_{i=1}^g S^1\to\bigvee_{i=1}^g S^1$.
\end{thm}

We think of this relationship as a geometric interpretation of the algebraic construction of Hoch\-schild--Pirash\-vili 
homology in \cite{Pir00}, and we prove it in \S\ref{sec:labeled configs}. 
This theorem reformulates a hard open problem of Hochschild--Pirashvili (co)homology \cite[\S 2.5]{TW19} as the following.

\subsection{Relations to previous work}
The representations $H^*_c(F(X,n))$ appear in previous work in the following forms. Turchin--Willwacher define in \cite[\S2.5]{TW19} a class of $\Out(F_g)$-representations which they call \emph{bead representations}. They show that these do not factor through $\GL(g, \Z)$, and are the smallest known representations with this property. Turchin--Willwacher pose the (still open) problem of describing the representations, and explain that they play a role in the rational homotopy theory of higher codimensional analogues of string links. Our current work provides a geometric interpretation of these representations as $H^*_c(F(\vee_{i=1}^g S^1,n))$, and computing our coefficients $\Phi[n,m]$ is equivalent to calculating the composition factors of the bead representations.

In later work Powell--Vespa \cite{PV18} discovered that the linear duals to the bead representations are fundamental to the theory of polynomial functors on the category of finitely generated free groups. They establish vast machinery to study these objects and compute the representations in some cases that are relevant to our work here. To make use of their work, and to make it more accessible to a topologically minded reader, we include a dictionary between their terminology and ours in \ref{sec:bead}. That section will recall the construction and significance of bead representations, and their relation to our work.

Another precursor to our project is Hô's study of factorization homology and its relations to configuration spaces \cite{Ho20}. On the one hand rational Hochschild homology is known to be a special case of factorization homology, and on the other hand Hô showed that many variants of configuration spaces arise as factorization homology with suitable coefficients \cite[Proposition 5.1.9]{Ho20}. Essentially our Theorem \ref{thm:hochschild and configurations} is a special case of his work, but for the benefit of the reader we include our own proof.

\subsection{Acknowledgments}
We are grateful to Greg Arone, Dan Petersen, Orsola Tommasi and Victor Turchin for helpful conversations related to this work. We further thank Andrea Bianchi, Guillaume Laplante-Anfossi, Geoffrey Powell and Christine Vespa for their feedback on a previous version of the paper.
The second author benefited from the guidance of Dan Petersen as his Ph.D. supervisor. We thank ICERM and Brown University for generously providing the computing resources on which we ran some of our programs.
N.G. is supported by NSF Grant No. DMS-1902762; L.H. is supported by ERC-2017-STG 759082.

\section{Preliminaries}
\subsection{Polynomial functors on \texorpdfstring{$\Vect{\Q}$}{VectQ}} \label{sec:prelim polynomial functors}

Eilenberg and MacLane introduced in \cite{EM54} the notion of polynomiality for functors between categories of modules, defined using the notion of \emph{cross-effect functors}.
In the present work we use functors $\Vect{\Q}\to\Vect{\Q}$, sending (finite-dimensional) rational vector spaces to rational vector spaces. In this specific case the category of polynomial functors is semi-simple and the general theory of polynomial functors admits a simpler presentation,  as recalled below. For additional details we refer to Joyal \cite{Joyal86} and Macdonald \cite[Appendix I.A]{MacDonald}.
The end of this section includes a brief discussion of the notion of polynomiality for functors $\bm{grp} \to \Vect{\Q}$, from finitely generated free groups to rational vector spaces, as considered in \cite{DV15}.

\begin{defi}[\textbf{Symmetric sequences}]
    A \emph{symmetric sequence} (of $\Q$-vector spaces) is a sequence of vector spaces $(A_m)_{m=0}^\infty$ such that $A_m$ is equipped with a linear action of the symmetric group $\SymGroup[m]$. The sequence is \emph{finitely supported} if $A_m=0$ for all $m\gg 0$; denoted as a finite sequence $(A_m)_{m=0}^n$ for some $n\geq 0$.
\end{defi}

\begin{prop}[\textbf{Polynomial functors}{, \cite[Appendix I.A]{MacDonald}}]\label{prop:polynomiality-Qmod}
    Fix $n\geq 0$ and let $(A_m)_{m=0}^n$ be a finitely supported symmetric sequence. A functor $\Psi\colon \Vect{\Q}\to\Vect{\Q}$ of the form 
    \begin{equation} \label{eq:polynomial definition}
        \Psi(V) = \bigoplus_{m=0}^n{A_m\otimes_{\SymGroup[m]}{V^{\otimes m}}}
    \end{equation}
    is \emph{polynomial of degree $\leq n$} in the sense of Eilenberg--MacLane. Moreover, every polynomial functor $\Psi: \Vect{\Q}\to \Vect{\Q}$ is isomorphic to one obtained by the above construction, for representations $A_m$ determined uniquely by $\Psi$ up to isomorphism.
\end{prop}
 In the above proposition, the representation $A_m$ is called the $m$-th \emph{coefficient} of the polynomial functor $\Psi$. The largest $m$ for which $A_m \neq 0$ is the \emph{degree} of $\Psi$.
 \begin{defi}
    Analogously, for a general symmetric sequence $(A_m)_{m= 0}^{\infty}$, the functor
    \[
        \Psi(V) = \bigoplus_{m= 0}^\infty{A_m\otimes_{\SymGroup[m]}{V^{\otimes m}}}
    \]
    is called \emph{analytic}.
\end{defi}

Famously, analytic functors $\Vect{\Q}\to\Vect{\Q}$ contain exactly the data of their sequence of coefficients (see e.g. \cite[Appendix I.A]{MacDonald}). Stated precisely,

\begin{prop}\label{prop:equiv symmetric sequence analytic functors}
    {The construction $(A_0,\ldots, A_m,\ldots)\mapsto \Psi$ defines an equivalence of categories between the category of symmetric sequences and the category of analytic functors. It further restricts to  an equivalence between the subcategory of finitely supported symmetric sequences and the subcategory of polynomial functors.}

The inverse construction is given by sending $\Psi$ to $(A_0,A_1,\ldots)$ where $A_n$ is the `multi-linear part'
\[
A_n = \Psi( \Q^n  ) ^{(1,\ldots,1)}
\]
on which diagonal $(n\times n)$-matrices 
act with weight $(1,\ldots,1)$.
\end{prop}

\begin{rmk}
    In \cite{MacDonald} Macdonald gives a slightly different definition of polynomial functors than Eilenberg--Maclane's, but this difference is immaterial in our context.
    In light of the previous proposition, one may take functors of the form \eqref{eq:polynomial definition} as a definition of polynomial functors $\Vect{\Q}\to\Vect{\Q}$. However, when $\Vect{\Q}$ is replaced with a different category of modules, there do exist polynomial functors that are not characterized by a symmetric sequence of coefficients. In that case, functors as in \eqref{eq:polynomial definition} define a proper subcategory of polynomial functors.
\end{rmk}

Given two polynomial functors $\Psi_1$ and $\Psi_2$ of respective degrees $\deg \Psi_1 = d_1$ and $\deg \Psi_2 = d_2$, their sum $\Phi\oplus\Psi$, product $\Phi\otimes\Psi$, and composition $\Phi\circ\Psi$ are also polynomial functors. Moreover
\begin{equation*}
    \deg (\Psi_1\oplus \Psi_2) = \max (d_1, d_2) \quad
    \deg (\Psi_1\otimes \Psi_2) = d_1+d_2 \quad
    \deg (\Psi_1\circ \Psi_2) = d_1 d_2.
\end{equation*}

Proposition \ref{prop:equiv symmetric sequence analytic functors} implies that polynomial functors form a semi-simple abelian category, and the irreducible objects are the so called \emph{Schur functors}
\begin{equation}\label{eq:defi-Schur-functors}
    \Schur{\lambda}: V \longmapsto \Specht{\lambda} \otimes_{\SymGroup[m]} V^{\otimes m}
\end{equation}
where $\Specht{\lambda}$ is the irreducible $\SymGroup[m]$-representation corresponding to the partition $\lambda\vdash m$. Most familiar are the symmetric and alternating powers $\Schur{(m)}(V) = \Sym^m(V)$ and $\Schur{(1^m)}(V) = \Lambda^m(V)$.
Recall that the space $\Schur{\lambda}(V)$ is non-trivial if and only if the number of parts in $\lambda$ does not exceed $\dim V$, and in that case the resulting $\GL(V)$-representations for different $\lambda$ are all distinct.

More generally, if $\Psi$ is a polynomial functor and its coefficients decompose as $ A_m \cong \bigoplus_{\lambda\vdash m}A_\lambda \otimes\Specht{\lambda}$, i.e. the $\Specht{\lambda}$-multiplicity space is $A_\lambda$, then 
\[
\Psi(-) \cong \bigoplus_{\lambda} A_\lambda\otimes \Schur{\lambda}(-)
\]
an irreducible decomposition of $\Psi$ into Schur functors, with sum over all partitions $\lambda$.

In this paper we will actually not work with the category $\Vect{\Q}$ but with the category $\GradedVect_{\Q}$ of \emph{graded} rational vector spaces, with the symmetry of $\otimes$ obeying the Koszul sign rule; everything said until now applies verbatim to this situation. The coefficients of polynomial functors $\GradedVect_{\Q}\to\GradedVect_{\Q}$ can be determined by considering graded vector spaces $V$ of finite type, concentrated in a single degree:

\begin{prop}\label{prop:image-polynomial-concentrated}
Let $V = \Q^d[-i]$ be a graded vector space of rank $d$ concentrated in degree $i$, and let $\Psi$ be a polynomial functor with coefficients $A_m \cong \bigoplus_{\lambda \vdash m}{A_{\lambda}\otimes \Specht{\lambda}}$ where $A_\lambda$ are graded vector spaces. Then
\[
    \Psi(V) = \begin{cases}
        \bigoplus_{m\geq 0}\bigoplus_{\lambda\vdash m}{A_{\lambda}\otimes \Schur{\lambda}(\Q^d)[-mi]} & \textrm{if $i$ is even}  \\
        \bigoplus_{m\geq 0} \bigoplus_{\lambda\vdash m}{A_{\TransposePartition{\lambda}}\otimes \Schur{\lambda}(\Q^d)[-mi]} & \textrm{if $i$ is odd}
    \end{cases}
\]
where the sum runs over all partitions and
$\TransposePartition{\lambda}$ denotes the conjugate partition of $\lambda$.
\end{prop}

In other words, when $V$ is concentrated in even degree, the value $\Psi(V)$ as a $\GL(V)$-representation determines the coefficients of $\Psi$ with at most $\dim(V)$ parts. On the other hand when $V$ is concentrated in odd degree, the $\GL(V)$-representation $\Psi(V)$ determines the coefficients with all parts having size at most $\dim(V)$.

\subsubsection{Polynomial functors from free groups}\label{sec:polynomiality-gr}
We also consider the notion of polynomiality for functors out of the category $\bm{grp}$ of \emph{finitely generated free groups}. The relevant variant we need is that of contravariant functors  $\Psi\colon \bm{grp}^{op}\to \Vect{\Q}$, however since this notion is not the main focus of this paper, we will only sketch the general idea and refer the interested reader to \cite[\S 3.2]{HPV15} for a detailed presentation.

The category $\bm{grp}$ consists of free groups $F_d \cong \langle a_1,\ldots, a_d\rangle$ for every nonnegative integer $d$ and group homomorphisms. The free product $F_n\ast F_m \cong F_{n+m}$ and the trivial group $F_0$ endow this category with the structure of a pointed monoidal category; and \cite[\S 3]{HPV15} define contravariant polynomial functors of degree $n-1$ on $\bm{grp}$ to be ones for which the $n$-th \emph{cross effect} functor
\[
\operatorname{cr}_n\Psi(F_{g_1},\ldots,F_{g_n}) := \ker\left( \Psi(F_{g_1}*\cdots*F_{g_n} ) \to \bigoplus_{i=1}^n \Psi(F_{g_1}*\cdots*\widehat{F_{g_i}}*\cdots*F_{g_n}) \right)
\]
induced by the morphisms $1\to F_{g_i}$ vanish identically.

\begin{rmk}
    The reference \cite[\S 3]{HPV15}, as well as a large part of the literature, focuses more on \emph{covariant} polynomial functors. While the categories of covariant and contravariant polynomial functors on $\bm{grp}$ are not equivalent, the subcategories of functors taking finite-dimensional values are related via linear duality, see \cite[\S 7]{Pow21} for details. 
    In particular, since the polynomial functors that we will encounter take finite-dimensional values, we will allow ourselves to dualize statements about covariant polynomial functors. 
\end{rmk}

A major difference between polynomial functors $\bm{grp}^{op}\to \Vect{\Q} $ and  $\Vect{\Q}\to \Vect{\Q}$ is that the former category is not semi-simple, and the analogue of Proposition \ref{prop:polynomiality-Qmod} fails. Indeed, polynomial functors on $\bm{grp}$ have a more complicated structure, described next.

An important class of polynomial functors on $\bm{grp}$ is given by precomposing polynomial functors $\Vect{\Q}\to\Vect{\Q}$ by the rationalization $\iota\colon\bm{grp}\to \Vect{\Q}$
\[
F_d \overset{\mathfrak{ab}}\longmapsto \mathbb{Z}^d \overset{\otimes \Q}\longmapsto \Q^d
\]
and linear duality. In this way, every Schur functor $\Schur{\lambda}$ as defined in \eqref{eq:defi-Schur-functors} gives rise to a polynomial functor $\bm{grp}^{op}\to \Vect{\Q}$ of degree $|\lambda|$, denoted by $\iota^* \Schur{\lambda}$. As before, these functors still comprise the simple polynomial functors, however now they admit non-trivial extensions.

\begin{prop}[{Polynomial filtration, \cite[Corollaires 3.6, 3.7]{DV15}}] 
    Every polynomial functor $\Psi: \bm{grp}^{op}\to \Vect{\Q}$ admits a natural `polynomial' filtration by degree. 
    The associated graded functor splits as the direct sum of functors of the form $\iota^* \Schur{\lambda}$.
\end{prop}

Stated differently, every polynomial functor on $\bm{grp}$ is obtained as an iterated extension of functors $\iota^* \Schur{\lambda}$. These functors $\iota^* \Schur{\lambda}$ are basic examples of the following notion.

\begin{defi}[Outer functors]\label{def:outer-functor}
    A functor $\Psi: \bm{grp}^{op}\to\Vect{\Q}$ is said to be an \emph{outer functor} if it sends inner automorphisms to the identity map. That is, if for every free group $F_g$ the $\Aut(F_g)$-action on $\Psi(F_g)$ factors through $\Out(F_g) = \Aut(F_g)/\operatorname{Inn}(F_g)$ (recall that an automorphism is inner if it is the conjugation by a fixed element).

    An outer functor is \emph{polynomial} if it is both outer and polynomial in the sense discussed above.

\end{defi}

\subsection{Cohomology with compact support of configuration spaces}\label{sec:configurations with compact support}

Homology and cohomology in this paper will be taken with rational coefficients unless otherwise stated, and we will henceforth suppress the coefficients from the notation. That is, $H^*_c(-)$ will denote rational singular cohomology with compact support $H^*_c(-;\Q)$. 

Recall that $H_c^*(-)$ is related to homology and cohomology as follows. First, when $X$ is a locally compact space then $H_c^*(X)\cong \tilde H^*(X^+)$ for $X^+$ the one-point compactification, and when $X$ is an $n$-manifold Poincaré duality identifies $H^i_c(X) \cong H_{n-i}(X) \cong H^{n-i}(X)^{\vee}$ (the second isomorphism needs the homology to be finite-dimensional). Furthermore, $H^*_c(X)$ is itself the linear dual of \emph{Borel-Moore homology}, which we denote by $H^{BM}_*(X)$.

The configuration spaces $F(X, n)$ are famously not a homotopy invariant of the space $X$. For example the real line $\R$ is homotopy equivalent to a point, but $F(\{\ast\},2) = \emptyset$ while $F(\R,2) \neq \emptyset$; a highly nontrivial example of failure of homotopy invariance is found in \cite{LS05}. Also, a continuous map $f\colon X \to Y$ does not induce a map $F(X, n) \to F(Y, n)$ if $f$ is not injective and $n\geq 2$. We show however that at the level of cohomology with compact support homotopy invariance and functoriality do hold, at least for proper maps.

\begin{prop}
    Let $f\colon X\to Y$ be a proper map between locally compact Hausdorff spaces. For every positive integer $n$, the map $f$ induces a natural morphism in each degree $i$
    \[
        H_c^i(F(Y, n)) \to H_c^i(F(X, n)).
    \]
\end{prop}

\begin{proof}
    Since $f$ is proper, the $n$-fold cartesian product $f^{\times n}\colon X^n\to Y^n$ is also proper. The subset $U = (f^{\times n})^{-1}(F(Y, n))$ is open in $F(X, n)$ and the restriction of $f^{\times n}$ to $U$ is a proper map $f_n:= f^{\times n}|_U\colon U\to F(Y, n)$. We can therefore construct the morphism $H_c^i(F(Y, n)) \to H_c^i(F(X, n))$ as the composition
    \begin{equation}\label{eq:induced map on configurations}
        H_c^i(F(Y, n)) \overset{f_n^*}{\longrightarrow} H_c^i(U) \longrightarrow H_c^i(F(X, n))
    \end{equation}
    where the second map is extension by zero.
    
    Alternatively, $f$ induces a map $F(X,n)^+\to F(Y,n)^+$ between the one-point compactifications, where any collisions in $Y$ are sent to $\infty$. The induced map on cohomology is, with the identification $H^*_c(F(X,n)) \cong \tilde{H}^*(F(X,n)^+)$, the composition in \eqref{eq:induced map on configurations}.
\end{proof}

With this alternative definition it is easy to see that if $f,g:X\to Y$ are homotopic through a \emph{proper} homotopy $H:X\times I \to Y$, then they induce equal maps on $H^*_c(F(-,n))$.


    
    

\begin{cor}\label{cor:proper-homotopy-invariance}
    A proper homotopy equivalence $f\colon X\to Y$ induces an isomorphism $H_c^*(F(Y, n)) \cong H_c^*(F(X, n))$ for each $n$. In particular, if $X$ and $Y$ are compact, the same follows for any homotopy equivalence.
\end{cor}

A special case of this corollary is that for a finite connected graph $G$, the groups $H_c^*(F(G, n))$ only depend on the first Betti number (or loop order) of $G$.

Another implication of functoriality and homotopy invariance is the following.
{\begin{cor}\label{action-monoid-self-maps}
    Let $X$ be a compact space and $n$ a positive integer. The monoid $\EndSpace{X}$ of continuous self maps acts on $H_c^{\ast}(F(X, n))$, and this action factors through its monoid of connected components $[X,X] = \pi_0(\EndSpace{X})$.
\end{cor}}

\section{Geometry of Hochschild homology}
\label{sec:hochschild}
This section highlights the geometric description of Hochschild--Pirashvili cohomology as compactly supported cohomology of a configuration space. This presentation highlights the fact that the cohomology acquires an additional $\Lie$-module structure, discussed in \S\ref{sec:lie structure}. Readers interested only in our new calculations and results can safely skip this section.

Let us first briefly explain the notation used below. A more detailed introduction to Hochschild--Pirashvili cohomology can be found in \cite{Pir00}.
Throughout this section $\kk$ is any commutative ring.

\begin{rmk}[Basepoints]
If $X$ is a (simplicial) set, we will denote by $X_{+}$ the pointed (simplicial) set obtained by adjoining to $X$ a disjoint basepoint. Let $\Fin_{\ast}$ be the category of finite pointed sets, a skeleton of which is given by the sets $\FinPointedSet{n} = \{\ast, 1,\ldots, n\}$ for all $n\in \N$.
\end{rmk}

Let $A$ be an augmented commutative $\kk$-algebra, with augmentation morphism $\epsilon\colon A\to\kk$. Define a covariant functor $A_{\bullet}\colon \Fin_{\ast}\to \operatorname{Alg}_{\kk}$, sending $\FinPointedSet{n}$ to $A^{\otimes n}$ and the morphism $\alpha\colon \FinPointedSet{m}\to \FinPointedSet{n}$ to the morphism $\alpha_{\ast}\colon A^{\otimes m}\to A^{\otimes n}$ defined on pure tensors by
\[
\alpha_{\ast}(a_1\otimes\ldots\otimes a_m) = \Big(\bigotimes_{j=1}^n{\prod_{i\in \alpha^{-1}(j)} a_i}\Big)\cdot \epsilon\Big(\prod_{i\in\alpha^{-1}(*)}{a_i}\Big),
\]
with the convention that $\prod_{i\in \emptyset}a_i = 1$ and $a_* = 1$. More conceptually, $A_{\bullet}$ is the unique coproduct-preserving functor $\Fin_{\ast}\to\operatorname{Alg}_{\kk}$ sending $\FinPointedSet{1}$ to $A$ and the map $\FinPointedSet{1}\to\FinPointedSet{0}$ to the augmentation map $\epsilon\colon A\to\kk$.

\begin{defi}[Hochschild--Pirashvili homology]\label{def:HP-homology}
    Let $X\colon \Delta^{op}\to \Fin_{\ast}$ be a pointed simplicial set with finitely many $p$-simplices for all $p$ and let $A$ be an augmented commutative $\kk$-algebra. The composition $A_{\bullet}\circ X$ produces a simplicial $\kk$-module, and its associated chain complex is denoted $CH_{\ast}(X; A)$. Define the \emph{Hochschild--Pirashvili homology of $X$ with coefficients in $A$} to be the homology of this chain complex, denoted $HH_{\ast}(X; A)$.
\end{defi}

\begin{rmk}\label{rmk:Kan-extension-A-bullet}
    The previous definition extends to general simplicial sets $X$ by considering the colimit of $CH_{\ast}(X';A)$ over all finite simplicial subsets $X' \subseteq X$, i.e. the left Kan extension. However, this extension will not be needed in our applications below.
\end{rmk}

When working instead with $C$, a coaugmented cocommutative coalgebra, dualizing the previous definition gives $HH^{\ast}(X; C)$, the \emph{Hochschild--Pirashvili cohomology of $X$ with coefficients in $C$}. Explicitly, one defines a contravariant functor $C^\bullet:\Fin_*^{op} \to \operatorname{Mod}_{\kk}$ (in fact taking values in coalgebras), and the composition $C^\bullet \circ X$ is a cosimplicial $\kk$-module. The associated cochain complex $CH^*(X;C)$ has cohomology $HH^{\ast}(X; C)$ by definition.

\begin{rmk}[Loday construction]
    Definition \ref{def:HP-homology} is a special case of a more general construction, taking a pair $(A, M)$ of a commutative $\kk$-algebra $A$ and a module $M$ over $A$ and producing a covariant functor $\Loday{A}{M}\colon \Fin_*\to \operatorname{Mod}_{\kk}$ (see \cite[\S 1.7]{Pir00}). Then $HH_{\ast}(X; \Loday{A}{M})$ is again defined as the homology of the chain complex associated to the simplicial $\kk$-module $\Loday{A}{M}\circ X$. This construction has the property that $HH_{\ast}(S^1; \Loday{A}{M})\cong HH_{\ast}(A, M)$, the classical Hochschild homology. The special case we consider in this paper is related to those in \cite{TW19} and \cite{PV18} via the following identification.
    


    For $X$ a pointed simplicial set, $A$ an augmented $\kk$-algebra and $C$ a coaugmented coalgebra, there exist isomorphisms (compare \cite[Lemma 13.11]{PV18})
    \begin{equation}
        CH_{\ast}(X_+; A) = CH_{\ast}(X_+; \Loday{A}{\kk}) \cong CH_{\ast}(X; \Loday{A}{A}),
    \end{equation}
    \begin{equation}
        CH^{\ast}(X_+; C) = CH^{\ast}(X_+; \Loday{C}{\kk}) \cong CH^{\ast}(X; \Loday{C}{C}),
    \end{equation}
    natural in $X$, $A$ and $C$, and thus the Hochschild (co)homologies are the same.
\end{rmk}

To state the main claim of this section we introduce some terminology. Let $V$ be a free $\kk$-module of finite rank, then the \emph{square-zero} algebra $A_V := \kk \oplus V$ is the unital $\kk$-algebra with trivial multiplication on $V$. Dually, define $A^V := \kk \oplus V^{\vee}$ the coalgebra cogenerated by primitive elements $V^\vee$, and note that $A^V \cong (A_V)^\vee$.
The central goal of the section is the following.

\begin{prop}\label{prop:geometry of HH}
Let $X$ be a simplicial set with finitely many simplices in every degree, and for every $n\geq 0$ consider the coalgebra $A^{\Q^n}\cong \Q \oplus \Q\epsilon_1\oplus \ldots \oplus\Q\epsilon_n$ cogenerated by the primitive cogenerators $\epsilon_1,\ldots, \epsilon_n$. There is an $\SymGroup$-equivariant isomorphism
\begin{equation}\label{eq:iso-Hochschild-compact-support}
    HH^{\ast}(X_+; A^{\Q^n})^{(1,\ldots, 1)} \cong H_c^{\ast}(F(|X|, n))
\end{equation}
where the superscript $(1,\ldots,1)$ denotes the multi-graded summand of Hochschild--Pirashvili cohomology that is multi-linear in the $\epsilon_i$'s (see Proposition \ref{prop:equiv symmetric sequence analytic functors}).
\end{prop}

\begin{rmk}
For a general simplicial set $X$ there is a homotopy invariant reformulation as
\[
HH^{\ast}(X_+; A^{\Q^n})^{(1,\ldots, 1)} \cong \tilde{H}^{\ast}(|X|^n/\Delta_n(|X|)).
\]
This becomes expression \eqref{eq:iso-Hochschild-compact-support} involving configuration spaces when $|X|$ is compact.
\end{rmk}


\begin{rmk}
    An equivalent formulation of Proposition \ref{prop:geometry of HH} is the following: Consider the symmetric sequence given by $n \mapsto H^{*}_c\left( F(|X|,n)\right)$. Then the corresponding functor
    \begin{equation}\label{eq:Hochschild functor}
    V \mapsto \prod_{n\geq 0} H^{*}_c\left( F(|X|,n)\right)\otimes_{\SymGroup} (V^\vee)^{\otimes n}
    \end{equation}
    is naturally isomorphic to Hochschild--Pirashvili cohomology $V \mapsto HH^*(X_+; {A^V})$. In particular the above functor is the linear dual of the functor at the subject of Powell--Vespa's \cite{PV18}. Hence the close relation between our work and theirs.
\end{rmk}

\subsection{Labeled configuration spaces}
We proceed by constructing a simplicial set whose simplicial homology with $\kk$-coefficients is the Hochschild--Pirashvili homology. Let $\operatorname{ComMon}(\Set_{\ast})$ be the category of commutative monoids in the category of pointed sets, i.e. $M\in \operatorname{ComMon}(\Set_{\ast})$ is a pointed set equipped with a unit $1\in M$ and a product $M\wedge M \to M$ satisfying the usual axioms of commutative monoids (here $\wedge$ is the smash product). Note that $M$ has a canonical augmentation $\epsilon: M\to \{*,1\}$, sending every non-invertible element to the basepoint.

Given a monoid $M\in \operatorname{ComMon}(\Set_{\ast})$, its \emph{Loday construction} is defined analogously to Definition \ref{def:HP-homology}. This is the covariant functor $M_{\bullet}\colon \Fin_{\ast} \to \operatorname{ComMon}(\Set_{\ast})$, defined on objects as $\FinPointedSet{n} \mapsto M^{\wedge n}$.

\begin{defi}[Hochschild simplicial set]\label{def:Loday-pointed-sets}
    Given $M\in \operatorname{ComMon}(\Set_{\ast})$ as above, and $X$ a pointed simplicial set with finitely many $p$-simplices for every $p$, the \emph{Hochschild simplicial set} is defined as the composition $M_{\bullet} \circ X$.
    
    As in Remark \ref{rmk:Kan-extension-A-bullet}, this definition extends to general pointed $X$ by considering the colimit over finite simplicial subsets $X'\subseteq X$.
\end{defi}

Given a pointed set $S$, let the reduced $\kk$-module spanned by $S$ be $\tilde{\kk}[S]:= \kk[S]/\kk[*]$ where $*\in S$ is the basepoint. Then, $\tilde{\kk}[S\vee T] \cong \tilde{\kk}[S]\oplus \tilde{\kk}[T]$ and $\tilde{\kk}[S\wedge T]\cong \tilde{\kk}[S]\otimes \tilde{\kk}[T]$. It is thus immediate that $\tilde{\kk}[-]$ takes monoids in pointed sets to augmented $\kk$-algebras. Furthermore, applying $\tilde{\kk}[-]$ to a pointed simplicial set $X$ results in a simplicial $\kk$-module, which under the Dold-Kan correspondence coincides with the reduced simplicial chain complex $\tilde{C}_*(X)$. These constructions are all compatible with $\vee$ and $\wedge$ of pointed simplicial sets.

\begin{rmk}\label{rmk:simplicial-set-Hochschild}
    Due to the strong monoidality of $\tilde{\kk}[-]$, it follows that the complex of reduced simplicial $\kk$-chains of the Hochschild simplicial set $M_{\bullet}\circ X$ is isomorphic to the Hochschild--Pirashvili chain complex $CH_*(X; \tilde{\kk}[M])$.
\end{rmk}

In the rest of this subsection we introduce the notion of the \emph{configuration space of $X$ with labels in $M$} and prove that this simplicial set is isomorphic to the Hochschild simplicial set. 

For every simplicial set $X$, the Yoneda embedding provides a contravariant functor $X^{\bullet}\colon \Fin_{\ast}^{op}\to \operatorname{sSet}_{\ast}$, defined as the pointwise hom-functor $\operatorname{Map}_{\ast}(-, X)$, sending $\FinPointedSet{n}$ to the pointed simplicial set $\operatorname{Map}_{\ast}(\FinPointedSet{n}, X)\cong X^{\times n}$ whose basepoint is the constant pointed map.

\begin{defi}
    Given $M\in \operatorname{ComMon}(\Set_{\ast})$ and a pointed simplicial set $X$, the \emph{configuration space of $X$ with labels in $M$} is the pointed simplicial set defined as the following coend
    \begin{equation}
        \ConfLabeled{X}{M} := \left. \left(\bigvee_{n\geq 0} X^{\times n} \wedge M^{\wedge n}\right) \middle /\sim \right.
\end{equation}
where $\sim$ is the equivalence relation
\[
    (\alpha^{\ast}(x_1, \ldots, x_n), (s_1, \ldots, s_m))\sim((x_1,\ldots, x_n), \alpha_{\ast}(s_1,\ldots, s_m))
\]
for all pointed maps $\alpha\colon \FinPointedSet{m}\to\FinPointedSet{n}$.
\end{defi}
Note in particular that if $\alpha$ is a bijection of $\FinPointedSet{n}$, then $\alpha^{\ast}$ and $\alpha_{\ast}$ act on $X^{\times n}$ and $M^{\wedge n}$ by mutually inverse permutations of the coefficients.

One can think of a point in the labeled configuration space as a tuple of points in $X$, each of which decorated by an element of $M$. When two such labeled points collide in $X$, one replaces them by a single point labeled by the product of the two original labels. Points labeled by the unit $1\in M$ can be introduced or deleted freely.

\begin{prop}\label{prop:coend iso with Loday}
    Let $X$ be a pointed simplicial set 
    and let $M$ be a commutative monoid in pointed sets. Then the Hochschild simplicial set $M_{\bullet}\circ X$ 
    is naturally isomorphic to the labeled configuration space $\ConfLabeled{X}{M}$.
\end{prop}

\begin{proof}
    When $X$ has finitely many nondegenerate simplices, this is essentially the co-Yoneda lemma enriched in simplicial sets. Indeed, at the level of $p$-simplices the co-Yoneda lemma precisely gives a natural isomorphism 
    \begin{equation*}
        \ConfLabeled{X_p}{M} = \operatorname{Map}_{*}(-,X_p)\otimes_{\Fin_*} M_{\bullet} = M_{\bullet}(X_p)
    \end{equation*}
    and by naturality these isomorphisms form an isomorphism of simplicial sets. 
    
    When $X$ is any simplicial set, it is the filtered colimit over its finite simplicial subsets. Since maps from finite sets commute with filtered colimits, and $\otimes_{\Fin_*}$ commutes with all colimits, the same co-Yoneda argument extends to $X$. 
\end{proof}

Combining this proposition with Remark \ref{rmk:simplicial-set-Hochschild}, one obtains the following.

\begin{cor}
    Let $X$ be a pointed simplicial set and let $M$ be a commutative monoid in pointed sets. Then the Hochschild chain complex $CH_*( X; \tilde{\kk}[M] )$ 
    is isomorphic to the reduced simplicial chain complex of the labeled configuration space $\ConfLabeled{X}{M}$ with $\kk$-coefficients.
    Thus, the Hochschild--Pirashvili homology is the homology of the labeled configuration space.
\end{cor}

The motivating example for this discussion is the subject of the next section.

\subsection{Configuration spaces of distinct points} \label{sec:labeled configs}
Let $T$ be a finite set. The \emph{square-zero monoid in pointed sets} generated by $T$ is the pointed set $M[T]:= \{0,1\} \coprod T$, with basepoint $0$, unit $1$ and otherwise trivial multiplication: $s \cdot s' = 0$ for all $s, s'\in S$. 

\begin{claim}\label{geometric-realization-labeled}
Let $X_+$ be a simplicial set with a disjoint basepoint and finitely many non-degenerate simplices. Given a finite set $T$, the configuration space with labels in the square-zero monoid $M[T]$ has geometric realization 
\begin{equation*}
|M[T]_{\bullet}\circ (X_+)| \cong  \bigvee_{n\geq 0} \left(F(|X|,n)\times_{\SymGroup[n]} T^n   \right)^+
\end{equation*}
where $Y^+$ is the one-point compactification of a topological space $Y$. That is, the resulting labeled configuration space is the wedge sum of compactified configuration spaces of particles in $|X|$ with decorations in $T$.
\end{claim}

\begin{proof}

First, we prove that there is an isomorphism of simplicial sets 
\[
    M[T]_{\bullet}\circ (X_+) \cong \bigvee_{n\geq 0}{\Big(X^n/{\Delta_n(X)}\Big) \wedge_{\SymGroup}(T^n)_+}
\]
where $\Delta_n(X) \subseteq X^n$ is the \emph{fat diagonal} -- the simplicial subset whose $p$-simplices are the $n$-tuples of elements in $X_p$ not all of whose coordinates are distinct; equivalently its $p$-simplices are the non-injective functions from $\bm{n}$ to $X_p$.

On the left-hand side, $p$-simplices are described as follows. There is a natural bijection $M[T]_{\bullet}((X_p)_+) \cong \left((\{1\}\coprod T)^{X_p}\right)_+$, further simplified by
\begin{equation}\label{eq:from functor to configurations}
\big(\{1\}\coprod T\big)^{X_p} \cong \coprod_{Y\subseteq X_p} T^Y \cong \coprod_{n\geq 0}\operatorname{Inj}(\bm{n},X_p) \times_{\SymGroup} T^n,
\end{equation}
where $\operatorname{Inj}(\bm{n},X_p)$ denotes the set of injective functions from $\bm{n}$ to $X_p$. Remembering the basepoint, we obtain the natural bijection
\[
    M[T]_{\bullet}((X_p)_+) \cong \bigvee_{n\geq 0}\operatorname{Inj}(\bm{n},X_p)_+\wedge_{\SymGroup}(T^n)_+ \cong \bigvee_{n\geq 0}\Big(X_p^n/{\Delta_n(X_p)}\Big)\wedge_{\SymGroup}(T^n)_+,
\]
where the last isomorphism uses the obvious identification $X_p^n/\Delta_n(X_p) \cong \operatorname{Inj}(\bm{n},X_p)_+$
which sends noninjective functions in $X_p^n$ to the basepoint. It remains to verify that the simplicial maps on both sides commute with these bijections. Crucially, this uses the fact that the multiplication on $T$ is trivial.


Face and degeneracy maps $\alpha: X_p\to X_q$ act on $(\{1\}\coprod T)^{X_p}$ by multiplying the labels of points in every fiber (with the empty product being $1$). But since the product of two elements from $T$ is trivial, a function $\varphi\in (\{1\}\coprod T)^{X_p}$ is sent to the basepoint unless every fiber of $\alpha$ contains at most one element with label in $T$ and every other element is labeled by the unit $1$, in which case it is sent to the composition 
\[
\alpha(\varphi^{-1}(T)) \overset{\alpha^{-1}}{\longrightarrow} \varphi^{-1}(T) \overset{\varphi}\longrightarrow T
\]
extended by $1$ when this is undefined. On the RHS of \eqref{eq:from functor to configurations} the function $\varphi: \varphi^{-1}(T)\to T$ represents the element $((x_1,\varphi(x_1)),\ldots, (x_n,\varphi(x_n)))$ for some enumeration of $\varphi^{-1}(T)$. This tuple is sent under $\alpha$ to $((\alpha(x_1),\varphi(x_1)),\ldots, (\alpha(x_n),\varphi(x_n)))$, which is easily seen to be the tuple corresponding to $\alpha_*(\varphi)$ described above. 



Since geometric realization commutes with wedge sums, Cartesian products and quotients, one obtains a homeomorphism
\[
    |M[T]_{\bullet}\circ(X_+)| \cong \bigvee_{n\geq 0}{\Big(|X|^n/\Delta_n(|X|)\Big)\wedge_{\SymGroup}(T^n)_+},
\]
where one observes that $\Delta_n(|X|) = |\Delta_n(X)|\subseteq |X|^n$ is the topological \emph{fat diagonal}, in which at least two coordinates are equal.

Up to this point, the analysis applies to any simplicial set. When $|X|$ is compact (equivalently, $X$ has finitely many non-degenerate simplices), $|X|^n/\Delta_n(|X|)$ is homeomorphic to the one-point compactification $F(|X|, n)^+$, inducing the claimed homeomorphism.
\end{proof}

    

Every labeled configuration space is naturally filtered by number of points in a configuration. In the case of the last claim the filtration actually splits, and the resulting graded factors are the compactified configuration spaces with a fixed number of points.
This grading is refined further to a multi-grading by listing the multiplicities of every label $s\in T$. E.g. if $T = \{\text{red},\text{blue}\}$, then the component with multi-degree $(n_r,n_b)$ is the configuration space of $n_r$ red points and $n_b$ blue points, all of which distinct, compactified by one point at $\infty$.

\begin{cor}
For $T = \{1,\ldots,n\}$, the $(1,\ldots,1)$ multi-degree component of the realisation $|M[T]_{\bullet}\circ(X_+)|$ is the one-point compactification of the ordinary ordered configuration space of $n$ distinct points in $|X|$. The natural $\SymGroup[n]$-action on $T$ by permutations agrees with the usual action on the configuration space by relabeling points.
\end{cor}

\begin{rmk}
From this point on we will abuse notation and use $X$ to refer both to the simplicial set and its geometric realization. The notation $F(X,n)^+$ for compactified configuration spaces can be understood either as a topological space or as the simplicial set $X^n/\Delta_n(X)$, so that $|F(X,n)^+|\cong F(|X|,n)^+$ when $X$ has finitely many non-degenerate simplices. 
Since simplicial and singular homology agree, the ambiguity in notation is immaterial.
\end{rmk}

Note that linearizing $M[T]$ results in an augmented square-zero $\kk$-algebra in the usual sense: $\tilde{\kk}[M[T]] \cong \kk\cdot 1 \oplus \kk[T]$ with trivial multiplication on $\kk[T]$, denoted by $A_T$. Consider the Hochschild chain complex with coefficient in this square-zero algebra, 

\begin{cor}
Let $X$ be a simplicial set with finitely many non-degenerate simplices. For $T=\{1,\ldots,n\}$ and $A_T = \kk\oplus \kk[T]$ the square-zero $\kk$-algebra generated by $T$, the Hochschild chain complex $CH_\ast(X_+, A_T ) $ is naturally multi-graded by the multiplicity of each $s\in T$, so that the $(1,\ldots,1)$ multi-degree component is equivalent to the reduced chains on the one-point compactification
\begin{equation}\label{eq:chain level hochschild as configurations}
CH_*(X_+, A_T )^{(1,\ldots,1)} \simeq \tilde{C}_{*}\left( F(X,n)^+ ; \kk\right)
\end{equation}
compatibly with the $\SymGroup[n]$-action on both chain complexes.

Similarly, the $(d_1,\ldots,d_n)$ multi-degree component is naturally isomorphic to the reduced chains on the one-point compactification of the configuration space of $\sum d_i$ distinct points with exactly $d_i$ many having label $i\in T$.
\end{cor}

Taking homology on both sides gives a natural isomorphism
\begin{equation}\label{eq:iso-Hochschild-hom}
    HH_{\ast}(X_+; A_S)^{(1,\ldots,1)} \cong \tilde{H}_{\ast}(F(X, n)^+; \kk).
\end{equation}


\begin{cor}
With the notation of the previous corollary, and with $A^T:= \kk\cdot 1 \oplus \kk[T]^\vee$ the coalgebra dual to the square-zero algebra $A_T$,
there is a natural isomorphism
\begin{equation}\label{eq:iso-Hochschild-cohom}
    HH^*(X_+; A^T )^{(1,\ldots,1)} \cong H^{*}_c\left( F(X,n) ; \kk\right).
\end{equation}
\end{cor}

\begin{proof}
We dualize the quasi-isomorphism in \eqref{eq:chain level hochschild as configurations}. On the left-hand side, Hochschild chains dualize to cochains.  
On the right-hand side, one uses the identification between cohomology with compact support and reduced cohomology of the one-point compactification.
\end{proof}

Specializing to rational coefficients and in light of the equivalence between symmetric sequences and analytic functors, these isomorphisms can be restated as the following isomorphisms, natural in the vector space $V$:
\begin{eqnarray} \label{eq:HH lower star}
    HH_{\ast}(X_+; A_V) & \cong & \bigoplus_{n\geq 0}{\tilde{H}_{\ast}(F(X, n)^+; \Q)\otimes_{\SymGroup}V^{\otimes n}}\\ \label{eq:HH upper star}
    HH^{\ast}(X_+; A^V) & \cong & \prod_{n\geq 0}{H_c^{\ast}(F(X, n); \Q)\otimes_{\SymGroup}(V^\vee)^{\otimes n}}
\end{eqnarray}

These facts are the ones claimed in Proposition \ref{prop:geometry of HH}.

\subsection{Lie structure}\label{sec:lie structure}

The interpretation of Hochschild--Pirashvili homology as related to configuration space reveals additional structure, as explained to us by Victor Turchin. 
\begin{prop}
    Let $X$ be a compact CW complex. Then the symmetric sequence $H^*_c(F(X,\bullet))$ is naturally endowed with a right module structure over the suspended Lie operad $\OperSuspension\Lie$, that is a map $H^*_c(F(X,\bullet)) \circ \OperSuspension\Lie \to H^*_c(F(X,\bullet))$ satisfying the usual right-module axioms.
\end{prop}
The operadic suspension $\OperSuspension\Lie$ is defined by $\OperSuspension\Lie(n) \cong \sgn[n]\otimes\Lie(n)[1-n]$, so that an algebra over it is equivalent to a Lie algebra structure on the (de)suspension.

\begin{rmk}
Around the time a first version of this article was made public, Christine Vespa informed us that Geoffrey Powell recently studied (outer) polynomial functors of $\bm{grp}$ via right Lie-modules 
\cite{Pow21, Pow22, PowBabyBead}, and the corresponding right Lie structure is the same as described in the previous Proposition. Briefly, Powell uses Morita theory to construct an equivalence of categories between outer polynomial functors on free groups -- i.e. compatible sequences of $\Out(F_g)$-representations -- and representations of the PROP associated with $\Lie$ operad. The equivalence centers around the representations $H^*_c(F(X,n))$ for $X$ a wedge of circles.
See also \cite[\S 2.4-2.5]{Ves22}. 
\end{rmk}

Unpacking the definitions, this structure attaches a map
    \[
    m^\varphi: H^*_c(F(X,n-1)) \to H^{*+1}_c(F(X,n))
    \]
to every surjection $\varphi:[n]\twoheadrightarrow [n-1]$, that are compatible with the symmetric group actions in the obvious way and satisfy appropriate versions of anti-symmetry and Jacobi identity.

The existence of such a structure follows from Koszul duality of the commutative and Lie operads, and the fact that the $\Lie$-module $n\mapsto H^*_c(F(X,n))$ is Koszul dual to the $\operatorname{Com}$-module $n\mapsto H^*(X^n)$ (see \cite[Lemma 11.4]{arone-turchin}). But let us be more explicit -- the operations $m^\varphi$ on $H^*_c(F(X,n))$ are obtained as follows.

Let $X$ be a compact CW complex and denote the hypersurface 
\[
H_{ij} = \{ (x_1,\ldots,x_n)\in X^n \mid x_i=x_j \} \subseteq X^n
\]
for each $i\neq j$ so that the fat diagonal $\Delta_n(X) = \bigcup H_{ij}$. Any surjection $[n]\twoheadrightarrow [n-1]$ that sends $i$ and $j$ to the same image gives an identification $X^{n-1} \cong H_{ij} \subseteq X^n$ by pullback, and the image of the $(n-1)$-st fat diagonal $\Delta_{n-1}(X)\subset X^{n-1}$ can be identified as follows.

Decomposing the $n$-th fat diagonal as a union of two closed subspaces
\begin{equation}
    \Delta_n(X) = H_{ij} \cup C \;\text{ where }\; C = \bigcup_{\{k,l\}\neq \{i,j\}} H_{kl}
\end{equation}
the intersection $H_{ij}\cap C$ is precisely the image of $\Delta_{n-1}(X)$ in $H_{ij}$. In particular, there is an isomorphism in relative cohomology
\begin{equation}
    H^*_c(F(X,n-1)) \cong H^*(X^{n-1},\Delta_{n-1}(X)) \cong H^*(H_{ij},H_{ij}\cap C).
\end{equation}

Now by excision in cohomology
\begin{equation}
    H^*(H_{ij}, H_{ij}\cap C) \cong H^*(\Delta_n(X), C), 
\end{equation}
so the long exact sequence of the triple $(X^n\supset \Delta_n(X) \supset C)$ gives a coboundary map
\begin{equation}
    H^*(\Delta_n(X), C) \overset{d}\longrightarrow H^{*+1}(X^n,\Delta_n(X)) \cong H^{*+1}_c(F(X,n)).
\end{equation}
The composition of the above maps gives an operation
\begin{equation}
    H^*_c(F(X,n-1)) \cong H^*(X^{n-1},\Delta_{n-1}(X)) \to H^{*+1}_c(F(X,n)),
\end{equation}
well-defined for a given choice of surjection $[n]\twoheadrightarrow [n-1]$.

Under the identification of symmetric sequences and analytic functors $(\Phi[\bullet])\leftrightarrow \Psi$, a right $\OperSuspension\Lie$-module structure on a symmetric sequence $(\Phi[\bullet])$ becomes the map of analytic functors
\begin{equation}
    \Psi\left( \operatorname{FreeLie}(V) \right) \to \Psi(V)
\end{equation}
such that every Lie-bracket increases the grading by $+1$.

\begin{cor}
Let $X$ be a simplicial set with finitely many simplices in every degree. The Hochschild--Pirashvili cohomology functor with square-zero coefficients $HH^*(X;A^{\bullet})$ is endowed with the following natural structure:
\begin{equation}
HH^*(X;A^{\operatorname{FreeLie}(V)}) \to HH^*(X;A^V)
\end{equation}
graded such that if $\operatorname{FreeLie}^k(V)$ denotes the subspace of $k$-fold nested brackets then
\begin{equation}
    HH^i(X;A^{\operatorname{FreeLie}^k(V)}) \to HH^{i+k}(X;A^V).
\end{equation}
\end{cor}

If one is interested in Hochschild homology instead of cohomology, one need only dualize this $\Lie$ structure. More explicitly, the Hochschild--Pirashvili homology with square-zero coefficients carries a natural (shifted) right co$\Lie$-comodule structure.

\begin{rmk}[Geometric interpretation]
Thinking of Hochschild--Pirashvili cohomology $HH^*(X;A^V)$ as related to cohomology of configurations of points in $|X|$ with labels in $V$, the Lie-module structure above comes from the following geometric construction.

Suppose a configuration $\{x_1,\ldots,x_n\}\subseteq |X|$ with labels $\operatorname{FreeLie}(V)$ has the point $x_n$ labeled by the bracket $[v,w]$. The retraction $\varphi:[n+1]\twoheadrightarrow [n]$ with $\varphi(n+1)=\varphi(n) = n$ induces the map
\[
m^\varphi: H^*_c(F(X,n)) \to H^{*+1}_c(F(X,n+1))
\]
effectively splitting the $n$-th point in two, as described above. Label the points of this new configuration so that $x_n$ has label $v$, $x_{n+1}$ has label $w$, while all other points retain their original label.
This process can be repeated until there are no points labeled with brackets. The corollary above shows that the composition of these operations is independent of the order in which they were performed, and produces a well-defined cohomology class on configurations with labels in $V$.
\end{rmk}

\section{Geometric approaches to computation} \label{sec:geometric approach}
Now let $X$ be a finite wedge of spheres. Our approach to computations is centered around contrasting two distinct tools for computing $H^*_c(F(X,n))$, each providing complementary information. First we use the so called collision spectral sequence, obtained by filtering the cartesian power $X^n$ by its diagonals. This tool gives a handle on the polynomial structure of $H^*_c(F(X,n))$, but makes the symmetric group action by permutations less accessible.

Our second main tool, applying only when $X$ is a wedge of circles, is the cellular chain complex of the one-point compactification $F(X,n)^+$. This second tool gives a handle on the symmetric group action on $H^*_c(F(X,n))$, but obscures the polynomial structure somewhat.

Our approach lets us calculate an associated graded $\gr H^*_c(F(X,n))$ for all $n\leq 10$. For example, a complete tabulation of the composition factors of $H^9_c(F(X,10))$ can be found in Table \ref{table:final_answer10}. The output of all our calculations can be found by following \href{https://louishainaut.github.io/GH-ConfSpace/}{this URL}\footnote{\url{https://louishainaut.github.io/GH-ConfSpace/}}.

\subsection{Collisions and the CE-complex}\label{section:CE-complex}
We begin by considering the collision spectral sequence. This sequence has been used to study configuration spaces since the 80's, see e.g. \cite{CT78}, \cite{Kriz94}, \cite{Tot96}. More recently, Hô \cite{Ho17} and Petersen \cite{Pet20} recast the spectral sequence as the Chevalley--Eilenberg complex of a twisted Lie algebra. The same spectral sequence also appears in Powell--Vespa's \cite{PV18} as coming from the polynomial filtration of Hochschild--Pirashvili homology, referred to as the Hodge filtration (though this is not in the sense of Hodge theory).

The following construction first appeared in Getzler's \cite{Get99}.
Recall that a twisted dg Lie algebra is a symmetric sequence $\mathfrak{g} = (\mathfrak{g}(n))_{n\in \mathbb{N}}$ of chain complexes that is a Lie algebra object in the category of symmetric sequences, with $\otimes$ given by Day convolution: there is a bracket
\[
[-,-]: \Ind{\SymGroup[n]\times\SymGroup[m]}{\SymGroup[n+m]}{\mathfrak{g}(n)\otimes \mathfrak{g}(m)} \to \mathfrak{g}(n+m)
\]
that is $\SymGroup[n+m]$-equivariant and satisfies appropriate versions of anti-symmetry and Jacobi identity. Equivalently, $\mathfrak{g}$ is equipped with a left module structure over the Lie operad with respect to the composition product $\Lie \circ \mathfrak{g} \to \mathfrak{g}$.

To every such $\mathfrak{g}$ one associates the Chevalley--Eilenberg bi-complex
\[
C^{CE}_{-k}(\mathfrak{g}) := \Sym^k(\mathfrak{g}[1]),
\]
where we view $\mathfrak{g}[1]$ as a bigraded object, with horizontal grading $-1$ and vertical grading given by the internal cohomological grading of $\mathfrak{g}$.
The differentials of this bi-complex are $d$ and $\delta$, where $d$ is the extension of differential on $\mathfrak{g}$ to a signed derivation on tensors, and
\[
    \delta(sg_1 \otimes \ldots \otimes sg_k) = \sum_{i<j}{\epsilon_{ij} s[g_i, g_j]\otimes sg_1 \otimes \ldots \otimes \widehat{sg_i}\otimes \ldots \otimes \widehat{sg_j} \otimes \ldots\otimes sg_k},
\]
with $\epsilon_{ij}$ being the sign induced by the Koszul sign rule when moving $sg_i$ and $sg_j$ to the front. 
The bigrading of $\mathfrak{g}[1]$ induces a bigrading of the double complex $C^{CE}_*(\mathfrak{g})$. The sign $\epsilon_{ij}$ depends on both the horizontal and the vertical grading. 
Note that since we use cohomological grading while the CE complex computes homology, the horizontal grading is negative.

\begin{rmk}\label{rmk:arity-grading-and-filtration-bound}
    Since $\mathfrak{g}$ is a symmetric sequence, the \emph{Chevalley--Eilenberg homology} $H_*^{CE}(\mathfrak{g})$, defined as the cohomology of the total complex associated to $C_*^{CE}(\mathfrak{g})$, admits a further filtration by arity. Moreover both differentials $d$ and $\delta$ preserve the arity filtration, therefore the CE complex naturally splits by arity. With this, we note that if $\mathfrak{g}(0) = 0$ (as will be the case in our situation), then $C_{-k}^{CE}(\mathfrak{g})(n) \neq 0$ only for $-n\leq -k\leq -1$.
\end{rmk}

Further recall that if $A$ is any commutative dg algebra, then $A\otimes \mathfrak{g} = ( A\otimes \mathfrak{g}(n) )_{n\in \mathbb{N}}$ admits a twisted Lie bracket by extending $[-,-]$ in an $A$-bilinear manner.

The central input to our calculation is the following result.
\begin{prop}[\cite{Pet20}, Corollary 8.8]\label{prop:CE-homology is Conf}
Let $X$ be a paracompact and locally compact Hausdorff space, and let $A$ be a cdga model for the compactly supported cochains $C_c^*(X;\Q)$. Then there is a natural isomorphism of symmetric sequences in graded vector spaces
\[ \label{isomorphism_CE}
   {H_c^{*}(F(X, n), \Q)} \cong H_{*}^{CE}(A\otimes \Suspension\Lie)(n)
\]
where $\Lie$ is the Lie operad with its tautological Lie structure coming from operad multiplication, and $\Suspension\Lie$ is its suspension as a twisted Lie algebra. More explicitly,
\[
\Suspension\Lie(n) \cong \sgn[n] \otimes \Lie(n) [ -n ]
\]
where $\Lie(n)$ is the multi-linear part of the free Lie algebra on $(x_1,\ldots,x_n)$, with $|x_i| = 0$, and $V[-n]$ is the chain complex with $V$ placed in degree $n$.
\end{prop}


When $A$ models the cochains on a formal space, as is the case for wedges of spheres, the following simplification holds. Recall that a double complex is naturally filtered by its columns $\mathfrak{F}_p \left(\bigoplus_{s,t} C_{s,t}\right) = \bigoplus_{s\geq p, t} C_{s,t}$; we consider the resulting spectral sequence next.

Let $(C(X),d)$ be a functorial commutative cochain model for $C^*_c(X;\Q)$, where $X$ is any paracompact and locally compact Hausdorff space. Then the canonical filtration by columns of the bicomplex $C^{CE}_*(C(X)\otimes \Suspension\Lie)$ gives rise to a functorial spectral sequence converging to the symmetric sequence $n\mapsto {H^{*}_c(F(X,n);\Q)}$. 

\begin{lem}[\textbf{Formal spaces}, {compare \cite[Remark 4.6]{TW19}}]\label{lem:formal spaces}
When $X$ is a compact formal space, i.e. the cdgas $(C(X),d)$ and  $(H^*(X),0)$ are quasi-isomorphic, the spectral sequence induced by filtering $C^{CE}_*(C(X)\otimes \Suspension\Lie)$ by columns collapses at its $E^2$-page, and $E^\infty \cong E^2 \cong H^{CE}_*( H(X)\otimes \Suspension\Lie)$ functorially in $X$.
\end{lem}

\begin{note}
Even when $X$ is formal, there might exist nonformal maps $X\to X$, i.e. such that the map $C(X)\to C(X)$ is not equivalent to the induced $H(X)\to H(X)$. If not for those, one would simply replace the model $(C(X),d)$ by $(H(X),0)$ and conclude that the CE-bicomplex splits into the direct sum of its rows. However, nonformal maps are responsible for nontrivial extensions between those rows, as is already the case for $X = S^1\vee S^1 \vee S^1$ (see Remark \ref{rmk:extensions}).
\end{note}

\begin{proof}
Filtering the Chevalley--Eilenberg double complex by columns gives a spectral sequence with $E^0$ coinciding with $C^{CE}_*(C(X)\otimes \Suspension\Lie)$ but only with the vertical differentials of $C(X)$. Since $C^{CE}_*$ involves tensor powers and we are working over $\Q$, the Künneth formula gives a natural isomorphism
\begin{equation}\label{eq:chevalley-eilenberg}
E_1 \cong C^{CE}_*(H(C(X)\otimes \Suspension \Lie)) \cong C^{CE}_*(H(X)\otimes \Suspension \Lie)
\end{equation}
where the latter isomorphism follows from the fact that $(-)\otimes \Suspension\Lie$ is exact.

We claim that the spectral sequence collapses at its $E^2$-page. Indeed, since $X$ is compact and formal, the cohomology $H(X) = H^*(X;\Q)$ itself constitutes a commutative cochain model for $C_c^{\ast}(X)$ with vanishing differential. With this model the Chevalley--Eilenberg bicomplex computing $H^{*}_c(F(X,n))$ is exactly the $E^1$-page considered in the previous paragraph. But since the vertical differentials of the latter bicomplex are zero, its spectral sequence collapses at $E^2$. The fact that the spectral sequences for $C(X)$ and $H(X)$ both compute the same cohomology then forces all higher differentials of the former spectral sequence to also vanish.
\end{proof}

\begin{terminology}[\textbf{Collision filtration}]\label{terminology:collision}
The \emph{collision filtration} on $H^{*}_c(F(X,n))$ is the filtration induced by the following shift of the filtration $\mathfrak{F}$ by columns of the CE-bicomplex. For every $n\in \N$ let the $p$-th level of the collision filtration be the $(p-n)$-th filtration by columns
\[
F_p\, H_c^{i}(F(X,n)) \cong \mathfrak{F}_{p-n}\, H_{i+n}^{CE}(C(X)\otimes\Suspension\Lie)(n).
\]
\end{terminology}
The reasoning behind the name is that this filtration comes from filtering the pair 
$(X^n, \Delta^n(X))$ by the number of collisions in the fat diagonal (see \cite[\S3.7]{BG} for details in the dual setting of Borel--Moore homology, but see the next remark).

\begin{rmk}
    Comparing with \cite[Definition 3.7.1]{BG}, we note that this older definition was incorrect. Instead of defining the filtration directly on the homology, one should define the filtration $F_k C_*^{BM}(\operatorname{Conf}^n(X))$ at the chain level in the same fashion as described there, and consider the induced filtration on homology,
    \[
        F_k H_*^{BM}(\operatorname{Conf}^n(X)) := \operatorname{Im}\big[H_*(F_k C_*^{BM}(\operatorname{Conf}^n(X))) \to H_*^{BM}(\operatorname{Conf}^n(X))\big].
    \]
\end{rmk}

\begin{cor} \label{cor:factoring through homological action}
    When $X$ is a compact formal space, e.g. a finite wedge of spheres, then
    $\gr^F H_c^{*}(F(X))$ is naturally isomorphic to the CE-homology $H^{CE}_{*}( H(X) \otimes \Suspension\Lie )$. 
    
    Moreover, for any map $X\to Y$ between compact formal spaces, the induced map $\gr^F H^*_c(F(Y,n)) \to \gr^F H^*_c(F(X,n))$ is computed 
    from the natural map $H^*(Y) \to H^*(X)$ via its action on the CE-homology $H^{CE}_*( H^*(-)\otimes \Suspension\Lie)$.

    In particular, the action of the monoid $\EndSpace{X}$
    on $\gr^F H^*_c(F(X,n))$ factors through $\End(H^*(X))$.
\end{cor}

Unpacking the definition of the CE-complex when $X$ is a wedge of equidimensional spheres and forgetting the $\SymGroup$-action, the first page of the collision spectral sequence takes the following explicit form involving tensor powers.

\begin{prop}\label{prop:spectral_sequence_unidim}
When $X = \bigvee_{i=1}^g S^d$ is a wedge of equidimensional spheres so that $H^d(X) \cong \Q^g$, the $E_1$-page of the collision spectral sequence takes the form
\begin{equation}\label{eq:E1 term unidim}
    E_1^{p, q} \cong \begin{cases}
    T^k(H^d(X))^{\oplus \left(|s(n, n-p)|\cdot \binom{n-p}{k}\right)} & \text{ if $q=dk$ and $p+k \leq n$} \\
    0 & \text{ otherwise}
    \end{cases}
\end{equation}
with $T^k(V) = V^{\otimes k}$, and $|s(n, n-p)|$ denoting an unsigned Stirling number of the first kind (see e.g. \cite[\S 1.3]{Stanley11}).
\end{prop}

Each of these $E_1^{p,q}$-terms admits an action by $\SymGroup$, but we do not have an explicit description of the latter in closed form.
    \begin{proof}[Proof sketch]
        Recall that Chevalley--Eilenberg complex of a twisted Lie algebra $\mathfrak{g}$ has underlying symmetric sequence of vector spaces the composition product $\operatorname{Com}\circ \mathfrak{g}$, with $\operatorname{Com}$ the commutative operad. When $\mathfrak{g} = (\Q \oplus H^d(X)) \otimes \Suspension\Lie$, this composition consists of tensors power of $H^d(X)$ tensored with the composition $\operatorname{Com}\circ \Suspension\Lie$.
        
        On the other hand, applying Petersen's formula for $X=\mathbb{R}^2$ identifies $\operatorname{Com}\circ \Suspension\Lie$ with the cohomology of $F(\mathbb{R}^2, n)$ -- the configuration spaces of $n$ points in the plane. Their Betti numbers are expressed in terms of Stirling numbers -- see \cite{Get-resolving-Hodge}. Keeping track of the filtration and tensor degree gives the stated formula.
    \end{proof}
    
\subsection{Action by outer automorphisms of free groups}\label{sec:out-action}
The wedge $\vee_{i=1}^g S^1$ is a classifying space for the free group $F_g$, thus its self-maps up to homotopy are given by $\End(F_g)$ up to conjugation. As explained below, it follows that the graded vector spaces $H^*_c( F(\vee_{i=1}^g S^1,n))$ are representations of this endomorphism monoid (the statements in this section apply equally well to cohomology with integer coefficients, but we will not need that). 

Pursuing further naturality, we use the description in terms of relative cohomology $H^*_c(F(X,n)) \cong H^*(X^n,\Delta^n(X))$ for every compact Hausdorff space $X$, where $\Delta^n(X)$ is the fat diagonal.  Since $\vee_{i=1}^g S^1$ is homotopy equivalent to the standard simplicial classifying space $BF_g$, and since the functors $X\mapsto X^n$ and $X\mapsto \Delta^n(X)$ 
are homotopy invariant, there is a homotopy equivalence of pairs 
\begin{equation}
    ((\vee_{i=1}^g S^1)^n,\Delta^n(\vee_{i=1}^g S^1)) \;\tilde\longrightarrow\; ((BF_g)^n,\Delta^n(BF_g))
\end{equation}
with the latter pair obviously functorial in $F_g$.
More precisely, let $\mathbf{grp}$ be the category of finitely generated free groups. The following is clear.
\begin{prop}
The construction $F_g \mapsto H^*( (BF_g)^n,\Delta^n(BF_g))$
is a contravariant functor from $\mathbf{grp}$ to graded vector spaces, and for every $g\geq 1$ it is isomorphic to $H^*_c( F(\vee_g S^1,n))$ as $\SymGroup$-representations.

In particular, for every $g$ and $n\geq 1$ the vector spaces $H^*_c( F(\vee_g S^1,n))$ are equipped with a natural action of the monoid $\End(F_g)$.
\end{prop}


Every pointed map of spaces $f:\vee_{i=1}^g S^1 \to \vee_{i=1}^h S^1$ is determined up to homotopy by the induced homomorphism  $\pi_1(f): F_g\to F_h$, and so the above functor from $\mathbf{gr}$ completely characterizes the functoriality of $H^*_c( F(\vee_{i=1}^g S^1,n))$ on wedges of circles.

Of particular interest, every $H^*_c( F(\vee_{i=1}^g S^1,n))$ is a representation of the automorphism group $\Aut(F_g)$. But since an inner automorphism -- conjugation by some $\sigma\in F_g$ -- is given by a map that is nonpointed-homotopic to the identity, its action on $H^*_c( F(\vee_{i=1}^g S^1,n))$ must be trivial. In other words, it is an outer functor (see Definition \ref{def:outer-functor}).
Note that the configuration space functor  $F(-,n)$ admits a natural $\SymGroup$-action, and thus $\SymGroup$ acts on the resulting outer functor by natural transformations.

Lastly, pass to the associated graded of the collision filtration. By Corollary \ref{cor:factoring through homological action} the action of a homomorphism $F_g\to F_h$ on $\gr^F H^*_c( F(-,n))$ is determined by the action on cohomology $\mathbb{Z}^h \to \mathbb{Z}^g$. Making this precise (compare with \cite[Theorems 6.9, 17.8]{PV18}),

\begin{prop}\label{prop:GL action after grading}
Under the collision filtration, the associated graded of the outer functor $F_g \mapsto H^*_c( F(\vee_{i=1}^g S^1,n) ; \mathbb{Q})$ 
is the restriction of a well-defined functor on the category of finitely-generated free abelian groups along the abelianization,
\begin{equation}\label{eq:free-group-polynomial-functor}
    F_g \;\overset{\operatorname{ab}}{\longmapsto} \; \mathbb{Z}^g \longmapsto \gr H^*_c( F(\vee_{i=1}^g S^1,n))
\end{equation}
In particular, for every $g$ and $n\geq 1$ the graded quotients $\gr H^*_c( F(\vee_{i=1}^g S^1,n))$ admit a natural action of the matrix ring $\End(\mathbb{Z}^g)$.
\end{prop}

\begin{proof}
    Every homomorphism $f:\mathbb{Z}^g\to \mathbb{Z}^h$ is the abelianization of some $\bar{f}:F_g\to F_h$. Define $f^*: \gr H^*_c( F(\vee_{i=1}^h S^1,n))\to \gr H^*_c( F(\vee_{i=1}^g S^1,n))$ by $\bar{f}^*$.
    Corollary \ref{cor:factoring through homological action} shows that this does not depend on the choice of $\bar{f}$.
\end{proof}
We will see below that \eqref{eq:free-group-polynomial-functor} is in fact a polynomial functor on $\mathbf{grp}$ and that the collision filtration coincides with the polynomial filtration. Our eventual goal is to work towards computing its composition factors (see \S\ref{sec:polynomiality}).

\begin{rmk}[\textbf{Extensions of the $\Out$-action}]\label{rmk:extensions}
   Since rational polynomial representations of $\GL_g(\mathbb{Z})$ admit no non-trivial extensions, one might expect that the extension problem for the collision filtration is trivial and thus that $H^*_c( F(\vee_{i=1}^g S^1,n) ; \mathbb{Q})$ is isomorphic to $H^{CE}_*(H^*(\vee_{i=1}^g S^1)\otimes \Suspension \Lie)$ as representations.
   
   However, this is not the case: the $\Out(F_g)$-representation $H^*_c( F(\vee_{i=1}^g S^1,n) ; \mathbb{Q})$ does not factor through $\GL_g(\mathbb{Z})$, and the collision filtration exhibits enormously complicated nontrivial extensions -- see \cite[\S2.3]{TW19} and \cite[Theorem 13]{PV18}.
\end{rmk}

\subsection{Equivariant CW-structure}\label{section:2-step}
We now bring in a completely orthogonal approach to computing $H^*_c(F(X,n);\Q)$ when $X$ is a wedge of circles: using an $\SymGroup[n]$-equivariant CW-structure on the one-point compactifiction $F(X,n)^+$.
This was introduced and featured as the central computational tool in the first author's recent work \cite{BCGY21}.

The resulting cellular cochain complex consists of only two nontrivial chain groups, each of which is free as an $\SymGroup[n]$-representation. The linear dual of this $2$-step complex also featured in Powell--Vespa \cite{PV18}, though the vast generality of their framework needs some unpacking to make it amenable to computations.

\begin{prop}[{\cite[Theorem 1.2]{BCGY21}}]\label{prop:2-step}
    Let $X= \bigvee_{i=1}^g S^1$ be a wedge of $g$ circles. The $\SymGroup$-representation $H^*_c(F(X,n))$ is computed by a $2$-step complex of free $\SymGroup$-modules
    \begin{equation} \label{eq:2-step integral}
        \ldots \to 0\to \underbrace{\Q[\SymGroup]^{\oplus\binom{n+g-2}{g-1}}}_{\text{cohomological degree }n-1} \to \underbrace{\Q[\SymGroup]^{\oplus\binom{n+g-1}{g-1}}}_{\text{cohomological degree }n} \to 0 \to \ldots
    \end{equation}
\end{prop}

As in the previous section, this statement holds equally well for integral cohomology, but we will not need that here.
\begin{proof}
    First note that $H^*_c(F(X,n))$ coincides with the ordinary reduced cohomology of the one-point compactification $F(X,n)^+$. Then \cite{BCGY21} gives an $\SymGroup$-equivariant CW-structure on this compactification, with cells in dimensions $n-1$ and $n$ only and a free $\SymGroup$-action. The resulting cellular cochain complex is as claimed.
\end{proof}

\cite{BCGY21} also gives explicit formulas for the coboundary map and for the actions of endomorphisms of $X$ on this complex. We do use these in explicit calculations, but their details are not important for our discussion.

What is important is that the complex splits into its isotypic components. That is, for every irreducible representation $\Specht{\lambda}$ of $\SymGroup$, the multiplicity space $ \Specht{\lambda}\otimes_{\SymGroup} H^*_c(F(X,n))$ is also computed by the same complex
\begin{equation} \label{eq:2-step isotypic}
    \ldots\to 0 \to \Specht{\lambda}^{\oplus\binom{n+g-2}{g-1}} \to \Specht{\lambda}^{\oplus\binom{n+g-1}{g-1}} \to 0 \to \ldots
\end{equation}
where one only needs to specialize the coboundary map to $\Specht{\lambda}$.
This small complex efficiently computes the multiplicity of $\Specht{\lambda}$ in $H^*_c(F(X,n))$ for various values of $n$ and $g$. Its downside, however, is that the collision filtration is not as readily accessible as in the Chevalley--Eilenberg complex of \S\ref{section:CE-complex}. The following discussion explains how to `see' the collision filtration on the terms of \eqref{eq:2-step isotypic}.

\begin{rmk}[\textbf{$\End(F_g)$-action}]

Recall that upon taking the associated graded of the collision filtration on $H^*_c(F(X,n);\Q)$, the $\End(F_g)$-action discussed in \S\ref{sec:out-action} factors through the matrix ring $\End(\mathbb{Z}^g)$ (see Corollary \ref{prop:GL action after grading}). Furthermore, on the associated graded, the filtration degree is exhibited by the weights of the action of the diagonal matrices $\mathbb{Z}^g \subset \End(\mathbb{Z}^g)$. However, these weights can already be read-off from simple chain-level endomorphisms of the 2-step complex \eqref{eq:2-step isotypic}.
\end{rmk}

\begin{lem}
Let $X= \bigvee_{i=1}^g S^1$ be a wedge of $g$ circles and let $\Specht{\lambda}$ be an irreducible representation of $\SymGroup$. The action of the diagonal matrix $\operatorname{diag}(d_1,\ldots,d_g) \in M_g(\mathbb{Z})$ on the $\Specht{\lambda}$-multiplicity space of $\gr H^*_c(F(X,n);\Q)$ is realized by the (non-invertible) space-level map $\varphi: X\to X$ that for every $1\leq i \leq g$ wraps the $i$-th circle around itself with degree $d_i$.

This induces an operator on the 2-step complex \eqref{eq:2-step isotypic} that preserves each $\Specht{\lambda}$-summand, and is thus an effectively computable block-diagonal transformation whose eigenspaces determine the collision filtration.


\end{lem}

\begin{proof}
    The transformation $\varphi$ acts on the homology of $X$ by the diagonal matrix $\operatorname{diag}(d_1,\ldots,d_g)$.
    Since the induced map on $F(X,n)^+$ is equivariant and cellular, it further induces an endomorphism of the complex \eqref{eq:2-step isotypic}.
    
    The result of wrapping the $i$-th circle around itself is realized as a sum of shuffles of the points lying on that circle. But regardless of their order, the number of points on each circle is preserved by the operation. These numbers of points are the invariants that differentiate the $\SymGroup$-orbits of cells in $F(X,n)^+$, hence the summands of \eqref{eq:2-step isotypic} are preserved. In particular, the chain operator $\varphi_*$ can be diagonalized within every summand.
    
    Lastly, since a scalar matrix $\operatorname{diag}(n,\ldots,n)$ acts on $H^1(X)$ by the scalar $n$, will act on the subspace of collision filtration degree $\leq p$ by eigenvalues $n^d$ with $n-p \leq d \leq n$. In this way the eigenspaces of $\varphi$ detect the collision filtration at the chain level.
\end{proof}

\begin{rmk}
 Varying the diagonal entries $(d_1,\ldots,d_g)$ in the above lemma, the trace by which $\varphi$ acts on \eqref{eq:2-step isotypic} characterizes the $\GL_g(\Z)$-action on $\gr H^*_c(F(X,n))$ completely, which is the ultimage goal of this project -- see \S\ref{sec:polynomiality} below.
 
 Furthermore, the block-diagonal structure of the action of $\varphi$ on \eqref{eq:2-step isotypic} is effective for computer calculations: it can be diagonalized on every $\Specht{\lambda}$ summand individually, thus allowing for parallel calculations on relatively small matrices.

\end{rmk}

\subsection{Bead representations}\label{sec:bead}

Let us recall where the $\SymGroup\times \Out(F_g)$-representations $H^*_c(F(\bigvee_{i=1}^g S^1,n))$ previously appeared in the literature and establish a dictionary with related work. Turchin and Willwacher studied the Hochschild--Pirashvili cohomology of $\bigvee_{i=1}^g S^1$, and in \cite[Section 2.5]{TW19} they consider the $(1,\ldots,1)$-multigraded component, equipped with its $\SymGroup$-action -- by Proposition \ref{prop:geometry of HH} this is the same as our $H^*_c(F(\bigvee_{i=1}^g S^1,n))$. They then split the $\SymGroup$-action into isotypic components and call the resulting $\Out(F_g)$-representation \emph{bead representations}.
\begin{prop}\label{cor:bead}
For every $\lambda \vdash n$, Turchin--Willwacher's bead representations $U_\lambda^{I},U_\lambda^{II}$ are the $\Out(F_g)$-equivariant multiplicity of the Specht module $\Specht{\lambda}$ in $H^*_c(F(\bigvee_{i=1}^g S^1, n))$ for $* = n$ and $n-1$ respectively. That is,
\begin{equation}
    U_\lambda^*(g) = H^*_c(F(\vee_{i=1}^g S^1, n))\otimes_{\SymGroup} \Specht{\lambda}.
\end{equation}
\end{prop}
They pose a (still open) problem to describe these representations, starting with computing the decomposition of $\gr H^*_c(F(\vee_{i=1}^g S^1, n))$ into Schur functors. This latter point is exactly the subject of this paper, e.g. they are the rows of Table \ref{table:n=7 colors}. The decomposition of all $U_{\lambda}^{I}$ and $U_{\lambda}^{II}$ with $|\lambda| = n\leq 10$ can be found on \href{https://louishainaut.github.io/GH-ConfSpace/}{this webpage}\footnote{\url{https://louishainaut.github.io/GH-ConfSpace/}}.

\begin{rmk}[Extensions of Schur functors]
The bead representations turn out to be central to the theory of outer polynomial functors (see Definition \ref{def:outer-functor}). Indeed, recall that the irreducible polynomial $\GL_g(\Z)$-representations are given by Schur functors $\Schur{\lambda}(\Q^g)$, which admit no non-trivial extensions. On the other hand, every Schur functor can be pulled back to $\Out(F_g)$, and as such they do admit extensions, i.e. there exist non-split surjections of $\Out(F_g)$-representations $E \onto \iota^*\Schur{\lambda}(\Q^g)$. Amazingly,
Powell--Vespa give a canonical surjection $H^n_c(F(X,n)) \otimes_{\SymGroup} \Specht{\TransposePartition{\lambda}} \onto \iota^* \Schur{\lambda}$ and prove that 
it is the \emph{maximal indecomposable extension} of outer polynomial functors, i.e. a projective cover (see \cite[Theorem 19.1]{PV18} for the dual statement). Moreover, every outer polynomial functor
admits a minimal projective resolution by sums of bead representations.

Decomposing the bead representations into their composition factors is thus a fundamental task in the representation theory of $\Out(F_g)$, serving as further motivation for our present study.
\end{rmk}

Powell and Vespa prove a multitude of facts about these functors in \cite{PV18}, and we think it valuable to make their results accessible to the topologically minded reader. We will therefore devote the rest of this section bridging the terminology gap between their setup and what we consider to be more natural in the topological context.

\begin{itemize}

    \item Powell--Vespa study the Hochschild--Pirashvili homology of $X$ that we discuss in \S\ref{sec:hochschild}. In \eqref{eq:HH lower star} we highlight its relation to homology of compactified configuration spaces as a functor in both a space $X$ and a vector space $V$. Powell--Vespa denote this bi-functor by $(X, V) \mapsto HH_*(X, \Loday{A_V}{A_V})$, which is isomorphic to our $HH_*(X_+, A_V)$ in \eqref{eq:HH lower star}. This is an analytic functor in $V$, and for fixed $X$ our Proposition \ref{prop:geometry of HH} identifies the corresponding symmetric sequence of coefficients as the Borel--Moore homology $n\mapsto H_*^{BM}(F(X, n))$, linearly dual to $H_c^*(F(X, n))$. In Powell--Vespa's notation, this symmetric sequence of coefficients is the functor $HH_*(X;\vartheta^*\operatorname{Inj}^{\textbf{Fin}})$.
    
    \item They get a functor from the category of free groups by composing with the classifying space $B(-):F_g\mapsto B(F_g)$ in place of the space $X$. Since $B(F_g)$ is coherently homotopy equivalent to the wedge $\vee_{i=1}^g S^1$, their functors agree with the ones we study here.
    
    \item 
    They show that these Hochschild homology groups form an outer polynomial functor on the category of free groups, and so over a field of characteristic $0$ all their composition factors are of the form $\iota^*\Schur{\lambda}$ for various partitions $\lambda$. They denote these functors by $\alpha S_\lambda$. Calculating their multiplicities is the subject of our work.
    
    The associated graded representation we call $\gr H_c^*(F(\vee_{i=1}^g S^1, n))$ corresponds in their formalism to $\alpha_n \operatorname{cr}_n HH_*(B(-); \vartheta^*\operatorname{Inj}^{\textbf{Fin}})$, evaluated at $F_g$. Here $\operatorname{cr}_n$ is the functor that extracts the $n$-th coefficient of a polynomial functor, returning an $\SymGroup$-representation, and $\alpha_n$ converts this representation into a sum of Schur functors.
    
    \item For an integer partition $\lambda\vdash n$, a subscript $\lambda$ either on Hochschild homology $HH_*(-;\vartheta^*\operatorname{Inj}^{\textbf{Fin}})$ or on the coefficients $\vartheta^*\operatorname{Inj}^{\textbf{Fin}}$ refers to the $\Specht{\lambda}$-multiplicity space of the corresponding $\SymGroup$-representation.
    
    In particular, our coefficients $\Phi[\lambda, m]$ for a partition $\lambda$ and $m\leq |\lambda|$ is the $\SymGroup[m]$-representation $ \operatorname{cr}_m HH_*(B(-); \vartheta^*\operatorname{Inj}^{\textbf{Fin}}_{\lambda})$, with $\Phi[\lambda, \mu]$ giving the multiplicity of $\alpha S_{\TransposePartition{\mu}}$ (note that the partition $\mu$ needs to be transposed).
    
    \item As they discuss in \cite[\S 16.3]{PV18}, for $\lambda\vdash n$ a partition of $n$, the bead representations $U_{\lambda}^{I}$ and $U_{\lambda}^{II}$, mentioned at the beginning of \S\ref{sec:bead}, are dual to $HH_n(B(-); \vartheta^*\operatorname{Inj}^{\textbf{Fin}}_{\lambda})$ and $HH_{n-1}(B(-); \vartheta^*\operatorname{Inj}^{\textbf{Fin}}_{\lambda})$ respectively. Powell--Vespa refer to these by $\omega\beta_n(S_{\lambda})$ and $(\operatorname{Coker}_{\overline{\operatorname{ad}}}\mathbb{P}^{\Sigma}_{\operatorname{coalg}})_{\lambda}$ respectively.

    
\end{itemize}

\begin{expl}[$(2,1^{n-2})$ bead representation] \label{ex:(2,1)-bead rep}
To illustrate the translation from Powell--Vespa's formalism, consider \cite[Example 4]{PV18}:
\begin{equation}
    HH_*(B(-);\vartheta^*\operatorname{Inj}^{\textbf{Fin}}_{(2,1^{n-2})}) \cong \begin{cases}
    \alpha S_{(n-1,1)} & *=n \\
    0 & \text{otherwise},
    \end{cases}
\end{equation}
the $*=n$ case is also denoted $\omega \beta_{n} S_{(2,1^{n-2})}$.

We reinterpret this line as stating that for $X$ a wedge of circles, the cohomology $H^{*}_c(F(X,n))$ has as its $\Specht{(2,1^{n-2})}$-multiplicity space isomorphic to the Schur functor $\iota^*\Schur{\TransposePartition{(n-1,1)}}(\tilde{H}^1(X))$ when $*=n$, and otherwise vanishes. In terms of the coefficients $\Phi[-,-]$, this means that the graded vector space $\Phi[(2,1^{n-2}), (2,1^{n-2})]$ has rank $1$ concentrated in degree $0$, and that for every other partition $\lambda\neq (2,1^{n-2})$ the graded vector space $\Phi[(2,1^{n-2}), \lambda]$ is trivial.
\end{expl}

For completeness of the dictionary, we include:
\begin{itemize}
    \item In their calculations Powell--Vespa frequently use the functors $\beta_d S_\lambda$, which in our terminology are linear dual to the top cochains in the $2$-step complex \eqref{eq:2-step isotypic} equipped with $\Aut(F_g)$-actions . These also assemble to a polynomial functor on free groups, but it does not factor through $\Out(F_g)$, i.e. conjugations act nontrivially as noted in \cite[Remark 2.12]{BCGY21}.
    
\end{itemize}

\section{Polynomiality and consequences} \label{sec:polynomiality}

Let us next consider configurations on $X$  for $X$ a wedge of spheres. As we let the number of spheres vary, the compactly supported cohomology acquires the structure of a polynomial functor evaluated on the vector space $\tilde{H}^*(X)$, as shown in this section.

\begin{rmk}\label{rmk:using geometry}
Having a single polynomial functor compute the cohomology of $F(X,n)$ for any finite wedge of spheres constrains the functor and endows it with further structure that one would not have expected. For example consider the following three cases:
\begin{enumerate}
\item For configurations on a wedge of $1$-spheres, the cohomology is closely related to the algebraic construction of Hochschild--Pirashvili homology as an exponential functor, see \cite{PV18}. On the other hand, the configuration spaces admit a Fox-Neuwirth cell decomposition\footnote{This is a decomposition into locally closed sets that become cells of the one-point compactification.} with cells in only the top two dimensions, freely permuted by the symmetric group. This gives a free presentation of the cohomology as an $\SymGroup$-module, leading to rather efficient calculations -- details in \S\ref{section:2-step}.


    \item For configurations on a wedge of $2$-spheres, the ambient space $X= \bigvee_{i=1}^g S^2$ can be realized as a complex algebraic curve of genus $g$. The cohomology in question thus admits a mixed Hodge structure, and an action by the Galois group of $\Q$.
    Moreover, for a single sphere $F(\mathbb{C}P^1,n) \cong PSL_2(\C)\times \ModuliCurve{0,n}$, which explains the appearance of these moduli spaces in our calculations below.
    
    \item For configurations on a wedge of $3$-spheres, since $S^3\cong SU(2)$ is a Lie group, it follows that $F(S^3,n) \cong S^3 \times F(\R^3,n-1)$. This lets us identify the equivariant multiplicity of all exterior powers. 
\end{enumerate}

\end{rmk}

\subsection{Polynomiality for wedges of spheres}

The polynomiality statement of Theorem \ref{thm:polynomiality} will follow from the more refined main result of this section. Let $\Sigma$ denote the category of finite sets and bijections\footnote{In the representation stability literature this category is called $\textbf{FB}$.}, with skeleton the finite sets $\n = \{1,2,\ldots, n\}$.

\begin{thm}\label{thm:graded polynomiality}
    Let $X$ be a finite wedge of spheres, possibly of different dimensions, and consider the collision filtration on $H_c^*(F(X, n))$ (defined in \ref{terminology:collision}). Its associated graded quotients
    \[
        \gr_p^F H_c^{p+*}(F(X, n)) = F_p H_c^{p+*}(F(X, n))/F_{p-1} H_c^{p+*}(F(X, n))
    \]
    admit the following algebraic description.
    
    There exists a functor in two variables $\Psi^p\colon \Sigma \times \GradedVect_{\Q} \to \GradedVect_{\Q}$ such that $\Psi^p(\n,-)$ is a polynomial functor of degree $n-p$
    with a natural $\SymGroup$-action, and such that for any finite wedge of spheres $X = \bigvee_{i\in I}{S^{d_i}}$ there is a natural isomorphism of graded $\SymGroup$-representations
    \[
        \gr_p^F H_c^{p+\ast}(F(X, n)) = \Psi^p(\n,\tilde{H}^*(X)).
    \]
\end{thm}

\begin{note}
Let us reiterate that while the two inputs of $\Psi^p$ are different types of objects, the equivalence of categories mentioned in Proposition \ref{prop:equiv symmetric sequence analytic functors} relates such functors and bi-functors $\Sigma\times\Sigma \to \GradedVect_\Q$ as well as $\GradedVect_\Q\times \GradedVect_\Q \to \GradedVect_\Q$ taking inputs of the same type. We prefer the presentation of $\Psi^p$ given here as it lends itself well to a simple geometric interpretation.
\end{note}

A key step in the proof of this theorem is the following lemma. Recall that we defined $A_W$ to be the square-zero algebra $\Q\oplus W$, with trivial multiplication on $W$. Furthermore, the CE-complex $C^{CE}_*(A_W\otimes \Suspension\Lie)(n)$ is bigraded and its filtration by columns is denoted $\mathfrak{F}_{\bullet}$. Thus the graded quotient $\gr^{\mathfrak{F}}_{-p}$ is exactly the $p$-th column of the bicomplex.

\begin{lem}
Let $W$ be a graded vector space and fix degree $p\geq 0$. Then for all $n\in \N$ the functor sending $W$ to the arity $n$ term of the Chevalley--Eilenberg homology $\gr_{-p}^{\mathfrak{F}} H^{CE}_{\ast}(A_W\otimes\Suspension\Lie)(n)$ is a polynomial functor of degree $p$ taking values in graded vector spaces. 

\end{lem}

\begin{proof}
First, the construction $W\mapsto A_W\otimes \Suspension\Lie$ produces a twisted dg Lie algebra that in every arity is a polynomial functor of degree $1$. Clearly this construction is functorial in $W$, i.e. linear maps $W\to W'$ induce morphisms of twisted dg Lie algebras. 

Second, for a Lie algebra $\mathfrak{g}$ in symmetric sequences, the $p$-th column of the CE-bicomplex is defined as a natural quotient of $\mathfrak{g}^{\otimes p}$. In every arity $n$ this expression is given up to degree shifts by 
\[
\mathfrak{g}^{\otimes p} = \bigoplus_{B_1\coprod \ldots \coprod B_p = \n} \mathfrak{g}[B_1] \otimes \ldots \otimes \mathfrak{g}[B_p],
\]
that is, a tensor product of exactly $p$ terms from $\mathfrak{g}$. Thus since every one of the $\mathfrak{g}[k]$'s is a polynomial functor of degree $1$, then the CE-complex in homological degree $p$ is polynomial of degree $p$.

Since the category of polynomial functors is abelian, it only remains to note that the CE-differentials respect the polynomial functor structure, i.e. that they are natural transformations in $W$. This follows from the general fact that the CE-complex $C_{-\ast}^{CE}(L)$ is functorial in its input, the Lie algebra $L$. 

\end{proof}

\begin{proof}[Proof of Theorem \ref{thm:graded polynomiality}]
Recall that in Section \ref{section:CE-complex} we show that the associated graded of the collision filtration on $H^*_c(F(X,n))$ is computed, after suitable regrading, by the Chevalley--Eilenberg (CE) homology of a twisted Lie algebra, naturally in $X$. Explicitly, there is an isomorphism of graded $\SymGroup$-representations
\begin{equation}
    \gr_p^F H^{*}_c(F(X,n)) \cong \gr_{p-n}^{\mathfrak{F}} H^{CE}_{\ast}( H^*(X)\otimes \Suspension\Lie (n)).
\end{equation}

Thus the polynomiality statement reduces to one about this CE-homology.

Furthermore, the cohomology algebra of a wedge of spheres is the square-zero algebra $H^*(X) \cong \Q1\oplus \tilde{H}^*(X)$. Therefore the claim of Theorem \ref{thm:graded polynomiality} follows from the last lemma.
\end{proof}

The polynomiality result thus proved has many consequences for the cohomology $H^*_c(F(X,n))$, but the converse also holds: we next use the cell structure on $F(X,n)^+$ from \S\ref{section:2-step} to constrain the nonzero `coefficients' of the polynomial functor and bound its degree.

\begin{prop}\label{prop:vanishing coefficients}
    For every $n\in \N$ the polynomial functors $\Psi^p(\n,-)\colon \GradedVect_{\Q}\to\GradedVect_{\Q}$ from the previous theorem have the following properties:
    \begin{enumerate}
        \item $\Psi^p(\n,-) = 0$ for all $p > n-1$.
        \item $\Psi^p(\n,W)$ decomposes as
        \[
            \Psi^p(\n,W) \cong \bigoplus_{m=n-p-1}^{n-p} \Phi^p[n,m]\otimes_{\SymGroup[m]} W^{\otimes m}
        \]
        for some $\SymGroup[n]\times\SymGroup[m]$-representations $\Phi^p[n,m]$. In particular the polynomial functor $\Psi^p(\n,-)$ has degree $n-p$ and only two homogeneous terms.
        
    \setcounter{continue-counter}{\value{enumi}}
\end{enumerate}
In the `leading' special case $p=0$, the two nontrivial terms $\Phi^0[n,n]$ and $\Phi^0[n,n-1]$ are given as follows.
\begin{enumerate}
\setcounter{enumi}{\value{continue-counter}}
        \item The coefficient $\Phi^0[n,n]$ is the "diagonal" representation
        \[
            \Phi^0[n,n] = \bigoplus_{\lambda\vdash n}{\Specht{\lambda}\boxtimes\Specht{\lambda}}.
        \]
        \item The coefficient $\Phi^0[n,n-1]$ is
        \[
            \Phi^0[n,n-1] = \Specht{(1^n)}\boxtimes\Specht{(1^{n-1})}.
        \]
    \end{enumerate}
\end{prop}

We further note that the highest filtration terms not covered by this proposition, $\Psi^{n-1}[\n,-]$ and $\Psi^{n-2}[\n,-]$, are computed completely in \S\ref{sec:multiplicity symmetric powers}, as they consists only of functors of degree $\leq 2$ and are thus multiples of symmetric and alternating powers.
\begin{proof}
    The first statement is simply the claim that the associated graded quotient $\gr^p H_c^{p+*}(F(X, n))$ is trivial for $p<0$ and $p\geq n$, both cases being clear.
    
    For the remaining statements, we use Proposition \ref{prop:image-polynomial-concentrated} to deduce that it is enough to consider $X$ a wedge of $g$ circles with arbitrarily large $g$ to uniquely determine the coefficients $\Phi^p[n,m]$. Then, Proposition \ref{prop:2-step} shows that $H^i_c(F(X,n))=0$ unless $i\in \{n-1,n\}$, so its associated graded is similarly $0$.

    On the other hand, set $W := \tilde{H}^*(X)$ and consider the graded vector space
    \[
    \gr_{p-n}^{\mathfrak{F}} H^{CE}_{\ast}(A_W\otimes\Suspension\Lie)(n) = \bigoplus_{q\in \mathbb{Z}} E^{p,q}
    \]
    with $E^{p,q}$ in grading $q$. Lemma \ref{lem:formal spaces} shows that $E^{p,q} \cong \gr_p^F H^{p+q}_c(F(X,n))$. In light of the previous paragraph it follows that $E^{p,q} = 0$ unless $p+q\in \{n-1,n\}$.
    
    But since $W$ is concentrated in grading $q=1$, the polynomial description of the CE-homology as $\oplus \Phi^p[n,m]\otimes_{\SymGroup[m]}W^{\otimes m}$ is graded such that its $m$-th summand is placed in grading $q=m$. It follows that 
    \begin{equation} \label{eq:vanishing of coefficients}
     \Phi^p[n,m]\otimes_{\SymGroup[m]}W^{\otimes m} = E^{p,m} =0
    \end{equation}
    unless $p+m\in \{n-1,n\}$. The second claim now follows since the above calculation is valid for every $g\geq 1$, and since for $g\geq m$ the vanishing in \eqref{eq:vanishing of coefficients} implies that $\Phi^p[n,m]=0$.
    
    
    The third statement follows from Proposition \ref{prop:spectral_sequence_unidim}: the $n$-th row with grading $q=n$ of the $E_1$-page has only one nonzero term $E_1^{0,n} = T^n(\Q^g)$, the $n$-th tensor power, with the $\SymGroup$-action by permutation of the tensor factors. Since this term can support no differentials, it survives to $E_\infty^{0,n}$ unchanged. Schur--Weyl duality gives a decomposition of $T^n(\Q^g)$ as an $\SymGroup$-equivariant polynomial functor, agreeing with the claimed expression for $\Phi^0[n,n]$.
    
    Our proof of the last statement is long and technical. We defer it to Section \ref{sec:conjectures patterns}.
\end{proof}

Proposition \ref{prop:vanishing coefficients} now implies Theorem \ref{thm:polynomiality}, since the associated graded $\gr H^*_c(F(X,n))$ is a sum $\oplus_p \Psi^p(\n,-)$ up to cohomological shifts, and for $p > n$ the functor $\Psi^p(\n,-)=0$ so only finitely many functors contribute for any given $n$.

\begin{rmk}
    In Theorem \ref{thm:polynomiality}, the coefficients are the graded vector spaces $\Phi[n,m]$, while in Proposition \ref{prop:vanishing coefficients} the coefficients are the (non-graded) vector spaces $\Phi^p[n,m]$. These two objects are related as suggested by the notation: $\Phi[n,m]$ decomposes as
    \[
        \Phi[n,m] = \bigoplus_{p=0}^m{\Phi^p[n,m]},
    \]
    with $\Phi^p[n,m]$ being the part of $\Phi[n,m]$ in degree $p$.
\end{rmk}

\begin{rmk}[Powell--Vespa's polynomiality result]
We bring to the reader's attention the fact that Powell and Vespa proved in \cite[Theorem 5]{PV18} a result that is extremely close to our Theorem \ref{thm:graded polynomiality}. They show that the Hochschild--Pirashvili homology of the classifying spaces $B(F_g)$ with square-zero coefficients form a polynomial functor from the category of finitely generated free groups to graded vector spaces -- see  \S\ref{sec:polynomiality-gr}.

Recalling that $B(F_g)\simeq \bigvee_g S^1$ a wedge of 1-spheres, and that Hochschild--Pirashvili homology with square-zero coefficients is dual to compactly supported cohomology of the configuration spaces (Theorem \ref{thm:hochschild and configurations}),
their polynomiality result generalizes our Theorem \ref{thm:polynomiality} for wedges of circles in that we only work at the associated graded level and with the classical notion of polynomial functors. However, their setup does not include wedges of spheres of higher dimensions as we consider here. 
\end{rmk}

For a first geometric consequence of the polynomiality result we give the following.
\begin{cor}\label{cor:vanishing terms in E2}
    Let $X$ be a finite wedge of spheres, all with the same dimension $d$. Considering the collision filtration on $H^i_c(F(X,n))$, the $p$-th graded quotient $\gr^F_p H^{i}_c(F(X,n))$ is nonzero only when $i= dn- (d-1)p$ or $i=d(n-1) - (d-1)p$.

    In particular, when $d\geq 3$ non-zero graded pieces $\gr^F_p H^{i}_c(F(X, n))$ appear in distinct cohomological degrees $i$, so there are no non-trivial extensions between the graded pieces.
\end{cor}
\begin{proof}
From Lemma \ref{prop:vanishing coefficients} the graded quotient $\gr^F_p H_c^{p+*}(F(X,n))$ is polynomial with coefficients $\Phi^p[n,m]\neq 0$ only for $p+m\in \{n-1,n\}$. Since $W:= \tilde{H}^*(X)$ is concentrated in grading $q=d$, the summand
\[
\Phi^p[n,m]\otimes_{\SymGroup[m]} W^{\otimes m}
\]
is concentrated in grading $q=dm$. It follows that $\gr^p H^{p+*}(F(X,n))$ is nontrivial only in grading $p+dm$ where $p+m\in \{n-1,n\}$. Setting $i=p+dm$ and substituting $m \in \{ n-1 - p,\ n-p \}$ gives the claim.

For the second part of the claim note that the equations $dn - (d-1)p = dn - (d-1)p'$ and $d(n-1) - (d-1)p = d(n-1) - (d-1)p'$ immediately give $p = p'$ when $d > 1$, while the equation $dn - (d-1)p = d(n-1) - (d-1)p'$ gives $d = (d-1)(p' - p)$, which has no integral solution for $d\geq 3$ since the last equation means that $d-1$ is a divisor of $d$, but in that case $d-1$ is greater than $1$ and coprime with $d$.
\end{proof}

We conclude this section by relating the collision filtration with more familiar filtrations defined on Hochschild--Pirashvili homology. Consider the contravariant functor from finitely generated free groups
    \begin{equation}
        F_g \mapsto H^*_c\left( F(\vee_{i=1}^g S^1,n) \right)
    \end{equation}
    discussed in \S\ref{sec:out-action}. Djament--Vespa \cite{DV15} define a filtration on functors of this sort, whose associated graded quotients factor through the abelianization $F_g\mapsto \Z^g$ and are polynomial in the classical sense. They call this the \emph{polynomial filtration}, and \cite[Theorem 17.8]{PV18} shows that it agrees with another natural filtration -- Pirashvili's so called \emph{Hodge} filtration on Hochschild--Pirashvili homology. Adapting the filtrations to the contravariant setting of interest here, we have the following.
    
\begin{cor}[Coincidence of filtrations]
    For $X$ a wedge of circles, the collision filtration on $H^*_c(F(X,n))$ coincides with the polynomial filtration of contravariant functors.
\end{cor}
\begin{proof}
The collision spectral sequence along with Theorem \ref{thm:graded polynomiality} give a natural isomorphism $\gr_p^F H^i_c(F(X,n)) \cong \oplus_m \Phi^p[n,m]\otimes_{\SymGroup[m]}\left[\tilde{H}^*(X)^{\otimes m}\right]^{i-p} $, where $[W]^{q}$ is the $q$-graded part of a graded vector space $W$. Indeed $\Phi^p[n,m]$ is concentrated in grading $p$, so $\tilde{H}^*(X)^{\otimes m}$ must contribute to grading $i-p$. Since $\tilde{H}^*(X)$ is concentrated in grading $1$, it means that only $m=i-p$ contributes non-trivially.

In other words, for every $i$ the functor $H^i_c(F(X,n))$ is filtered by the collision filtration, and the $p$-th graded factor is a homogeneous polynomial functor of degree $i-p$. The dual version for contravariant functors of \cite[Remark 6.10]{PV18} is stating exactly that such a filtration is unique, and it is the polynomial filtration. 
\end{proof}

\subsection{Schur functor multiplicity}
One can make sense of the polynomial functor structure of $\gr H^*_c(F(X,n))$ geometrically using the following fact. Let $[X,Y]$ denote the set of homotopy classes of maps between topological spaces $X$ and $Y$.

\begin{lem}
    Let $X$ and $Y$ be wedges of $g$ and $h$ spheres, respectively, all of equal dimension $d$. Then the operations induced by $[X,Y]$ on homology give a surjection onto the integer matrix space 
    \[
     \Hom_{\Z}(H_d(X;\Z),H_d(Y;\Z))\cong M_{g\times h}(\Z).
    \]
    In particular, $[X,X]$ surjects onto the matrix ring $M_g(\Z)$.
\end{lem}

\begin{proof}
    When dealing with wedges of circles, the claim follows since wedges of circles are classifying spaces of free groups and the abelianization map $\Hom_{\bm{grp}}(F_g, F_h)\to M_{g\times h}(\Z)$ is surjective. For general $d$ one only needs to observe that $X$ and $Y$ are the reduced suspensions of corresponding wedge of circles. The $(d-1)$-fold suspension of an appropriate basepoint preserving map between wedges of circles realizes any prescribed homological action.
\end{proof}

From Corollary \ref{cor:factoring through homological action} the functor $X\mapsto \gr H^*_c(F(X,n))$ factors through $H^*(X)$. This uniquely characterizes the polynomial structure on $\gr H^*_c(F(X,n))$, as explained next.

\begin{cor}[\textbf{Uniqueness}]
Let $\Psi'(\n,-)$ be any polynomial functor on graded vector spaces, such that its composition with reduced cohomology $X\mapsto \tilde{H}^*(X)$ admits a natural isomorphism
\[
\Psi'(\n,\tilde{H}^*(X)) \cong \gr^p H^{p+*}_c(F(X,n))
\]
as functors out of the full subcategory of finite wedges of spheres.  Then $\Psi'(\n,-)\cong \Psi^p(\n,-)$, the functor from Theorem \ref{thm:graded polynomiality}. In fact, it is already uniquely determined by its restriction to the subcategory of wedges of $d$-dimensional spheres for any fixed $d\geq 1$.

\end{cor}
\begin{proof}
    By naturality in $X$, there is a natural isomorphism $\Psi^p(\n,\tilde{H}^*(-))\cong \Psi'(\n,\tilde{H}^*(-))$. In particular, the two functors agree on every homomorphism $H^*(X)\to H^*(Y)$ induced by a map of spaces. But since the previous lemma shows that these include all integrally defined homomorphisms, and since the integer lattice is Zariski dense in the space of all linear maps, the two polynomial functors must agree on all linear maps.
\end{proof}

A structural consequence of the polynomiality of $\gr H^*_c(F(X,n))$ is that it factors into Schur functors, as discussed in \S\ref{sec:prelim polynomial functors}. Explicitly,
\begin{equation}\label{eq:schur decomposition}
    \gr^p H^{p+*}_c(F(X,n)) \cong \bigoplus_{\lambda} \Phi^p[n,\lambda]\boxtimes \Schur{\lambda}\left( \tilde{H}^*(X) \right)
\end{equation}
where $\Phi^p[n,\lambda]$ is some $\SymGroup$-representation and $\Schur{\lambda}(-)$ is a Schur functor. Note that $\Phi[n,\lambda]$ is the $\Specht{\lambda}$-multiplicity space in the $\SymGroup[m]$-representation $\Phi[n,m]$ from the introduction.

\begin{defi}
The $\SymGroup$-representation $\Phi^p[n,\lambda]$ appearing in \eqref{eq:schur decomposition} is the \emph{equivariant multiplicity} of the Schur functor $\Schur{\lambda}$ in $\gr^p H^{p+*}_c(F(X,n))$.
\end{defi}

Understanding the cohomology $\gr H^*_c(F(X,n))$ amounts to describing these equivariant multiplicities, e.g. giving their characters for all $p$.

\begin{observation}[\textbf{Genus bound principle}]\label{lem:genus bound}
    Let $\lambda\vdash m$ be a partition with $\ell$ parts. Then given $X = \bigvee_{i=1}^g S^d$, the $\SymGroup$-equivariant multiplicity of the Schur functor $\Schur{\lambda}(\tilde{H}(X))$ 
    under the natural action by the monoid of homotopy classes $[X,X]$ on $H_c^{*}(F(X, n))$ is independent of $g$ once $g\geq \ell$. In particular, this multiplicity could be read off from $X$ a wedge of exactly $\ell$ spheres.
\end{observation} 

\begin{proof}
    Let $\lambda$ be a partition with $\ell$ parts, that is $\lambda = (\lambda_1, \lambda_2,\ldots,\lambda_{\ell})$, and consider $X$ a wedge of $g\geq \ell$ spheres of dimension $d$. 
    
    
    Consider the surjection $[X,X] \onto \End(\Z^g)$ given by the homological action. The Schur functors $\Schur{\mu}(\Q^{g})$ for partitions $\mu$ with $\leq g$ parts are distinct nonzero irreducible representations of $\End(\Z^{g})$, differentiated e.g. by their characters on diagonal matrices: these are evaluations of the respective Schur polynomial $s_\mu(x_1,\ldots,x_g)$ at the diagonal entries. In particular, $\Schur{\lambda}(\tilde{H}^*(X))$ is a nonzero irreducible representation of $[X,X]$, distinct from all other Schur functors appearing in Decomposition \eqref{eq:schur decomposition}. Thus by Schur's lemma, there is an $\SymGroup$-equivariant isomorphism
    \begin{equation}
        \Phi^p[n,\lambda] \cong \Hom_{[X,X]}\left( \Schur{\lambda}(\tilde{H}^*(X)), \gr^p H^{p+*}_c(F(X,n)) \right)
    \end{equation}
    determining the equivariant multiplicity independently of $g$.
\end{proof}

Swapping even spheres for odd spheres has the effect of conjugating the partition up to grading shifts
\begin{equation}
    \Schur{\lambda}(\tilde{H}^*(X)) \cong \begin{cases}
    \Schur{\lambda}(\Q^g) & \text{ if } X = \bigvee_{i=1}^g S^{2d} \\
    \Schur{\TransposePartition{\lambda}}(\Q^g) & \text{ if } X = \bigvee_{i=1}^g S^{2d+1}.
    \end{cases}
\end{equation}
Thus the notion of `simple' partitions, detectable by wedges of few spheres, includes partitions $\lambda$ such that either $\lambda$ or its conjugate $\TransposePartition{\lambda}$ have few parts.

\begin{cor}\label{cor:equivariant multiplicities}
Let $X$ be a wedge of $g$ spheres of dimension $d$ and fix a partition $\lambda$. Then under the $M_g(\Z)$-action induced by $[X,X]$ on homology, the $\SymGroup$-equivariant multiplicity $\Hom_{M_g(\Z)}\left( \Schur{\lambda}(\Q^g), \gr^p H^{p+*}_c(F(X,n)) \right)$ is isomorphic to
\begin{equation}
    \begin{cases}
    \Phi^p[n,\lambda] & d\text{ even and  $\lambda$ has $\leq g$ parts} \\
    \Phi^p[n,\TransposePartition{\lambda}] & d\text{ odd and  $\TransposePartition{\lambda}$ has $\leq g$ parts} \\
    0 & \text{otherwise}.
    \end{cases}
\end{equation}
\end{cor}

\begin{rmk}
Proposition \ref{prop:vanishing coefficients} above showed that $\Phi^p[n,\lambda] \neq 0$ only if $p+|\lambda| \in \{n-1,n\}$. In other words, the only equivariant multiplicities to compute are the ones with $p=n-1-|\lambda|$ and $n-|\lambda|$.
\end{rmk}

The first examples of `simple' partitions are $\lambda = (m)$ and $(1^m)$. We utilize the principle thus outlined in the next section.

\subsection{Symmetric and alternating powers}\label{sec:multiplicity symmetric powers}
We now compute the $\SymGroup$-equivariant multiplicity of $\Sym^m(-)$ and $\Lambda^m(-)$ occurring in $\gr^p H_c^{p+*}(F(-,n))$. The determination are given in terms of other geometric objects whose homology is well-understood, and they will prove Theorem \ref{thm:symmetric and alternating powers}.

Since symmetric and alternating powers are $\Schur{\lambda}$ for partition $\lambda$ with only one row or column, their equivariant multiplicity in $\gr H^*_c(F(X,n))$ are determined by configurations on a single sphere. These spaces have been studied extensively, most notably \cite{FZ00} computed the integral cohomology rings of $F(S^d,n)$ for all $(d,n)$.

Our calculations for $\Sym^m$ and $\Lambda^m$ follow a very similar pattern: projecting configurations on a sphere to a moduli space of such configurations. Let us begin with the simpler case of alternating powers.

\begin{prop}[\textbf{Alternating powers}]\label{prop:alternating powers}
    For the partition $\TransposePartition{\lambda} = (1^m)$ the Schur functor $\Schur{\TransposePartition{\lambda}} = \Lambda^m$. The equivariant multiplicity $\Phi^p[n,(1^m)]$ of $\Lambda^m$ in the functor $\Psi^p(\n,-)$ is
\begin{equation}
    \cong
    \begin{cases}
    H_{2(n-m)}(F(\R^3,n-1))\otimes \sgn & \text{if }p=n-m \\
    H_{2(n-m-1)}(F(\R^3,n-1))\otimes \sgn & \text{if }p=n-m-1 \\
    0 & \text{otherwise},
    \end{cases}
\end{equation}

where $\SymGroup$ acts on $F(\R^3,n-1)$ via the identification with $F(SU(2),n)/SU(2)$.
\end{prop}
These $\SymGroup$-representations are the Whitehouse modules up to sign, see \cite{ER19}.
\begin{proof}
    Let $\lambda = (m)$ so that $\TransposePartition{\lambda}=(1^m)$ and consider configurations on $S^3$. Since $\lambda$ has only one part, Corollary \ref{cor:equivariant multiplicities} identifies $\Phi^p[n,\TransposePartition{\lambda}]$ as the equivariant multiplicity of $\Sym^m(\Q)$ in $\gr^p H^{p+*}_c(F(S^3,n))$ as a representation of $\pi_0(\EndSpace{S^3})\cong \mathbb{Z}$. Here $a\in \mathbb{Z}$ acts on $\Sym^m(\Q)$ as multiplication by $a^m$. Thus we proceed by computing $H^*_c(F(S^3,n))$.
    
    First, since $F(S^3,n)$ is a manifold, Poincar\'{e} duality gives an $\SymGroup$-equivariant isomorphism
    \begin{equation}
        H^{3n-i}_c(F(S^3,n))\otimes \sgn \cong H_{i}(F(S^3,n))
    \end{equation}
    where the additional sign comes from the induced $\SymGroup$-action on the orientation bundle of $(S^3)^n$. 
    
    The cohomology of $F(S^3,n)$ was computed in \cite{FZ00}, but we seek a different description. Thinking of $S^3$ as the group $SU(2)$, the quotient by the diagonal action
    \begin{equation}
        F(SU(2),n) \to F(SU(2),n)/SU(2)
    \end{equation}
    is a trivial $SU(2)$-principal bundle. This is furthermore an $\SymGroup$-equivariant map. Since $SU(2)$ has homology only in degrees $0$ and $3$, there is an $\SymGroup$-equivariant isomorphism\footnote{For $\SymGroup$-equivariance one also needs to observe that the $\SymGroup$-action on the $SU(2)$ fibers is homologically trivial. Indeed, one can check that the $\SymGroup$-action commutes with the left $SU(2)$ action. Thus it must act on every copy of $SU(2)$ as right multiplication by a fixed matrix. But since $SU(2)$ is connected, such multiplication is homotopically trivial.}
    \begin{align}\label{eq:splitting configs on S^3}
    \begin{split}
        H_i(F(S^3,n)) \quad \cong \quad & H_i(F(SU(2),n)/SU(2)) \\ & \oplus  H_{i-3}(F(SU(2),n)/SU(2)).
    \end{split}
    \end{align}

    Via the identification $SU(2)\setminus \{1\} \cong S^3\setminus \{N\} \cong \R^3$ one gets a homeomorphism $F(SU(2),n)/SU(2) \cong F(\R^3,n-1)$ by mapping
    \[
    (x_1,\ldots,x_n) \mapsto (x_1^{-1}x_2, \ldots, x_1^{-1}x_n).
    \]
    It is furthermore well-known that, for $n\geq 2$, $H_*(F(\R^3,n-1))$ is concentrated in grading $2k$ for $0\leq k \leq n-2$ (see \cite[Lemma 6.2]{Cohen76}). Therefore the homology $H_*(F(S^3,n)))$ coincides with $H_*(F(\R^3,n-1))$ in even degrees, and with $H_{*-3}(F(\R^3,n-1))$ in odd ones. Under Poincar\'{e} duality these become $H^{3n-*}_c$ in even and odd codimension respectively.
    
    Let us match \eqref{eq:splitting configs on S^3} with the collision filtration. By Proposition \ref{prop:vanishing coefficients} the only nontrivial terms in $\gr^p H^{p+*}_c(F(S^3,n))$ are those in grading $*=3q$ where $p+q\in \{n-1,n\}$. If $p+q=n$ then the term in degree $p+3q$ has even codimension $3n-(p+3q)=2(n-q)$, and similarly $p+q=n-1$ implies odd codimension $2(n-q)+1$. Thus by parity of dimensions the collision filtration has no extensions:
    \begin{equation}
        H^{3n-i}_c(F(S^3,n)) = \begin{cases}
        \gr^{k} H^{3n-2k}_c &\text{if $i=2k$ even} \\
        \gr^{k-1} H^{3n-2k-1}_c &\text{if $i=2k+1$ odd}
    \end{cases}
    \end{equation}
    
    Now recall that $\gr^p H^{p+*}_c(F(S^3,n)) \cong \oplus_m \Phi^p[n,(1^m)]\boxtimes \Sym^m(\Q)$ where the $m$-th symmetric power has grading $*=3m$. Considering the case of even codimension first, with $k=p$ and $3m = * = 3n-3p$ we have an equivariant isomorphism
    \[
    \gr^p H_c^{3n-2p} = H_c^{3n-2p} \cong H_{2p}(F(\R^3,n-1))\otimes \sgn
    \]
    which equivariantly identifies the degree $m=n-p$ summand
    \begin{equation}
    \Phi^p[n,(1^{n-p})]\boxtimes \Sym^{n-p}(\Q) \cong H_{2p}(F(\R^3,n-1))\otimes \sgn.
    \end{equation}
    Considered as $\SymGroup$-representations, the term $\Sym^{n-p}(\Q)$ is a trivial $1$-dimensional representation, and thus the first case of the proposition follows.   
    
    The odd codimension case is similar: taking $p=k$ and grading $3m = * = 3n-3p$ we have
    \[
    \gr^{p-1} H^{3n-2p-1}_c = H^{3n-2p-1}_c \cong H_{2(p-1)}(F(\R^3,n-1))\otimes\sgn.
    \]
    This gives the equivariant identification in degree $m=n-p$
    \begin{equation}
        \Phi^{p-1}[n,(1^{n-p})]\boxtimes \Sym^{n-p}(\Q) \cong H_{2(p-1)}(F(\R^3,n-1))\otimes \sgn
    \end{equation}
    which produces the second case of the theorem. The remaining cases vanish due to Proposition \ref{prop:vanishing coefficients}.
\end{proof}

Let us now consider symmetric powers. In this case the collision filtration exhibits an extension, and we identify its terms using Deligne's theory of weights.

\begin{prop}[\textbf{Symmetric powers}]\label{prop:symmetric powers}
    Let $n\geq 3$. For the partition $\lambda = (m)$ the Schur functor $\Schur{\lambda} = \Sym^m$. The equivariant multiplicity $\Phi^p[n,(m)]$ of $\Sym^m$ in the functor $\Psi^p(\n,-)$ is
\begin{equation}
    \cong
    \begin{cases} 
    H_{n-m}(\ModuliCurve{0,n}) & \text{if }p=n-m \\
    H_{n-m-2}(\ModuliCurve{0,n}) & \text{if }p=n-m-1 \\
    0 & \text{otherwise},
    \end{cases}
\end{equation}
where $\ModuliCurve{0,n}$ is the moduli space of genus $0$ algebraic curves with $n$ marked points.
\end{prop}

\begin{proof}
Proceeding as in the previous case, Corollary \ref{cor:equivariant multiplicities} identifies $\Psi^p[n,(m)]$ with the multiplicity of $\Sym^m(\Q)$ in $\gr^p H^{p+*}_c(F(S^2,n))$ as a representation of the group $\pi_0(\EndSpace{S^2})\cong \Z$. We therefore seek to understand $H^*_c(F(S^2,n))$.

Since $F(S^2,n)$ is a manifold, Poincar\'{e} duality gives an $\SymGroup$-equivariant isomorphism
\[
H^{2n-*}_c(F(S^2,n)) \cong H_{*}(F(S^2,n)).
\]
Thinking of $S^2$ as $\mathbb{C}P^1$, the group $PSL_2(\mathbb{C})$ acts 3-transitively and \cite[Theorem 2.1]{FZ00} shows that the quotient map
\[
F(\mathbb{C}P^1,n) \to \ModuliCurve{0,n}
\]
is a trivial $PSL_2(\mathbb{C})$-principal bundle. This projection is clearly $\SymGroup$-equivariant. Moreover, all spaces and maps involved are algebraic, so their homology is equipped with a natural Hodge structure.

Since the rational cohomology of $PSL_2(\mathbb{C})$ is the same as that of $S^3 \simeq \mathbb{C}^2\setminus \{0\}$, by the same argument as in the case of alternating powers there is an $\SymGroup$-equivariant isomorphism
\begin{equation}\label{eq:splitting configs on S^2}
    H_i(F(S^2,n)) \cong H_i(\ModuliCurve{0,n}) \oplus H_{i-3}(\ModuliCurve{0,n})\otimes \Q(-2).
\end{equation}
It is also known that $H_i(\ModuliCurve{0,n})$ is pure of weight $2i$ as the complement of a hyperplane arrangement. Overall it follows that $H_i(F(S^2,n))$ is mixed of weights $2i$ and $2i-2$.

We want to compare \eqref{eq:splitting configs on S^2} with the collision filtration. By Proposition \ref{prop:vanishing coefficients} there is an extension
\begin{equation}
    0 \to \gr_i^F H^{2n-i}_c(F(S^2,n)) \to H^{2n-i}_c(F(S^2,n)) \to \gr_{i-2}^F H^{2n-i}_c(F(S^2,n)) \to 0
\end{equation}
and we wish to match it with the one in \eqref{eq:splitting configs on S^2}. It suffices to show that $\gr_i^F H^{2n-i}_c(F(S^2,n))$ is pure of weight $2n-2i$ and $\gr_{i-2}^F H^{2n-i}_c(F(S^2,n))$ is pure of weight $2n-(2i-2)$, for then under Poincar\'{e} duality they must match $H_i(\ModuliCurve{0,n})$ and $H_{i-3}(\ModuliCurve{0,n})\otimes \Q(-2)$ respectively.

Petersen showed in \cite[\S3.2]{Pet17} that the collision spectral sequence is compatible with Hodge structures, thus it suffices to show the claimed purity on the $E_1$ page. We will prove generally that $E_1^{p,2q}$ has pure weight $2q$. Indeed, Proposition \ref{prop:spectral_sequence_unidim} shows that $E_1^{p,2q}$ is a sum of $H^2(\mathbb{C}P^1)^{\otimes q}$, and thus has the claimed pure weight.

Now $\gr_p^F H^{p+*}_c(F(S^2,n)) \cong \oplus_{m} \Phi^p[n,(m)]\boxtimes \Sym^m(\Q)$ where the $m$-th symmetric power has grading $2m$. Together with $\gr_i^F H^{2n-i}_c(F(S^2,n)) \cong H_{i}(\ModuliCurve{0,n})$ it follows that for $i=p$ and $2m = * =2n-2p$ we have an isomorphism
\begin{equation}\label{eq:sym = M_0n}
\Phi^p[n,(n-p)] \boxtimes \Sym^{n-p}(\Q) \cong H_{p}(\ModuliCurve{0,n}).
\end{equation}
The first case of the theorem follows.

The same argument applied to $\gr_{i-2}^F H^{2n-i}_c(F(S^2,n)) \cong H_{i-3}(\ModuliCurve{0,n})\otimes \Q(-2)$ gives the second case. And the vanishing in all other cases follows from Proposition $\ref{prop:vanishing coefficients}$.
\end{proof}

\begin{rmk}[\textbf{$\End(\mathbb{P}^1)$-action on $\ModuliCurve{0,n}$}]
A consequence of the calculation in the proof is a description of the $\pi_0(\End(\mathbb{P}^1))$-action on $H_*(\ModuliCurve{0,n})$, defined via Poincar\'{e} duality $H_*(\ModuliCurve{0,n}) \cong H^*_c(\ModuliCurve{0,n})$ with $\End(\mathbb{P}^1)$ acting on the cohomology.
Explicitly, \eqref{eq:sym = M_0n} shows that a degree $k$ map on $\mathbb{P}^1$ acts on $H_p(\ModuliCurve{0,n})$ as multiplication by $k^{n-p}$.
\end{rmk}

\subsubsection{Explicit characters} \label{sec:characters of M0n and FR3n}
The $\SymGroup$-representations arising in the previous section can be understood combinatorially: let us recall their characters. 
For every $\SymGroup$-representation $W$, let $\operatorname{ch}_n(W)\in \Lambda$ denote the Frobenius characteristic of $W$,
\[
\operatorname{ch}_n(W) := \sum_{\lambda\vdash n} \frac{\chi_W(\lambda)}{z_\lambda}p_{\lambda}
\]
where $p_\lambda$ are the power-sum symmetric functions, $\chi_W(\lambda)$ is the character value of $W$ on a permutation of cycle type $\lambda$, and $z_\lambda = \prod i^{m_i} m_i!$ for $\lambda = (1^{m_1},2^{m_2},\ldots,n^{m_n})$.

The equivariant Poincaré polynomial for $H_{\ast}(\ModuliCurve{0,n})$ was computed by Getzler in \cite[Theorem 5.7]{Get95}, where he gave the following formula in the ring of symmetric functions. Below, raising one symmetric function to the power of another $f^g$ is interpreted using the plethystic exponential and logarithm of Getzler--Kapranov \cite{GK-modular-operad}
\[
f^g := \operatorname{Exp}(g \cdot \operatorname{Log}(f)).
\]
\begin{prop}
    The characters of the $\SymGroup$-representations $H_i(\ModuliCurve{0,n})$ are encoded by the generating function 
    $\operatorname{Ch}_t(\mathbf{m}) = \sum_{n=3}^{\infty}\sum_{i=0}^{\infty}{(-t)^i ch_n(H_i(\ModuliCurve{0,n}))}$ given by
    \begin{equation}\label{eq:character-M0n}
        \operatorname{Ch}_t(\mathbf{m}) = \kappa\frac{1+tp_1}{1-t^2}\prod_{n=1}^{\infty}{(1 + t^np_n)^{R_n(t)}}
    \end{equation}
    where $p_i=(x_1^i+x_2^i+\ldots)$ denote the power-sum symmetric functions, $\kappa$ is a truncation operator sending the monomials $1$, $p_1$, $p_{1}^2$ and $p_2$ to zero while fixing the other monomials, and
    \[
        R_n(t) = \frac{1}{n}\sum_{d\mid n}{\frac{\mu(n/d)}{t^d}}
    \]
    with $\mu$ the Möbius function.
\end{prop}
The character of $H_i(\ModuliCurve{0,n})$ is thus obtained from the degree $n$ part of the symmetric function appearing as coefficient for $t^i$ in \eqref{eq:character-M0n}.

A formula for the character of the $\SymGroup$-representation $H_i(F(\R^3,n-1))$ was given by Pagaria very recently.
\begin{thm}[{\cite[Corollary 4.11]{pagaria}}]
The characters of the $\SymGroup$-representations $H_i(F(\R^3,n-1))$ are encoded by 
    $\operatorname{Ch}_t(\mathbf{w}) = \sum_{n=1}^{\infty}\sum_{i=0}^{\infty}{t^i ch_n(H_{2i}(F(\R^3,n-1)))}$ where
    \begin{equation}
    \operatorname{Ch}_t(\mathbf{w}) = \frac{1}{1-t}\left( (1-t\cdot p_1)^{\frac{t-1}{t}} -1\right).
    \end{equation}
\end{thm}

\subsection{A lower bound}\label{sec:lower-bound}

The Euler characteristic of the collision spectral sequence \eqref{eq:chevalley-eilenberg} gives a concrete lower bound on $\gr H^*_c(F(X,n))$. Let us focus on configurations on wedges of circles: $X = \bigvee_{i=1}^g S^1$. In this section we will use the shorthand $H_c^*$ to denote $H_c^*(F(X, n))$. 

\begin{lem}[Lower bound]
    Decompose the terms of the Chevalley--Eilenberg complex of the twisted Lie algebra $(\Q\oplus V[-1])\otimes \Suspension\Lie$ as
    \begin{equation}
    \Sym^p\left((\Q\oplus V[-1])\otimes \Suspension\Lie[1]\right) (n) \cong \bigoplus_{q= 1}^n
    M^{p,q} \otimes_{\SymGroup[q]}V^{\otimes q}
\end{equation}
for $\SymGroup \curvearrowright M^{p,q} \curvearrowleft \SymGroup[q]$ explicitly computable (though complicated) $\SymGroup\times\SymGroup[q]$-bimodules.

Fix $q\geq 1$ and suppose the irreducible $\Specht{\lambda}\boxtimes\Specht{\mu}$ appears in the virtual representation
\[ \sum_{p=0}^{n-q} (-1)^{p+q} M^{p,q}
\]
with (signed) multiplicity $m$. When $m$ has sign $(-1)^n$, the top cohomology $\gr H^n_c$ contain at least $|m|$ copies of $\Specht{\lambda}\otimes\Schur{\mu}(\tilde{H}^1(X))$. Otherwise, $\gr H^{n-1}_c$ contains at least $|m|$ copies of $\Specht{\lambda}\otimes\Schur{\mu}(\tilde{H}^1(X))$.
\end{lem}
Explicit computations are easily carried out on a home computer. Some consequences are given below.
\begin{proof}
    The $E_1$-page of the collision spectral sequence \eqref{eq:chevalley-eilenberg} is given by the Chevalley--Eilenberg complex on $H^*(X)\otimes \Suspension\Lie$, and $H^*(X) = \Q \oplus H^1(X)[-1]$. Since the Euler characteristic can be computed at every page of a spectral sequence, the degree $q$ polynomial subfunctor of  $(\gr H^{n}_c) - (\gr H^{n-1}_c)$ has the form
\begin{equation}\label{eq:alternating sum Mpq}
    \Big(\sum_{p=0}^{n-q} (-1)^{p+q} M^{p,q}\Big)\otimes_{\SymGroup[q]}\tilde{H}^1(X)^{\otimes q}.
\end{equation}
    As this Euler characteristic is in fact the difference between two genuine representations, the claim follows.
\end{proof}

\begin{rmk}
Together with the exact multiplicities of symmetric and alternating powers from \S\ref{sec:multiplicity symmetric powers}, the lower bound produced in this way has thus far proved to be a good approximation for the true cohomology. For example, up to $n=11$ particles, the ranks our estimates produce capture $\sim 90\%$ of the cohomology.
\end{rmk}

Forgetting the $\SymGroup$-action, the lower bound gives the following explicit estimate on the (nonequivariant) multiplicity of every Schur functor.

\begin{prop}
Let $\lambda\vdash q$ be any partition, then the Schur functor $\Schur{\lambda}$ appears in $(\gr H^n_c) - (\gr H^{n-1}_c)$ with nonequivariant multiplicity
    \[
    (-1)^{n-1}\big(|s(n-1,q)|-|s(n-1,q-1)|\big) \dim(\Specht{\lambda})
    \]
    where $s(n,q)$ is a Stirling number of the first kind and $\dim(\Specht{\lambda})$ has a combinatorial description as the number of standard Young tableaux of shape $\lambda$, equivalently given by the hook-length formula.
\end{prop}

\begin{proof}
    In Proposition \ref{prop:spectral_sequence_unidim} we noted that $E_1^{p,q}$ consists of $|s(n,n-p)|\cdot \binom{n-p}{q}$ copies of the tensor power $T^q(\tilde{H}^1(X))$. By Schur--Weyl duality, the tensor power decomposes as
    \[
    \bigoplus_{\lambda \vdash q} \Specht{\lambda}\boxtimes \Schur{\lambda}(\tilde{H}^1(X)).
    \]
    Thus the Schur functor $\Schur{\lambda}$ appears in the Euler characteristic with multiplicity
    \[
    \sum_{p=0}^{n-q} (-1)^{p+q}|s(n,n-p)|\cdot \binom{n-p}{q}\cdot \dim(\Specht{\lambda}).
    \]
    
    Recalling that the Stirling numbers are defined by the relation $(x)_n:= x(x-1)\ldots(x-n+1) = \sum s(n,p)x^p$, the above sums simplify by considering the generating function
    \[
    \sum_{p,q} (-1)^{p}|s(n,n-p)|\cdot \binom{n-p}{q} x^q = (x+1)_{n}.
    \]
    Comparing the $q$-th coefficients of $(x+1)_n = (x+1)\cdot (x)_{n-1}$ one arrives at the claimed multiplicity.
\end{proof}

\begin{expl}\label{super-exponential (2,1)}
The partition $(2,1)\vdash 3$ is the smallest one not accounted for in \S\ref{sec:multiplicity symmetric powers}. The number of copies of $\Schur{(2,1)}$ counted by the Euler characteristic is 
\begin{equation}
\begin{split}
&= (-1)^{n-1}2\left(|s(n-1,3)|-|s(n-1,2)|\right)\\
&= (-1)^{n-1}(n-2)!\left[ {\left(\log\frac{n}{e^{1-\gamma}}\right)}^2 - O(1) \right]
\end{split}
\end{equation}
as follows from the well-know comparison between Stirling numbers and harmonic sums (see e.g. \cite{adamchik-stirling}). Here $\gamma$ is the Euler–Mascheroni constant so that $e^{1-\gamma} \approx 1.5$. In particular the multiplicity of $\Schur{(2,1)}$ in $\gr H^{n-1}_c$ grows super exponentially in $n$.
\end{expl}

More generally, the estimate $|s(n-1,q)| \sim \frac{(n-2)!}{(q-1)!}(\log(n)+\gamma)^q$ from \cite{adamchik-stirling} shows that for every $\lambda$ the multiplicity of $\Schur{\lambda}$ in $(\gr H^{n-1}_c) - (\gr H^{n}_c)$ is at least on the order of $C(n-2)!\log(n/c)^q$ for some constants $C$ and $c$.

\begin{rmk}
    The last estimate is surprising. On the one hand it implies that Schur functors of low degree $|\lambda|$ are hugely more prevalent in the bottom cohomology $\gr H^{n-1}_c$ than in the top $\gr H^{n}_c$. On the other hand the presentation in \eqref{eq:2-step integral} shows that $\dim H^{n}_c - \dim H^{n-1}_c \sim (n!)^{n^g}$ so the top cohomology $H^{n}_c$ is much larger overall. One concludes that the top cohomology contains many more Schur functors of high degree.
    
    To be more explicit, an estimate of Erd\"{o}s \cite{Erdos53} gives that the numbers $|s(n,q)|$ are monotonically increasing in $q$ until $q \approx \log(n)$ and then they decrease. This means that $\Schur{\lambda}$ with $|\lambda|\lesssim\log(n)$ are more common in $\gr H^{n-1}_c$, but all remaining ones are more prevalent in $\gr H^n_c$. It further follows that every Schur functor with $|\lambda|\leq n$ does appear in $\gr H^*_c$, with only one exception for $\lambda=(2)$, since $\Sym^2$ does not appear in $H^*_c(F(X,3))$. 
    
\end{rmk}

\section{Applications}

\subsection{Weight \texorpdfstring{$0$}{0} cohomology of \texorpdfstring{$\ModuliCurve{2, n}$}{M2,n}}\label{sec:cohomology-M2n}
Our original motivation for studying the cohomology $H^*_c(F(X,n))$ is its connection with the cohomology of moduli spaces of algebraic curves $\ModuliCurve{g,n}$. This relationship comes from the description of the weight $0$ part of $H^*_c(\ModuliCurve{g,n})$ in terms of tropical geometry, and is manifested most explicitly in genus $g=2$ as given by the following theorem, which can be found under an equivalent formulation in \cite{BCGY21}.

\begin{thm}[{\cite[Theorem 1.2]{BCGY21}}]\label{thm:cohomology-tropical}
Fix $n \in \N$ and let $X = S^1\vee S^1$ be a wedge of two circles. There exists an $\SymGroup$-equivariant isomorphism
\begin{equation}\label{eq:BCGY}
    \gr_0^{W}H_c^{3+\ast}(\ModuliCurve{2, n}) = (H_c^{\ast}(F(X, n))\otimes \sgn[3])^{\SymGroup[2]\times\SymGroup[3]}.
\end{equation}
\end{thm}
In this formula, the $(\SymGroup[2]\times\SymGroup[3])$-action is the following. On $\sgn[3]$ it is the sign representation of $\SymGroup[3]$, with $\SymGroup[2]$ acting trivially. On $H_c^{\ast}(F(X,n))$ it factors through the action of $\Out(F_2)$ via a homomorphism $\SymGroup[2]\times\SymGroup[3]\to\Out(F_2)$. Using Nielsen's identification $\Out(F_2)\cong \GL_2(\Z)$ \cite{Nielsen1917}, the latter homomorphism is the $2$-dimensional representation $\sgn[2]\otimes \operatorname{std}_3$, where $\sgn[2]$ is the sign representation of the $\SymGroup[2]$ factor and $\operatorname{std}_3 = \Z^3/\langle(1,1,1)\rangle$ is the standard $2$-dimensional representation of the $\SymGroup[3]$ factor.


With this fact at hand we can proceed to prove Proposition \ref{prop:cohom-moduli-space} from the introduction.

\begin{proof}[Proof of Proposition \ref{prop:cohom-moduli-space}]
    By \cite{Nielsen1917}, abelianization gives $\Out(F_2)\cong \GL_2(\Z)$. Since Theorem \ref{thm:polynomiality} shows that the associated graded $\gr H^*_c(F(X,n))$ is a polynomial representation of $\GL_2(\Z)$, and since those are semi-simple, the collision filtration splits canonically so that $H^*_c(F(X,n)) \cong \gr H^*_c(F(X,n))$ as $\SymGroup\times\Out(F_2)$-representations.
    
    The dimension of $(H_c^{\ast}(F(X, n))\otimes \sgn[3])^{\SymGroup[2]\times\SymGroup[3]}$ can be computed as the scalar product of $\SymGroup[2]\times \SymGroup[3]$-characters
    \[
        \left\langle \operatorname{triv}, \ \Res_{\SymGroup[2]\times\SymGroup[3]}^{\GL_2(\Z)}{H_c^{\ast}(F(X, n))}\otimes \sgn[3]\right\rangle = \left\langle \sgn[3], \ \Res_{\SymGroup[2]\times\SymGroup[3]}^{\GL_2(\Z)}{H_c^{\ast}(F(X, n))}\right\rangle.
    \]
    To remember the $\SymGroup$-action one may think of this calculation taking place in the ring of virtual $\SymGroup[n]$-representations.

    Expand the polynomial functor $H^*_c(F(X,n))$ as in \eqref{eq:schur decomposition},
    \[
    H^*_c(F(X,n)) \cong \bigoplus_{\lambda} \Phi[n,\TransposePartition{\lambda}] \boxtimes \Schur{{\lambda}}(\Q^2)
    \]
    where each $\Phi[n,\TransposePartition{\lambda}]$ is an $\SymGroup$-representation. Here the conjugate partition $\TransposePartition{\lambda}$ appears since $\tilde{H}^*(X)\cong \Q^2$ is concentrated in odd grading.
    It follows that the above character inner product expands similarly,
    \begin{equation}
    \begin{split}
         &\left\langle  \sgn[3] , \ \sum_{\lambda}\Res_{\SymGroup[2]\times\SymGroup[3]}^{\GL_2(\Z)}{\Phi[n, \TransposePartition{\lambda}] 
         \boxtimes \Schur{{\lambda}}(\Q^2)} \right\rangle \\ & = \sum_{\lambda}\left\langle  \sgn[3], \ \Res_{\SymGroup[2]\times\SymGroup[3]}^{\GL_2(\Z)}{\Schur{{\lambda}}(\Q^2)}\right\rangle \cdot \Phi[n,\TransposePartition{\lambda}].
            \end{split}
    \end{equation}
    
    This implies the first part of Proposition \ref{prop:cohom-moduli-space}, identifying the multiplicity of the $\SymGroup$-irreducible $\Specht{\lambda}$ as
    \begin{equation}\label{value-a-lambda}
        r_{\lambda} = \left\langle  \Res_{\SymGroup[2]\times\SymGroup[3]}^{\GL_2(\Z)}{\Schur{\lambda}(\Q^2)} \ ;\ \sgn[3] \right\rangle. 
    \end{equation}
The values of the coefficients $r_\lambda$ are given by the following lemma.
\end{proof}

\begin{rmk}
    The formula in Proposition \ref{prop:cohom-moduli-space} first came to our knowledge via private communications with Dan Petersen, who derived this formula during ongoing work joint with Orsola Tommasi. Interestingly, in their work the coefficients $r_{\lambda}$ are defined as
    \[
        r_{\lambda} = \dim \gr_0^{W}H_c^3(\ModuliCurve{2}, \mathbb{V}_{\lambda})
    \]
    with $\mathbb{V}_{\lambda}$ the local system associated to the corresponding irreducible representation of the symplectic group $\operatorname{Sp}_4$.
    
    Due to the restriction on the coefficient $\Phi[n,n]$ given in Proposition \ref{prop:vanishing coefficients}, it follows that our definition of $r_{\lambda}$ agrees with theirs. We plan to investigate further this coincidence in future work. 
\end{rmk}

\begin{lem}
    The coefficients $r_\lambda$ from Proposition \ref{prop:cohom-moduli-space} are the following. For $\lambda = (a,b)$ with $a\geq b \geq 0$,
    \[
        r_{(a,b)} = \begin{cases}
        \left\lfloor{\frac{a-b}{6}}\right\rfloor +1 & \text{if }a\equiv_2 b \equiv_2 1 \\
        \left\lfloor{\frac{a-b}{6}}\right\rfloor & \text{if }a\equiv_2 b \equiv_2 0 \\
        0 & \text{if }a\not\equiv_2 b .
        \end{cases}
    \]
\end{lem}
\begin{proof}
    The group $\SymGroup[2]\times\SymGroup[3]$ is included in $\GL_2(\Z)$ by sending the transposition $(12)\in \SymGroup[2]$ to $-I_2$, transpositions in $\SymGroup[3]$ map to matrices with eigenvalues $(1,-1)$ and its $3$-cycles map to matrices with eigenvalues $(\zeta_3,\zeta_3^{-1})$ where $\zeta_3$ is a primitive $3$rd root of unity.
    
    The character of the Schur representation $\Schur{(a,b)}(\Q^2)$ is given by sending a matrix $A\in \GL_2(\Z)$ with eigenvalues $(x,y)$ to the Schur polynomial $s_{(a,b)}(x,y) = \sum_{i=b}^a x^i y^{a+b-i}$ (see \cite[\S I.3]{MacDonald}).
    One readily checks that transpositions in $\SymGroup[3]$ acts with trace equal to $(-1)^b$ times the residue of $(a-b+1)$ mod $2$, and $3$-cycles act with trace $(a-b+1)$ mod $3$ (taking $2$ mod $3$ to have residue $-1$).
    With these character values the scalar product with $\sgn[3]$ is as claimed.
\end{proof}

\begin{expl}[Multiplicity of $\Specht{(2,1^{n-2})}$ and $\Specht{(n-1,1)}$]\label{ex:BCGY conjecture}
Let us use calculations of bead representations to see that the $\SymGroup$-representation $\Specht{(2,1^{n-2})}\cong \operatorname{Std}_n\otimes \sgn$ never occurs in $\gr^W_0 H^{n+2}_c(\ModuliCurve{2,n})$, as conjectured in \cite[Conj. 3.5]{BCGY21}. Indeed, in Example \ref{ex:(2,1)-bead rep} we explained that Powell--Vespa effectively compute $\Specht{(2,1^{n-2})}$ to never occur in the cohomology $H^{n-1}_c(F(X,n))$. In light of Proposition \ref{prop:cohom-moduli-space}, this immediately implies that this irreducible representation does not appear in $\gr^W_0 H^{n+2}_c(\ModuliCurve{2,n})$.

The other half of the conjecture involves the standard representation $\Specht{(n-1,1)}$. After a preprint of the present paper was made public the multiplicity of this representation has been computed by Powell \cite{PowBabyBead}. Combining his result with Proposition \ref{prop:cohom-moduli-space} shows that $\gr^W_0 H^{n+3}_c(\ModuliCurve{2,n})$ contains this representation with the conjectured multiplicity. Since the equivariant Euler characteristic of $\ModuliCurve{2,n}$ is known it follows that $\gr^W_0 H^{n+2}_c(\ModuliCurve{2,n})$ also contains this representation with the conjectured multiplicity.
\end{expl}

Together with our computations of $\Phi[n,\lambda]$ for $n\leq 10$ below, Proposition \ref{prop:cohom-moduli-space} recovers the $\SymGroup$-character of $\gr_0^{W} H_c^*(\ModuliCurve{2,n})$. For $n=11$ we have computed the relevant summands of $\Phi[11,\lambda]$, and obtained the result shown below. All our computations agree with those present in \cite{BCGY21}.
\begin{align*}
    \gr_0^{W} H_c^{13}(&\ModuliCurve{2,11}, \Q) = 3\Specht{(9,1^2)} + 3\Specht{(8,3)} + 5\Specht{(8,2,1)} + 3\Specht{(8,1^3)} + 2\Specht{(7,4)} + 16\Specht{(7,3,1)} +\\ &5\Specht{(7,2^2)} + 16\Specht{(7,2,1^2)} + 2\Specht{(7,1^4)} + 4\Specht{(6,5)} + 15\Specht{(6,4,1)} + 23\Specht{(6,3,2)} +\\ &28\Specht{(6,3,1^2)} + 24\Specht{(6,2^2,1)} + 21\Specht{(6,2,1^3)} + 5\Specht{(6,1^5)} + 10\Specht{(5^2,1)} + 19\Specht{(5,4,2)} +\\ &28\Specht{(5,4,1^2)} + 21\Specht{(5,3^2)} + 50\Specht{(5,3,2,1)} + 28\Specht{(5,3,1^3)} + 13\Specht{(5,2^3)} + 38\Specht{(5,2^2,1^2)} +\\ &17\Specht{(5,2,1^4)} + 7\Specht{(5,1^6)} + 8\Specht{(4^2,3)} + 29\Specht{(4^2,2,1)} + 20\Specht{(4^2,1^3)} + 25\Specht{(4,3^2,1)} +\\ &28\Specht{(4,3,2^2)} + 48\Specht{(4,3,2,1^2)} + 22\Specht{(4,3,1^4)} + 22\Specht{(4,2^3,1)} + 25\Specht{(4,2^2,1^3)} +\\ &11\Specht{(4,2,1^5)} + 2\Specht{(4,1^7)} + 13\Specht{(3^3,2)} + 8\Specht{(3^3,1^2)} + 22\Specht{(3^2,2^2,1)} + 20\Specht{(3^2,2,1^3)} +\\ &11\Specht{(3^2,1^5)} + 4\Specht{(3,2^4)} + 15\Specht{(3,2^3,1^2)} + 8\Specht{(3,2^2,1^4)} + 6\Specht{(3,2,1^6)} +3\Specht{(2^5,1)} +\\ &4\Specht{(2^4,1^3)} + 2\Specht{(2^3,1^5)} + 2\Specht{(2^2,1^7)} + \Specht{(1^{11})},
\end{align*}
which has dimension $850732$.


For numbers $n$ larger still, the super-exponential multiplicity of every Schur functor discussed in Example \ref{super-exponential (2,1)} implies similar growth in $\gr^W_0 H^*_c(\ModuliCurve{2,n})$\footnote{The $\SymGroup$-equivariant Euler characteristic of these representations is computed in \cite{CFGP}, and gives a lower bound on multiplicity of all $\Specht{\lambda}$ in cohomology. Our lower bound is orthogonal, and shows the codimension $1$ cohomology is much larger than can be deduced from \cite{CFGP}.}. 
The reader can find our equivariant lower bound for $\gr_0^{W} H_c^*(\ModuliCurve{2,n})$ presented on \href{https://louishainaut.github.io/GH-ConfSpace/}{this webpage}\footnote{\url{https://louishainaut.github.io/GH-ConfSpace/}}. Forgetting the $\SymGroup$-action, a lower bound on the dimension of the cohomology $\gr^W_0 H^{n+2}_c(\ModuliCurve{2,n})$ for all $n\leq 17$ are listed in Table \ref{tab:ranks-M2n}.

\begin{table}[h]
    \centering
    \begin{tabular}{c|c|c|c|c|c|c|c|c|c|c|c}
        $n$ & $0$ & $1$ & $2$ & $3$ & $4$ & $5$ & $6$ & $7$ & $8$ & $9$ & $10$ \\
        \hline
        $\dim$ & $0$ & $0$ & $0$ & $0$ & $1$ & $5$ & $26$ & $155$ & $1066$ & $8666$ & $81012$
    \end{tabular}
    \begin{tabular}{c|c|c|c|c}
        $n$ & $11$ & $12$ & $13$ & $14$ \\
        \hline
        $\dim$ & $850\,732$ & $\geq 7\,920\,155$ & $\geq 94\,325\,925$ & $\geq 1\,220\,494\,146$ 
    \end{tabular}
    \begin{tabular}{c|c|c|c}
        $n$ & $15$ & $16$ & $17$ \\
        \hline
        $\dim$ & $\geq 17\,048\,375\,436$ & $\geq 255\,669\,776\,040$ & $\geq 4\,096\,729\,778\,379$  
    \end{tabular}
    \caption{Dimension of $\gr_0^{W} H_c^{n+2}(\ModuliCurve{2,n}, \Q)$}
    \label{tab:ranks-M2n}
\end{table}

\subsection{Patterns and conjectures}\label{sec:conjectures patterns}

We start this section by proving the fourth statement of Proposition \ref{prop:vanishing coefficients},
\begin{prop}
    For a wedge of circles $X = \bigvee_{i=1}^g S^1$,
    \[
        \gr_0 H_c^{n-1}(F(X, n)) = \Specht{(1^n)}\boxtimes \Sym^{n-1}(\Q^g).
    \]
    Equivalently, the coefficient $\Phi^0[n,n-1]$ of the polynomial decomposition is
    \[
        \Phi^0[n, n-1] = \Specht{(1^n)}\boxtimes \Specht{{(1^{n-1})}}.
    \]
\end{prop}

\begin{proof}
    From Proposition \ref{prop:alternating powers} the equivariant multiplicity of the alternating power $\Lambda^{n-1}$ in $\Psi^0(\n, -)$ is 
    \[
        \Phi^0[n, {(1^{n-1})}] = H_0(F(\R^3, n-1))\otimes \sgn[n] = \sgn[n],
    \]
    so $\Phi^0[n, n-1]$ is at least as large as indicated, and it remains to prove that it does not contain any other terms.
    
    For this we look closely at the collision spectral sequence. Once again, it is enough to prove the statement when $X = \bigvee_{i=1}^g{S^1}$ is a wedge of sufficiently many circles. In fact $g\geq n-1$ is enough, but we may as well consider all $g$'s simultaneously. The polynomial representation $\Phi^0[n, n-1]\otimes_{\SymGroup[n-1]}{\Q^g[-1]}^{\otimes n-1}$ is by definition the kernel of the differential $d_1\colon E_1^{0,n-1}\to E_1^{1,n-1}$. The multiplicity of alternating powers implies that this kernel contains at least one copy of $\sgn[n]\boxtimes \Sym^{n-1}(\Q^g)[1-n]$, which has dimension $\binom{n+g-2}{n-1}$. We claim that the kernel of this differential cannot have larger dimension.
    
    A basis for $E_1^{0,n-1} = \Lambda^{n}\left(H^*(X)\otimes \Suspension\Lie(1)\right)^{n-1}$ is given by tuples
    \begin{equation*}
        \langle(x_1, \alpha_{i_1}), \ldots, (x_{k-1}, \alpha_{i_{k-1}}), (x_k, 1), (x_{k+1}, \alpha_{i_{k+1}}), \ldots, (x_n, \alpha_{i_{n}})\rangle
    \end{equation*}
    where the symbols $x_i$ are copies of the `Lie variable' in $\Suspension\Lie(1)\cong \Q x$, and the vectors $\alpha_1,\ldots,\alpha_g \in \tilde{H}^1(X)$ form a basis. Here $1$ refers to the unit $1\in H^0(X)$, and the indices are such that $1\leq k \leq n$ along with $1\leq i_1,\ldots, \widehat{i_k},\ldots, i_{n} \leq g$. We denote this basis element by $\mathcal{B}_{(i_1,\ldots, \widehat{i_k},\ldots, i_{n})}$.
    
    
    For $E_1^{1,n-1}
    \cong \bigoplus_{\binom{n}{2}}(H^*(X)\otimes \Suspension\Lie(2)) \bigotimes \Lambda^{n-2}\left( H^*(X)\otimes \Suspension\Lie(1) \right) ^{n-2}$ we work with the basis of tuples
    \begin{equation*}
        \langle([x_j, x_k], \alpha_{i_0}), (x_1, \alpha_{i_1}), \ldots, \widehat{(x_{j}, \alpha_{i_{j}})},\ldots, \widehat{(x_{k}, \alpha_{i_{k}})}, \ldots, (x_n, \alpha_{i_{n}})\rangle
    \end{equation*}
    for every pair $\{j,k\}$ and sequence of indices $(i_0,i_1,\ldots, \widehat{i_j},\ldots, \widehat{i_k},\ldots, i_n)$. Note that now all terms in the tuple contain a cohomology class $\alpha_i\in \tilde{H}^1(X)$. We denote this basis element by $\mathcal{C}_{(i_0,i_1,\ldots, \widehat{i}_j,\ldots, \widehat{i}_k,\ldots, i_n)}$.
    
    Now, the differential $d_1$ acts on our basis by
    \begin{equation}
        d_1: \mathcal{B}_{(i_1,\ldots, \widehat{i_k},\ldots, i_{n})} \longmapsto \sum_{j\neq k} \pm \mathcal{C}_{(i_j,i_1,\ldots, \widehat{i}_j,\ldots, \widehat{i}_k,\ldots, n)}
    \end{equation}
    and in particular all resulting basis elements appear with coefficient $\pm 1$. Let us consider the dual problem: which $\mathcal{B}_\mathbf{i}$'s are sent to a linear combination containing a nonzero multiple of a particular $c = \mathcal{C}_{(i_0,\ldots,\widehat{i}_j,\ldots,\widehat{i}_k,\ldots ,i_n)}$?
    
    From the description of $d_1$ there are exactly two basis elements with this property: $b_1 = \mathcal{B}_{(i_1,\ldots, i_0,\ldots,\widehat{i_k},\ldots, i_n)}$ and $b_2 = \mathcal{B}_{(i_1,\ldots, \widehat{i}_j,\ldots,i_0,\ldots, i_n)}$, with $i_0$ inserted in the $j$-th, resp. $k$-th, empty slot. It follows that every element of $\ker(d_1)$ that contains a nontrivial multiple of $b_1$ must also contain one of $b_2$, and the coefficient of one uniquely determines the coefficient of the other.
    
    Define an equivalence relation of the $\mathcal{B}_{\mathbf{i}}$'s, generated by the relation that $b_1\sim b_2$ if their indexing tuples differ by moving one $i_j$ of $b_1$ to the empty slot of $b_2$. Then the discussion in the previous paragraph shows that $\ker(d_1)$ is spanned by $\sim$-equivalence classes, and it remains to count these classes.
    
    It is straightforward to see that equivalence classes are in bijection with multisets $\langle i_1,\ldots,i_{n-1}\rangle$ with $1\leq i_j \leq g$ for all $j$. Thus there are exactly $\binom{n+g-2}{n-1}$ many equivalence classes, and the kernel has at most this dimension. This is what we wanted to show.
\end{proof}

\begin{rmk}
    Using the Euler characteristic as in \S\ref{sec:lower-bound}, one can obtain from the previous proposition a plethystic description of $\gr_1 H^n_c(F(X,n))$, equivalently $\Phi^1[n, n-1]$, identical to the one in \cite[Corollary 3]{PV18}.
\end{rmk}

Calculations for small values of $n$ also suggest the following pattern.
\begin{conj}
    For $n\geq 4$ and $X = \bigvee_g S^1$ a wedge of circles,
    \[
        \gr_1 H_c^{n-1}(F(X,n)) = \big(\Specht{(3,1^{n-3})}\boxtimes \Sym^{n-2}(\Q^g) \big) \oplus \big( \Specht{(n)}\boxtimes \Lambda^{n-2}(\Q^g)\big).
    \]
    Equivalently, the coefficients of $\Phi^1[n, n-2]$ are given by
    \begin{equation}
        \Phi^1[n, n-2] = (\Specht{(3, 1^{n-3})}\boxtimes \Specht{(1^{n-2})}) \oplus (\Specht{(n)}\boxtimes \Specht{(n-2)}).
    \end{equation}
\end{conj}

Up to this point we mainly discussed the $\SymGroup$-equivariant multiplicity of individual Schur functors in $\gr H^*_c(F(X,n))$, equivalently the coefficients $\Phi[-,\mu]$. Let us now shift perspective and instead consider for fixed $\SymGroup$-irreducible $\Specht{\lambda}$ the corresponding isotypic component as a sum of Schur functors. 
These are equivalently given by the coefficient $\Phi[\lambda, -]:= \oplus_{m=0}^{|\lambda|}\Phi[\lambda,m]$.

Through our calculations we have identified several patterns in these coefficients -- see below. We present our conjectural formulas only for multiplicities in codimension 1 cohomology of $F(\vee S^1,n)$, but using the spectral sequence to compute the Euler characteristic, one can readily compute the corresponding multiplicity in codimension 0 as well.


In some of the conjectural patterns below we use the notation 
\[
\MultTriv{n} = \sum_{2a+b = n}{\Schur{{(a, 1^b)}}}, \text{ a sum of hook shapes, }
\]
where the sum is over $a\geq 1$ and $b\geq 0$. The following can be deduced from \cite[Examples 3-4 and Corollary 19.8]{PV18} using the dictionary in \S\ref{sec:bead}.

\begin{prop}\label{prop:known-mult-PV}
For partition $\lambda\vdash n$, the equivariant multiplicity of $\Specht{\lambda}$ in $\gr H^{n-1}_c(F(X,n))$, equivalently the bead representation $U_\lambda^{II}$ from \cite[\S2.5]{TW19}, is the following polynomial functor.
\begin{itemize}
    \item For the trivial representation $\lambda = (n)$, it is $\MultTriv{n-1}$.
\item For the sign representation $\lambda = (1^n)$, it is $\Schur{(n-1)}$.
\item For $\lambda = (2,1^{n-2})$ it is $0$.
\end{itemize}
\end{prop}



We observed the following further patterns, verified computationally for up to $n\leq 11$ and compliant with the lower bound in \S\ref{sec:lower-bound}.
Any multiplicity involving $\Schur{\TransposePartition{(k)}}$ and $\Schur{\TransposePartition{(1^k)}}$ is completely determined by Theorem \ref{thm:symmetric and alternating powers}.

\begin{conj} \label{conj:patterns}
The equivariant multiplicity of $\Specht{\lambda}$ in $\gr H_c^{n-1}(F(X,n))$ are:
\begin{itemize}
    \item For $\lambda = (n-1, 1)$ it is $\begin{cases}
        \Schur{{(n/2)}} & \text{if } n \text{ is even} \\
        0 & \text{if } n \text{ is odd}
    \end{cases}$ (proved in Prop \ref{prop:Powell-baby-bead} below).
    
    \item For $\lambda = (n-2,2)$ it is the sum $\sum_{k=0}^{\lfloor{\frac{n}{4}}\rfloor - 1}{\MultTriv{n-2-4k}} + \sum_{k=0}^{\lfloor{\frac{n-7}{4}}\rfloor}{\MultTriv{n-5-4k}}$
    \[ + \left\{\begin{array}{ll}
        \sum_{k=0}^{\lfloor{\frac{n}{4}}\rfloor-2}\Schur{{\left(\frac{n-4}{2} - 2k\right)}} & \text{if } n \text{ is even} \\
        \sum_{k=0}^{\lfloor{\frac{n-5}{4}}\rfloor}\Schur{{\left(\frac{n-1}{2} - 2k\right)}} & \text{if } n \text{ is odd}
    \end{array}\right.\]
    \item For $\lambda = (n-2,1,1)$ it is the sum $\sum_{k=0}^{\lfloor{\frac{n-5}{4}}\rfloor}{\MultTriv{n-3-4k}} + \sum_{k=0}^{\lfloor{\frac{n-6}{4}}\rfloor}{\MultTriv{n-4-4k}}$
    \[ + \left\{\begin{array}{ll}
        \sum_{k=0}^{\lfloor{\frac{n-6}{4}}\rfloor}\Schur{{\left(\frac{n-2}{2} - 2k\right)}} & \text{if } n \text{ is even} \\
        \sum_{k=0}^{\lfloor{\frac{n-3}{4}}\rfloor}\Schur{{\left(\frac{n+1}{2} - 2k\right)}} & \text{if } n \text{ is odd}
    \end{array}\right.\]
    \item For $\lambda = \TransposePartition{(n-3,3)}$ it is $\sum_{k=0}^{n-6}{\lfloor{\frac{k+3}{3}}\rfloor \Schur{{(n-4-k)}}} + a_n \Schur{{(1,1)}} + b_n\Schur{{(1)}}$
    \item For $\lambda = \TransposePartition{(n-3,2,1)}$ it is $\sum_{k=0}^{n-5}{\left\lfloor{\frac{2k+3}{3}}\right\rfloor\Schur{{(n-3-k)}}} + \left\lfloor{\frac{n-2}{3}}\right\rfloor\Schur{{(1,1)}} + \left\lfloor{\frac{n-4}{3}}\right\rfloor\Schur{{(1)}}$
    \item For $\lambda = \TransposePartition{(n-3,1,1,1)}$ it is $\sum_{k=0}^{n-6}{\left\lfloor{\frac{k+3}{3}}\right\rfloor\Schur{{(n-4-k)}}} + c_n \Schur{{(1,1)}} + d_n \Schur{{(1)}}$
    
        \item For $\lambda = \TransposePartition{(n-2,2)}$ it is $\sum_{k=0}^{\lfloor{n/2}\rfloor-2}{\Schur{{(n-3-2k)}}}$
    \item For $\lambda = \TransposePartition{(n-2,1,1)}$ it is $\sum_{k=0}^{\lfloor{(n-3)/2}\rfloor}{\Schur{{(n-2-2k)}}}$

\end{itemize}
where the numbers $a_n, b_n, c_n$ and $d_n$ are $\lfloor{n/6}\rfloor$ plus $6$-periodic functions of $n$. Explicitly,
\begin{eqnarray*}
a_n = \lfloor{n/6}\rfloor & + &(-1, 0 ,-1, 0 , 0 , 0 )_n \\
b_n = \lfloor{n/6}\rfloor &+& (-1, -1, -1, 0 ,-1, 0)_n \\
c_n = \lfloor{n/6}\rfloor &+& (0 , 0 , 0 , 0 , 1 , 0 )_n \\
d_n = \lfloor{n/6}\rfloor &+& (0 , -1, 0,0,0,0)_n
\end{eqnarray*}
where $(-)_n$ returns the entry in position residue of $n$ modulo $6$ (starting at residue $0$).
\end{conj}



\begin{rmk}
    Interestingly, the complicated patterns for $\lambda = (n-2,2)$ and $(n-2,1,1)$ become far simpler when \emph{summed together}; and the same is true for their conjugate partitions. For the partitions $\lambda = \TransposePartition{(n-3,3)}$ and $\TransposePartition{(n-3,1,1,1)}$ it is now their \emph{difference} that has a significantly simpler form:
    one can observe that $c_n - a_n$ and $d_n - b_n$ both take the value $1$ if $n$ is even and $0$ if $n$ is odd. 
    We have no guess as to why this should be the case.
\end{rmk}


\begin{rmk}
    After a first version of this paper was made public, Powell \cite{PowBabyBead} computed the composition factors of $\Specht{(n-1,1)}$ in $\gr H_c^n(F(X,n))$, related to the first case of Conjecture \ref{conj:patterns} via Euler characteristic. That said, verifying that the conjecture follows from Powell's work is challenging, since calculating the Euler characteristic involves plethysm coefficients whose determination is still an open problem. In the next proposition we prove this case of the conjecture by combining several ideas from this paper.
\end{rmk}

\begin{prop}\label{prop:Powell-baby-bead}
    The equivariant multiplicity of $\Specht{(n-1,1)}$ in codimension $1$ cohomology $\gr H_c^{n-1}(F(X,n))$ is  $\Schur{{(n/2)}}$ when $n$ is even, and $0$ when $n$ is odd.
\end{prop}

\begin{proof}
    Powell \cite[Theorem 3]{PowBabyBead} provides the equivariant multiplicity of $\Specht{(n-1,1)}$ in codimension $0$ cohomology $\gr H_c^n(F(X,n))$. To see that the equivariant multiplicity in codimension $1$ is at least as large as we claim, we note that when $n$ is even, our Proposition \ref{prop:alternating powers} shows that the $\SymGroup$-equivariant multiplicity of the Schur functor $\Schur{(n/2)}$ in codimension $1$ is the same as that of $\Schur{(n/2+1)}$ in codimension $0$, and Powell's calculation shows the representation $\Specht{n-1,1}\boxtimes \Schur{(n/2+1)}$ indeed occurs in codimension $0$.
    
    To show that the multiplicity is no larger than claimed, it is therefore enough to bound its dimension. The (non-equivariant) Euler characteristic of the $\Specht{(n-1,1)}$-multiplicity space can be read-off from the 2-step complex \eqref{eq:2-step isotypic}, and takes the value $(-1)^n(n-1)\binom{n+g-2}{g-2}$.
    One can then compute the non-equivariant multiplicity in codimension $0$ provided by Powell, and verify that the multiplicity in codimension $1$ has the dimension we claim.
\end{proof}

Using Proposition \ref{prop:cohom-moduli-space}, Conjecture \ref{conj:patterns} leads to the following conjectural irreducible multiplicities in $\gr_0^W H_c^{n+2}(\ModuliCurve{2, n}, \Q)$.

\begin{conj}\label{conj:patterns-M2n}
    The irreducible representation $\Specht{\lambda}$ appears in $\gr_0^W H_c^{n+2}(\ModuliCurve{2,n}, \Q)$ with the following multiplicity:
    \begin{itemize}
        \item For $\lambda = (n-2,2)$ it is $\begin{cases}
            \lfloor \frac{k^2-k+1}{3}\rfloor & \text{if } n = 4k \\
            \lfloor \frac{3k^2-k+4}{6}\rfloor & \text{if } n = 4k+1 \\
            \lfloor \frac{k(k-1)}{6}\rfloor & \text{if } n = 4k+2 \\
            0 & \text{if } n = 4k+3 \\
        \end{cases}.$
        \item For $\lambda = (n-2,1,1)$ it is $\begin{cases}
            \lfloor \frac{(k-1)(k-2)}{6}\rfloor & \text{if } n = 4k \\
            0 & \text{if } n = 4k+1 \\
            \lfloor \frac{k^2+k+1}{3}\rfloor & \text{if } n = 4k+2 \\
            \lfloor \frac{3k^2+3k+2}{6}\rfloor & \text{if } n = 4k+3 \\
        \end{cases}.$
        \item For $\lambda = \TransposePartition{(n-2,2)}$ it is $\begin{cases}
            0 & \text{if } n = 2k \\
            \lfloor \frac{(k-2)(k-1)}{6}\rfloor & \text{if } n = 2k + 1\\
        \end{cases}.$
        \item For $\lambda = \TransposePartition{(n-2,1,1)}$ it is $\begin{cases}
            \lfloor \frac{(k-2)(k-1)}{6}\rfloor & \text{if } n = 2k \\
            0 & \text{if } n = 2k+1 \\
        \end{cases}.$
    \end{itemize}
\end{conj}

The multiplicities in Conjectures \ref{conj:patterns} and \ref{conj:patterns-M2n} have nice presentations as generating functions. These can be found in Appendix \ref{app:generating-functions}.

\newpage

\appendix
\section{Full decomposition for \texorpdfstring{$10$}{10} particles} \label{app:full-computation}

We explain here how to combine insights from the different presentations of \S\ref{sec:geometric approach} to compute the full cohomology of $F(X, 10)$ for $X$ any wedge of circles. Conceptually, the $2$-term complex of \S\ref{section:2-step} gives access to individual $\SymGroup$-isotypic components, most effective for Schur functors $\Schur{\lambda}$ for \underline{$\lambda$ with few parts}. On the other hand, the Chevalley--Eilenberg complex of \S\ref{section:CE-complex} accesses individual $\GL$-isotypic components, most effective for $\Schur{\lambda}$ of \underline{large total degree $|\lambda|$}. Between these two, one can uniquely determine the complete representation.

To illustrate our approach we focus on the multiplicity of a few specific $\SymGroup[10]$-irreducibles in the bottom cohomology $H^{9}(F(X,10))$. Note that by the bound on polynomial degree in Proposition \ref{prop:vanishing coefficients} and since the highest degree term $\Phi^0[10,9]$ is also computed there, the only remaining multiplicities that need to be computed are those of Schur functors of degree $\leq 8$. In the following we will show how to compute the equivariant multiplicity of the $\SymGroup[10]$-representations $\Specht{(8,2)}$, $\Specht{(6,3,1)}$, $\Specht{(4^2,2)}$ and $\Specht{(3,2,1^5)}$.

\subsection*{Step 1: Lower bound}

First compute the lower bounds for these $\SymGroup[10]$-irreducible as in \S\ref{sec:lower-bound}, corrected with the knowledge of the exact multiplicity of the symmetric and alternating powers $\Schur{(k)}$ and $\Schur{(1^k)}$ from \S\ref{sec:multiplicity symmetric powers}:

\begin{table}[h]
    \centering
    \begin{tabular}{|m{1.7cm}|m{9.5cm}|}
\hline
$\SymGroup[10]$-irrep & $\GL$-equivariant multiplicities \\
\hline
$(8,2)$ & $\geq \Schur{(4)} + \Schur{(3)} + \Schur{(2,1)} + \Schur{(2)} + \Schur{(1^7)} + \Schur{(1^4)} + \Schur{(1^3)}$ \\
\hline
$(6,3,1)$ & $\geq \Schur{(5)} + 6\Schur{(4)} + 11\Schur{(3)} + 2\Schur{(2,1)} + 10\Schur{(2)} + \Schur{(1^6)} + 2\Schur{(1^5)} + 3\Schur{(1^4)} + 5\Schur{(1^3)} + 7\Schur{(1^2)} + 4\Schur{(1)}$ \\
\hline
$(4^2,2)$ & $\geq \Schur{(5)} + 6\Schur{(4)} + 6\Schur{(3)} + \Schur{(2,1)} + 10\Schur{(2)} + 2\Schur{(1^4)} + 6\Schur{(1^3)} + 5\Schur{(1^2)} + \Schur{(1)}$ \\
\hline
$(3,2,1^5)$ & $ \geq \Schur{(7)} + \Schur{(6)} + 2\Schur{(5)} + 3\Schur{(4)} + 3\Schur{(3)} + 4\Schur{(2)} + 2\Schur{(1^2)} + 2\Schur{(1)}$ \\
\hline
    \end{tabular}
    \caption{Lower bounds for $n=10$}
    \label{table:lower_bound_10}
\end{table}

We already know the multiplicity of all symmetric powers $\Schur{(m)}$, so we proceed to determine the multiplicities of Schur functors with two parts $\Schur{(a,b)}$.

\subsection*{Step 2: Traces on Schurs with two parts}
To gain understanding of the $\End(\Z^2)$-action on each of the above multiplicity spaces, we use the 2-step complex from \ref{section:2-step} with genus $g=2$ to compute its character. Start by computing the trace of the action induced by specific diagonal matrices on each multiplicity space: $diag(1,1), diag(1,2)$ and $diag(2,2)$  -- these calculations are not too demanding as they only involve one or two doubling maps on $S^1\vee S^1$.

Once these traces are obtained, we subtract from them the traces of the Schur functors appearing in the lower bound. We thus obtain the traces of terms missing from our lower bound. In our example, the traces corresponding to the lower bounds, as well as the true computed traces of our $3$ diagonal matrices and the difference between them are as in Table \ref{table:traces_rank2}
\begin{table}[h]
    \centering
    \begin{tabular}{|c||c|c|c|}
\hline
$\SymGroup[10]$-irreducible & Lower bound & Computed traces & Difference \\
\hline
$(8,2)$ & $[14, 59, 140]$ & $[14, 59, 140]$ & $[0, 0, 0]$ \\ 
\hline
$(6,3,1)$ & $[129, 522, 1220]$ & $[145, 580, 1396]$ & $[16, 58, 176]$ \\ 
\hline
$(4^2,2)$ & $[99, 428, 1024]$ & $[109, 464, 1136]$ & $[10, 36, 112]$ \\ 
\hline
$(3,2,1^5)$ & $[72, 684, 2256]$ & $[72, 684, 2256]$ & $[0, 0, 0]$ \\ 
\hline
    \end{tabular}
    \caption{Traces for the matrices diag(1,1), diag(1,2) and diag(2,2)}
    \label{table:traces_rank2}
\end{table}


One thus immediately sees that for the first and last rows the lower bound is sharp, i.e. it accounts for all Schur functors $\Schur{(a,b)}$. The remaining rows are missing a $16$-dimensional and a $10$-dimensional representation respectively, with known character values at $diag(1,2)$ and $diag(2,2)$.


The problem thus reduces to finding a sum of Schur functors that produces the traces seen in the last column of Table \ref{table:traces_rank2}. Note that while the complete traces are rather large, their difference from our lower bound is significantly smaller. Moreover, as stated at the outset of this calculation, only Schur functors of degree $\leq 8$ may appear. The corresponding traces on these Schur functors are given in Table \ref{table:traces_Schur_rank2}.

\begin{table}[h]
    \centering
    \begin{tabular}{|c|c||c|c|}
\hline
Schur functor & Traces & Schur functor & Traces \\
\hline
$(2,1)$ & $[2, 6, 16]$ & $(3,1)$ & $[3, 14, 48]$ \\ 
\hline
$(2^2)$ & $[1, 4, 16]$ & $(4,1)$ & $[4, 30, 128]$ \\ 
\hline
$(3,2)$ & $[2, 12, 64]$ & $(5,1)$ & $[5, 62, 320]$ \\ 
\hline
$(4,2)$ & $[3, 28, 192]$ & $(3^2)$ & $[1, 8, 64]$ \\ 
\hline
$(6,1)$ & $[6, 126, 768]$ & $(5,2)$ & $[4, 60, 512]$ \\ 
\hline
$(4,3)$ & $[2, 24, 256]$ & $(7,1)$ & $[7, 254, 1792]$ \\ 
\hline
$(6,2)$ & $[5, 124, 1280]$ & $(5,3)$ & $[3, 56, 768]$ \\ 
\hline
$(4^2)$ & $[1, 16, 256]$ & &\\ 
\hline
    \end{tabular}
    \caption{Traces of Schur functors at diag(1,1), diag(2,1) and diag(2,2)}
    \label{table:traces_Schur_rank2}
\end{table}

It is useful to note that $s_{(a,b)}(2,2) = 2^{a+b}s_{(a,b)}(1,1)$ for any Schur polynomial $s_{(a,b)}$, i.e. the third trace in the table is always $2^{a+b}$ larger than the first. Therefore since $\Schur{(2,1)}$ is the only Schur functor in our table for which $2^{a+b} < 16$, any sum of other Schur functors must have its third trace at least $16$ times larger than than the first. Since the missing Schur functors for $\Specht{(6,3,1)}$ and $\Specht{(4^2,2)}$ have their third trace $< 16$ times their first, one immediately concludes that $\Schur{(2,1)}$ occurs several times.

Let us therefore account for these copies of $\Schur{(2,1)}$ and remove them from the trace until the third column is $\geq 16$ times the first. For example, from the traces $[10,36,112]$ corresponding to $\Specht{(4^2,2)}$ we subtract $m \times [2,6,16]$ with $m$ such that $16\cdot(10-2m) \leq 112 - 16m$. That is, $m \geq 3$, leaving traces $[4,18,64]$ to be accounted for. In general the $\Schur{(2,1)}$-multiplicity is $\geq c_1- c_3/16$ where $c_i$ is the $i$-th trace in our table.

At this point the remaining Schur functors can be uniquely determined. The Schur functors $\Schur{(a,b)}$ that need to be added can be found in Table \ref{table:additional_Schur_genus2}; the interested reader is invited to verify that no other combination of Schur functors gives the desired traces.

\begin{table}[h]
    \centering
    \begin{tabular}{|c|c|}
\hline
$\SymGroup[10]$-irreducible & Missing Schur functors \\
\hline
$(8,2)$ & $0$ \\
\hline
$(6,3,1)$ &  $5\Schur{(2,1)} + 2\Schur{(3,1)}$\\
\hline
$(4^2,2)$ & $3\Schur{(2,1)} + \Schur{(3,1)} + \Schur{(2,2)}$ \\
\hline
$(3,2,1^5)$ & $0$ \\
\hline
    \end{tabular}
    \caption{Corrections to the lower bounds in genus $2$}
    \label{table:additional_Schur_genus2}
\end{table}

We therefore add these Schur functors to our lower bound estimate and proceed to find Schur functors with $3$ or more parts.

\subsection*{Step 3: Ranks in genus 3 and 4}

One could attempt the same game in higher genus, i.e. finding multiplicities of $\Schur{(a,b,c)}$ and so on, but it turns out that one can (almost always) make due with only the nonequivariant ranks in genus $3$ and $4$. For example, Table \ref{table:ranks34} shows the ranks in genus $3$ and $4$ of the revised lower bound compared with the true rank computed using the $2$-step complex in \S\ref{section:2-step}.


\begin{table}[h]
    \centering
    \begin{tabular}{|c|c|c|c|}
\hline
$\SymGroup[10]$-irreducible & Lower bound & Actual ranks & Difference \\
\hline
$(8,2)$ & $[40, 90]$ & $[46, 126]$ & $[6, 36]$ \\ 
\hline
$(6,3,1)$ & $[405, 897]$ & $[411, 931]$ & $[6, 34]$ \\ 
\hline
$(4^2,2)$ & $[308, 691]$ & $[308, 691]$ & $[0, 0]$ \\ 
\hline
$(3,2,1^5)$ & $[217, 541]$ & $[217, 541]$ & $[0, 0]$ \\ 
\hline
    \end{tabular}
    \caption{Ranks in genus 3 and 4}
    \label{table:ranks34}
\end{table}

From Table \ref{table:ranks34} one sees that only rather small Schur functors may appear. Table \ref{table:rank_Schur34} lists the ranks of all Schur functors of degree $\leq 8$ with $3$ or $4$ parts. For most Specht modules of $\SymGroup[10]$ this table is already sufficient for uniquely determining the multiplicity of every Schur functor with $3$ or $4$ parts.


\begin{table}[h]
    \centering
    \begin{tabular}{|c|c||c|c|}
\hline 
Schur functor & Ranks in genus $3,4$ & Schur functor & Ranks in genus $3,4$ \\
\hline
$(2,1^2)$ & $[3, 15]$ & $(3,2^2)$ & $[3, 36]$ \\ 
\hline
$(3,1^2)$ & $[6, 36]$ & $(3,2,1^2)$ & $[0, 20]$ \\ 
\hline
$(2^2,1)$ & $[3, 20]$ & $(2^3,1)$ & $[0, 4]$ \\ 
\hline
$(2,1^3)$ & $[0, 4]$ & $(6,1^2)$ & $[21, 189]$ \\ 
\hline
$(4,1^2)$ & $[10, 70]$ & $(5,2,1)$ & $[24, 256]$ \\ 
\hline
$(3,2,1)$ & $[8, 64]$ & $(5,1^3)$ & $[0, 35]$ \\ 
\hline
$(3,1^3)$ & $[0, 10]$ & $(4,3,1)$ & $[15, 175]$ \\ 
\hline
$(2^3)$ & $[1, 10]$ & $(4,2^2)$ & $[6, 84]$ \\ 
\hline
$(2^2,1^2)$ & $[0, 6]$ & $(4,2,1^2)$ & $[0, 45]$ \\ 
\hline
$(5,1^2)$ & $[15, 120]$ & $(3^2,2)$ & $[3, 45]$ \\ 
\hline
$(4,2,1)$ & $[15, 140]$ & $(3^2,1^2)$ & $[0, 20]$ \\ 
\hline
$(4,1^3)$ & $[0, 20]$ & $(3,2^2,1)$ & $[0, 15]$ \\ 
\hline
$(3^2,1)$ & $[6, 60]$ & $(2^4)$ & $[0, 1]$ \\ 
\hline
    \end{tabular}
    \caption{Ranks of Schur functors}
    \label{table:rank_Schur34}
\end{table}

\begin{table}[h]
    \centering
    \begin{tabular}{|c|c|}
\hline
$\SymGroup[10]$-irreducible & Missing Schur functors \\
\hline
$(8,2)$ & $\Schur{(3,1^2)} + \Schur{(2,1^4)}$ \\
\hline
$(6,3,1)$ &  $2\Schur{(2,1^2)} + \Schur{(2,1^3)}$\\
\hline
$(4^2,2)$ & $0$ \\
\hline
$(3,2,1^5)$ & $0$ \\
\hline
    \end{tabular}
    \caption{Remaining corrections to the lower bounds}
    \label{table:remaining_Schur}
\end{table}

However, this is not always the case, e.g. for partition $(8,2)$ in our example. Let us explain why the remaining combination of Schurs occuring in our examples are as presented in Table \ref{table:remaining_Schur}. To gain more insight one can take the following steps:
\begin{itemize}
    \item As before, with $\Schur{(2,1)}$ having the smallest ratio between its first and third traces, now it is $\Schur{(2,1^2)}$ that has the smallest ratio between its ranks in genus $3$ and $4$ (at least among the Schurs that are small enough to contribute nontrivially). Thus $\Schur{(2,1^2)}$ must appear with multiplicity high enough to make the second column $\leq 6$ times the first. E.g. it appears at least once for partition $(6,3,1)$ leaving unknown ranks $[3,19]$, at which point it is clear that it must appear with multiplicity $2$ to account for the rank in genus $3$. The remaining ranks $[0,4]$ do not uniquely determine the multiplicities of Schur functors with $4$ parts.

    \item 
    For partition $(8,2)$ ranks alone are not enough to uniquely determine the multiplicity of $\Schur{(3,1^2)}$. However, it is feasible to compute the trace of the matrix $diag(2,1,1)$, and this additional information is sufficient.
    \item The second term of partition $(8,2)$, the functor $\Schur{(2,1^4)}$, does not even show up until one considers wedges of $\geq 5$ circles. However, since the representation $\Specht{(8,2)}$ has relatively small dimension, it is feasible to compute the nonequivariant rank of the isotypic component $H^9_c(F(X,10))_{(8,2)}$ in genus $5$ using the $2$-term complex. This is sufficient for determining the $\Schur{(3,1^2)}$ term from the previous point, and further detects that there exists some $\Schur{\lambda}$ for $\lambda$ with $5$ parts (though it can not decide which). 



\end{itemize}

With these additional considerations we have completely accounted for all Schur functors with $\leq 3$ parts and reduced the ambiguity regarding functors with $4$ parts to a small number of cases of very small dimension: deciding between $\Schur{(2,1^3)}, \Schur{(2^3,1)}$ and $\Schur{(2^4)}$. 

Looking further to $\Schur{\lambda}$ for $\lambda$ with $\geq 5$ parts we reach the computational limits of the 2-term complex, where it is no longer feasible to compute further ranks or traces. Instead, observe that any yet undetected Schur functor must have relatively large polynomial degree. More precisely, any Schur functor with $\geq 5$ parts is either an exterior power or it has total polynomial degree $\geq 6$, as are the degrees of $\Schur{(2^3,1)}$ and $\Schur{(2^4)}$.


\subsection*{Step 4: CE-complex in high polynomial degree}

Fortunately, at high polynomial degree a new tool becomes available: high polynomial degrees appear as $\Phi^{p-1}[10,10-p]$ for $p$ small, so the Chevalley--Eilenberg complex more readily gives access to the nonequivariant multiplicity of individual Schur functors. Although this complex forgets the data as to which $\SymGroup$-isotypical component they belong, and further simplification is needed.

This couples with the crucial observation that by plugging wedges of even spheres into the CE-complex, the partitions of $\Schur{\lambda}$ are the conjugate of the ones appearing for wedges of circles. For example the functor $\Schur{\TransposePartition{(2,1^6)}}$, which would have been computationally expensive to access and require on a wedge of $\geq 7$ circles, contributes $\Schur{(7,1)}$ on a wedge of only two even spheres! And its high polynomial degree makes it even easier to access.


With this, the last step in our calculation is to compute the nonequivariant multiplicity of Schur functors of polynomial degree $\geq 6$ and compare them with the total multiplicity of those functors in our lower bound. If they are found to agree, the bound is sharp and all Schur functors have been counted. Fortunately, this always turned out to be the case


To make the last computations feasible we use one more trick: inputting wedges of spheres of different dimensions. Let us explain one such computation in detail, and for the other ones we will only specify what the analogous computation produces.

\subsection*{Step 4': Mixing spheres of different dimensions}

Say we want to compute the nonequivariant multiplicity of the Schur functor $\Schur{\TransposePartition{(2,1^4)}}$, which on even spheres gives $\Schur{(5,1)}(\tilde{H}^*(X))$ so two spheres suffice. Consider the space $X = S^2\bigvee S^4$.

The CE-complex for this $X$ admits a multigrading where the multidegree $(a,b)$ consists of tensors
\[
    (H^2(S^2)^{\otimes a}\otimes H^4(S^4)^{\otimes b})^{\oplus \binom{a+b}{a}} \leq \tilde{H}^*(X)^{\otimes (a+b)}.
\]
This multigrading is preserved by the CE-differential, so one may decompose the complex and consider one multidegree at a time. The $(a,b)$-multigraded subcomplex is substantially smaller than most $\Schur{\lambda}$-isotypic components, and furthermore it is spanned by a basis of pure tensors. These facts make calculations feasible, and we are able to determine the $(a,b)$-multigraded part of the CE-homology.

At the same time, $\EndSpace{X}$ acts on these subcomplexes, giving them the structure of polynomial $\GL(1)\times \GL(1)$-representations, with the $(a,b)$-multigraded component coinciding with the $\Sym^a(H^2(X))\otimes \Sym^b(H^4(X))$-isotypic component of this action. Knowing the multiplicity of each of these representations in homology, one is able to deduce the multiplicity of every $\Schur{\lambda}$ for $\lambda$ with $2$ parts: one needs only to count how often every $\Sym^a\otimes \Sym^b$ occurs in each $\Schur{\lambda}$.



Let us consider the $(5,1)$-multigraded component.
One can compute\footnote{In general the multiplicity with which $\Sym^{a_1}\otimes\ldots\otimes \Sym^{a_r}$  appears in $\Schur{\lambda}$ is equal to the number of semistandard Young tableax of shape $\lambda$ filled with $a_i$ many $i$'s.} that $\Sym^5\otimes \Sym^1$ only occurs once in $\Schur{(5,1)}$ and once in $\Schur{(6)}$. It follows that its total multiplicity in CE-homology is $\Phi^3[10,(5,1)] \oplus \Phi^3[10,(6)]$. Given that we already know the multiplicities of all symmetric powers, this determines the rank of $\Phi^3[10,(5,1)]$.

One then compares this rank with the lower bound, and find that they agree. It follows that the bound accounts for the entire multiplicity of $\Schur{(2,1^4)}$ on wedges of circles, which is thus known equivariantly. Note further that approaching this calculation with the multigraded components has the added benefit that it is sensitive to multiple Schur functors at a time: if the lower bounds of all of them sum to the calculated multiplicity of the multigraded component, then all their lower bounds in fact exhaust the whole multiplicity space.

Other cases of this approach needed to completely determine $H^9_c(F(X,10))$ are:
\begin{itemize}
    \item $X = S^2 \bigvee S^4$, multidegree $(4,3)$: computing rank of $\Phi^2[10,(4,3)] + \Phi^2[10,(5,2)] + \Phi^2[10,(6,1)] + \Phi^2[10,(7)]$.
    \item $X = S^1 \bigvee S^2$, multidegree $(2,5)$: computing rank of $\Phi^2[10,(5,1^2)] + \Phi^2[10,(6,1)]$.
    \item $X = S^2 \bigvee S^4$, multidegree $(4,4)$: computing rank of $\Phi^1[10, (4,4)] + \Phi^1[10,(5,3)] + \Phi^1[10,(6,2)] + \Phi^1[10,(7,1)] + \Phi^1[10,(8)]$.
    \item $X = S^2 \bigvee S^4 \bigvee S^6$, multidegree $(5,2,1)$: computing rank of $\Phi^1[10,(5,2,1)] + \Phi^1[10,(5,3)] + \Phi^1[10,(6,1,1)] + 2\cdot\Phi^1[10,(6,2)] + 2\cdot\Phi^1[10,(7,1)] + \Phi^1[10,(8)]$.
    \item $X = S^1\bigvee S^2$, multidegree $(3,5)$: computing rank of $\Phi^1[10,(5,1^3)] + \Phi^1[10,(6,1^2)]$. 
\end{itemize}
Note that in some of these cases we mix even and odd dimensional spheres. The partition $(5,1^3)$ has $4$ parts, and its transpose has $5$ parts. So it would take $\geq 4$ even spheres to access its multiplicity. But mixing even and odd spheres makes $\Schur{(5,1^3)}$ contribute nontrivially already with only two spheres.

In all these cases the rank obtained from the CE-homology calculation agrees with the rank obtained from Table \ref{table:final_answer10}. Therefore all Schur functors $\Schur{\lambda}$ for which $\Phi[10,\TransposePartition{\lambda}]$ appeared in one of the cases above cannot occur with higher multiplicity in any $\SymGroup[10]$-isotypical component. 

Careful bookkeeping shows that at this point every multiplicity of every Schur functor has been uniquely determined, therefore Table \ref{table:final_answer10} contains the full decomposition of every isotypical component.

\begin{table}[h]
    \vspace{-1.5cm}
    \hspace{-1cm}
    \begin{tabular}{|m{1.7cm}|m{15cm}|}
\hline
$\SymGroup[10]$-irrep & $\GL$-equivariant multiplicities in $\gr H^9_c(F(X,10))$\\
\hline
$(10)$ & $\Schur{(4,1)} + \Schur{(3,1^3)} + \Schur{(2,1^5)} + \Schur{(1^8)}$ \\
\hline
$(9,1)$ & $\Schur{(5)}$ \\
\hline
$(8,2)$ & $\Schur{(4)} + \Schur{(3,1^2)} + \Schur{(3)} + \Schur{(2,1^4)} + \Schur{(2,1)} + \Schur{(2)} + \Schur{(1^7)} + \Schur{(1^4)} + \Schur{(1^3)}$ \\
\hline
$(8,1^2)$ & $\Schur{(4)} + \Schur{(3,1)} + \Schur{(3)} + \Schur{(2,1^3)} + \Schur{(2,1^2)} + \Schur{(2)} + \Schur{(1^6)} + \Schur{(1^5)} + \Schur{(1^2)} + \Schur{(1)}$ \\
\hline
$(7,3)$ & $\Schur{(5)} + 4\Schur{(3)} + \Schur{(2,1^2)} + 2\Schur{(2,1)} + 2\Schur{(2)} + \Schur{(1^5)} + \Schur{(1^4)} + \Schur{(1^3)} + 2\Schur{(1^2)} + \Schur{(1)}$ \\
\hline
$(7,2,1)$ & $\Schur{(5)} + 3\Schur{(4)} + \Schur{(3,1)} + 5\Schur{(3)} + \Schur{(2,1^3)} + \Schur{(2,1^2)} + 4\Schur{(2,1)} + 5\Schur{(2)} + \Schur{(1^6)} + \Schur{(1^5)} + 2\Schur{(1^4)} + 3\Schur{(1^3)} + 3\Schur{(1^2)} + 2\Schur{(1)}$ \\
\hline
$(7,1^3)$ & $\Schur{(6)} + 2\Schur{(4)} + 2\Schur{(3)} + \Schur{(2,1^2)} + \Schur{(2,1)} + 3\Schur{(2)} + \Schur{(1^5)} + \Schur{(1^4)} + \Schur{(1^3)} + 2\Schur{(1^2)}$ \\
\hline
$(6,4)$ & $2\Schur{(4)} + 2\Schur{(3)} + 2\Schur{(2,1)} + 4\Schur{(2)} + \Schur{(1^4)} + 3\Schur{(1^3)} + 2\Schur{(1^2)}$ \\
\hline
$(6,3,1)$ & $\Schur{(5)} + 6\Schur{(4)} + 2\Schur{(3,1)} + 11\Schur{(3)} + \Schur{(2,1^3)} + 2\Schur{(2,1^2)} + 7\Schur{(2,1)} + 10\Schur{(2)} + \Schur{(1^6)} + 2\Schur{(1^5)} + 3\Schur{(1^4)} + 5\Schur{(1^3)} + 7\Schur{(1^2)} + 4\Schur{(1)}$ \\
\hline
$(6,2^2)$ & $\Schur{(5)} + 4\Schur{(4)} + 7\Schur{(3)} + 6\Schur{(2,1)} + 8\Schur{(2)} + 3\Schur{(1^4)} + 6\Schur{(1^3)} + 4\Schur{(1^2)} + \Schur{(1)}$ \\
\hline
$(6,2,1^2)$ & $2\Schur{(5)} + 6\Schur{(4)} + \Schur{(3,1)} + 13\Schur{(3)} + 2\Schur{(2,1^2)} + 8\Schur{(2,1)} + 10\Schur{(2)} + 2\Schur{(1^5)} + 3\Schur{(1^4)} + 5\Schur{(1^3)} + 7\Schur{(1^2)} + 5\Schur{(1)}$ \\
\hline
$(6,1^4)$ & $\Schur{(5)} + 2\Schur{(4)} + 4\Schur{(3)} + 3\Schur{(2,1)} + 4\Schur{(2)} + 2\Schur{(1^4)} + 3\Schur{(1^3)} + 2\Schur{(1^2)} + \Schur{(1)}$ \\
\hline
$(5^2)$ & $\Schur{(5)} + 3\Schur{(3)} + 3\Schur{(2,1)} + \Schur{(1^4)} + \Schur{(1^3)} + 2\Schur{(1)}$ \\
\hline
$(5,4,1)$ & $\Schur{(5)} + 5\Schur{(4)} + \Schur{(3,1)} + 10\Schur{(3)} + \Schur{(2,1^2)} + 7\Schur{(2,1)} + 9\Schur{(2)} + \Schur{(1^5)} + 3\Schur{(1^4)} + 5\Schur{(1^3)} + 6\Schur{(1^2)} + 3\Schur{(1)}$ \\
\hline
$(5,3,2)$ & $2\Schur{(5)} + 7\Schur{(4)} + \Schur{(3,1)} + 18\Schur{(3)} + \Schur{(2,1^2)} + 12\Schur{(2,1)} + 13\Schur{(2)} + \Schur{(1^5)} + 4\Schur{(1^4)} + 7\Schur{(1^3)} + 9\Schur{(1^2)} + 6\Schur{(1)}$ \\
\hline
$(5,3,1^2)$ & $\Schur{(6)} + \Schur{(5)} + 13\Schur{(4)} + 2\Schur{(3,1)} + 16\Schur{(3)} + \Schur{(2,1^2)} + 10\Schur{(2,1)} + 20\Schur{(2)} + \Schur{(1^5)} + 4\Schur{(1^4)} + 10\Schur{(1^3)} + 12\Schur{(1^2)} + 5\Schur{(1)}$ \\
\hline
$(5,2^2,1)$ & $4\Schur{(5)} + 8\Schur{(4)} + \Schur{(3,1)} + 19\Schur{(3)} + \Schur{(2,1^2)} + 11\Schur{(2,1)} + 16\Schur{(2)} + \Schur{(1^5)} + 3\Schur{(1^4)} + 8\Schur{(1^3)} + 11\Schur{(1^2)} + 6\Schur{(1)}$ \\
\hline
$(5,2,1^3)$ & $\Schur{(6)} + 3\Schur{(5)} + 8\Schur{(4)} + 14\Schur{(3)} + 9\Schur{(2,1)} + 13\Schur{(2)} + 3\Schur{(1^4)} + 6\Schur{(1^3)} + 8\Schur{(1^2)} + 5\Schur{(1)}$ \\
\hline
$(5,1^5)$ & $\Schur{(7)} + 2\Schur{(5)} + \Schur{(4)} + 5\Schur{(3)} + 2\Schur{(2,1)} + 3\Schur{(2)} + \Schur{(1^3)} + 2\Schur{(1^2)} + 2\Schur{(1)}$ \\
\hline
$(4^2,2)$ & $\Schur{(5)} + 6\Schur{(4)} + \Schur{(3,1)} + 6\Schur{(3)} + \Schur{(2^2)} + 4\Schur{(2,1)} + 10\Schur{(2)} + 2\Schur{(1^4)} + 6\Schur{(1^3)} + 5\Schur{(1^2)} + \Schur{(1)}$ \\
\hline
$(4^2,1^2)$ & $3\Schur{(5)} + 3\Schur{(4)} + 14\Schur{(3)} + \Schur{(2,1^2)} + 8\Schur{(2,1)} + 7\Schur{(2)} + \Schur{(1^5)} + 2\Schur{(1^4)} + 3\Schur{(1^3)} + 6\Schur{(1^2)} + 5\Schur{(1)}$ \\
\hline
$(4,3^2)$ & $\Schur{(5)} + 4\Schur{(4)} + \Schur{(3,1)} + 8\Schur{(3)} + \Schur{(2,1^2)} + 4\Schur{(2,1)} + 6\Schur{(2)} + \Schur{(1^5)} + \Schur{(1^4)} + 2\Schur{(1^3)} + 5\Schur{(1^2)} + 3\Schur{(1)}$ \\
\hline
$(4,3,2,1)$ & $4\Schur{(5)} + 15\Schur{(4)} + 2\Schur{(3,1)} + 25\Schur{(3)} + 16\Schur{(2,1)} + 23\Schur{(2)} + 4\Schur{(1^4)} + 11\Schur{(1^3)} + 14\Schur{(1^2)} + 9\Schur{(1)}$ \\
\hline
$(4,3,1^3)$ & $\Schur{(6)} + 4\Schur{(5)} + 9\Schur{(4)} + 17\Schur{(3)} + 9\Schur{(2,1)} + 16\Schur{(2)} + 2\Schur{(1^4)} + 7\Schur{(1^3)} + 10\Schur{(1^2)} + 5\Schur{(1)}$ \\
\hline
$(4,2^3)$ & $\Schur{(5)} + 7\Schur{(4)} + \Schur{(3,1)} + 8\Schur{(3)} + 4\Schur{(2,1)} + 11\Schur{(2)} + \Schur{(1^4)} + 5\Schur{(1^3)} + 6\Schur{(1^2)} + \Schur{(1)}$ \\
\hline
$(4,2^2,1^2)$ & $\Schur{(6)} + 4\Schur{(5)} + 11\Schur{(4)} + \Schur{(3,1)} + 19\Schur{(3)} + 9\Schur{(2,1)} + 16\Schur{(2)} + \Schur{(1^4)} + 5\Schur{(1^3)} + 11\Schur{(1^2)} + 8\Schur{(1)}$ \\
\hline
$(4,2,1^4)$ & $2\Schur{(6)} + 3\Schur{(5)} + 7\Schur{(4)} + 8\Schur{(3)} + 3\Schur{(2,1)} + 11\Schur{(2)} + 4\Schur{(1^3)} + 6\Schur{(1^2)} + 3\Schur{(1)}$ \\
\hline
$(4,1^6)$ & $\Schur{(6)} + \Schur{(5)} + \Schur{(4)} + 2\Schur{(3)} + 2\Schur{(2)} + 2\Schur{(1^2)} + \Schur{(1)}$ \\
\hline
$(3^3,1)$ & $\Schur{(6)} + 5\Schur{(4)} + 5\Schur{(3)} + 2\Schur{(2,1)} + 8\Schur{(2)} + 3\Schur{(1^3)} + 5\Schur{(1^2)} + \Schur{(1)}$ \\
\hline
$(3^2,2^2)$ & $3\Schur{(5)} + 3\Schur{(4)} + 12\Schur{(3)} + 7\Schur{(2,1)} + 5\Schur{(2)} + \Schur{(1^4)} + 2\Schur{(1^3)} + 4\Schur{(1^2)} + 5\Schur{(1)}$ \\
\hline
$(3^2,2,1^2)$ & $\Schur{(6)} + 3\Schur{(5)} + 9\Schur{(4)} + \Schur{(3,1)} + 13\Schur{(3)} + 6\Schur{(2,1)} + 14\Schur{(2)} + \Schur{(1^4)} + 5\Schur{(1^3)} + 8\Schur{(1^2)} + 4\Schur{(1)}$ \\
\hline
$(3^2,1^4)$ & $\Schur{(7)} + 4\Schur{(5)} + 2\Schur{(4)} + 9\Schur{(3)} + 4\Schur{(2,1)} + 4\Schur{(2)} + \Schur{(1^3)} + 3\Schur{(1^2)} + 4\Schur{(1)}$ \\
\hline
$(3,2^3,1)$ & $3\Schur{(5)} + 6\Schur{(4)} + \Schur{(3,1)} + 8\Schur{(3)} + 3\Schur{(2,1)} + 8\Schur{(2)} + 2\Schur{(1^3)} + 5\Schur{(1^2)} + 3\Schur{(1)}$ \\
\hline
$(3,2^2,1^3)$ & $2\Schur{(6)} + 2\Schur{(5)} + 7\Schur{(4)} + 8\Schur{(3)} + 3\Schur{(2,1)} + 9\Schur{(2)} + 2\Schur{(1^3)} + 5\Schur{(1^2)} + 3\Schur{(1)}$ \\
\hline
$(3,2,1^5)$ & $\Schur{(7)} + \Schur{(6)} + 2\Schur{(5)} + 3\Schur{(4)} + 3\Schur{(3)} + 4\Schur{(2)} + 2\Schur{(1^2)} + 2\Schur{(1)}$ \\
\hline
$(3,1^7)$ & $\Schur{(8)} + \Schur{(6)} + \Schur{(4)} + \Schur{(2)}$ \\
\hline
$(2^5)$ & $2\Schur{(4)} + 2\Schur{(2)} + \Schur{(1^3)}$ \\
\hline
$(2^4,1^2)$ & $2\Schur{(5)} + \Schur{(4)} + 4\Schur{(3)} + 2\Schur{(2,1)} + \Schur{(2)} + \Schur{(1^2)} + 2\Schur{(1)}$ \\
\hline
$(2^3,1^4)$ & $\Schur{(6)} + \Schur{(5)} + \Schur{(4)} + 2\Schur{(3)} + 2\Schur{(2)} + \Schur{(1^2)}$ \\
\hline
$(2^2,1^6)$ & $\Schur{(7)} + \Schur{(5)} + \Schur{(3)} + \Schur{(1)}$ \\
\hline
$(2,1^8)$ & 0 \\
\hline
$(1^{10})$ & $\Schur{(9)}$ \\
\hline
    \end{tabular}
    \caption{Equivariant multiplicity of each isotypic component in $\gr H_c^{9}(F(X,10))$}
    \label{table:final_answer10}
\end{table}

\section{Conjectural generating functions}\label{app:generating-functions}
We present here the conjectural multiplicities from Conjectures \ref{conj:patterns} and \ref{conj:patterns-M2n} in the form of generating functions.

Let $\Lambda$ denote the ring of symmetric functions (see e.g. \cite{MacDonald}). Some of the generating functions below are expressed as elements of $\Lambda[[t]]$, the $\lambda$-ring of power series in $t$ with coefficients in $\Lambda$. Let $\operatorname{ch}: K_0(\operatorname{Fun_{Poly}})\to \Lambda$ be the linear function on the Grothendieck group of polynomial functors, defined by sending the Schur functor $\Schur{\lambda}$ to the Schur symmetric function $s_{\lambda}$, and let $\operatorname{Exp}\colon t\Lambda[[t]]\to \Lambda[[t]]$ be the plethystic exponential, defined e.g. in \cite{GK-modular-operad}.

Consider partitions with one long row $\lambda[n] = (n-|\lambda|,\lambda)$. The generating functions
\begin{align*}
    &\sum_{n\geq 1}{(-t)^n \cdot \operatorname{ch}_{\GL(g)}\left(\gr H_c^{n-1}(F(\bigvee_g S^1, n))\otimes_{\SymGroup}\Specht{\lambda[n]}\right)}\quad\text{and}\\
    &\sum_{n\geq 1}{t^n\cdot\langle\Specht{\lambda[n]},\ \gr_0^W H_c^{n+2}(\ModuliCurve{2,n},\Q)\rangle_{\SymGroup}}
\end{align*}
are, respectively,
\begin{itemize}
    \item for $\lambda[n] = (n)$ they are $\frac{t^2(\operatorname{Exp}(s_1(t^2-t))-1)}{1-t} - t$ and $\frac{(t^4-t^7 + t^{10})(1+t^3)}{(1-t^4)(1-t^{12})}$  (see \ref{prop:known-mult-PV}),
    \item for $\lambda[n] = (n-1,1)$ they are $\operatorname{Exp}(s_1 t^2)-1$ and $\frac{t^{12}}{(1-t^4)(1-t^{12})}$ (see \ref{prop:Powell-baby-bead}),
    \item for $\lambda[n] = (n-2,2)$ they are $\frac{1}{1-t^4}\Big((t-t^4)(1+s_1t^2-\operatorname{Exp}(s_1t^2))-(t^3+t^4+t^5)\cdot({\operatorname{Exp}(s_1(t^2-t))}-1)\Big)$ and $\frac{(t^5 + t^{13} + t^{14})(1+t^3)}{(1-t^4)^2(1-t^{12})}$,
    \item for $\lambda[n] = (n-2,1,1)$ they are $\frac{1}{1-t^4}\left((1-t^3)\frac{1+s_1t^2-\operatorname{Exp}(s_1t^2)}{t}+t^4(\operatorname{Exp}(s_1(t^2-t))-1)\right)$ and $\frac{(t^6-t^8+t^{11})(1-t^6)}{(1-t)(1-t^4)^2(1-t^{12})}$.
\end{itemize}

For partitions with one long column $\TransposePartition{\lambda[n]}$ the generating functions
\begin{align*}
    &\sum_{n\geq 1}{t^n \cdot \operatorname{ch}_{\GL(g)}\left(\gr H_c^{n-1}(F(\bigvee_g S^1, n))\otimes_{\SymGroup}\Specht{\TransposePartition{\lambda[n]}}\right)}\quad\text{and}\\
    &\sum_{n\geq 1}{t^n\cdot\langle\Specht{\TransposePartition{\lambda[n]}},\ \gr_0^W H_c^{n+2}(\ModuliCurve{2,n},\Q)\rangle_{\SymGroup}}
\end{align*}
are, respectively,
\begin{itemize}
        \item for $\TransposePartition{\lambda[n]} = (1^n)$ they are $t\operatorname{Exp}(s_1t)$ and $\frac{t^7}{(1-t^2)(1-t^6)}$  (see \ref{prop:known-mult-PV})
        \item for $\TransposePartition{\lambda[n]} = (2,1^{n-1})$ they are $0$ for both generating functions (see \ref{prop:known-mult-PV}),
        \item for $\TransposePartition{\lambda[n]} = (2,2,1^{n-4})$ they are $\frac{t^3(\operatorname{Exp}(s_1t)-1)}{1-t^2}$ and $\frac{t^9}{(1-t^2)^2(1-t^6)}$
        \item for $\TransposePartition{\lambda[n]} = (3,1^{n-3})$ they are $\frac{t^2(\operatorname{Exp}(s_1t)-1)}{1-t^2}$ and $\frac{t^8}{(1-t^2)^2(1-t^6)}$
        \item for $\TransposePartition{\lambda[n]} = (2,2,2,1^{n-6})$ they are $\frac{t^4(\operatorname{Exp}(s_1t)-1-s_1t)}{(1-t)(1-t^3)} + \frac{t^7(s_1t^2+s_{1,1})}{(1-t^2)(1-t^3)}$ and $\frac{t^7-t^8+t^{10}-t^{13}+t^{14}}{(1-t)(1-t^2)(1-t^3)(1-t^6)}$
        \item for $\TransposePartition{\lambda[n]} = (3,2,1^{n-5})$ they are $\frac{(t^3+t^5)(\operatorname{Exp}(s_1t)-1-s_1t) + t^5(s_1t^2 + s_{1,1})}{(1-t)(1-t^3)}$ and $\frac{t^5-t^7+t^9+t^{13}}{(1-t)(1-t^2)(1-t^3)(1-t^6)}$
        \item for $\TransposePartition{\lambda[n]} = (4,1^{n-4})$ they are $\frac{t^4(\operatorname{Exp}(s_1t)-1-s_1t)}{(1-t)(1-t^3)} + \frac{t^4(s_1t^2+s_{1,1})}{(1-t^2)(1-t^3)}$ and $\frac{t^4-t^5+t^{11}}{(1-t)(1-t^2)(1-t^3)(1-t^6)}$.
    \end{itemize}

\begin{rmk}
    The agreement of these expressions with the conjectural characters in Conjecture \ref{conj:patterns} has been verified computationally up to a high degree of confidence. However, we have not proved this formally. 
\end{rmk}

\printbibliography

\end{document}